\documentclass{article}
\pdfoutput=1

\usepackage[text={6.75in,9.5in},centering,letterpaper]{geometry}
\usepackage{amssymb,amsmath,amsthm}
\usepackage{graphicx}
\usepackage[margin = 10pt, font=small, labelfont=bf, labelsep = endash]{caption}
\usepackage{subcaption}
\usepackage{xcolor}

\newtheorem{remark}{\color{blue!30!black}Remark}

\newtheorem{proposition}{\color{blue!30!black}Proposition}
\newtheorem{lemma}{\color{blue!30!black}Lemma}
\newtheorem{theorem}{\color{blue!30!black}Theorem}
\newtheorem{corollary}{\color{blue!30!black}Corollary}

\newtheorem{definition}{\color{blue!30!black}Definition}

\DeclareMathOperator{\sech}{sech}

\title{Pulse solutions for an extended Klausmeier model with spatially varying coefficients\thanks{Last edited: \today.}}

\author{Robbin Bastiaansen\thanks{Mathematical Institute, Leiden University, 2300 RA Leiden, The Netherlands
  (r.bastiaansen@math.leidenuniv.nl, m.chirilus-bruckner@math.leidenuniv.nl, doelman@math.leidenuniv.nl).}
\and Martina Chirilus-Bruckner\footnotemark[2]
\and Arjen Doelman\footnotemark[2]}

\begin{document}

\maketitle

\begin{abstract}
Motivated by its application in ecology, we consider an extended Klausmeier model, a singularly perturbed reaction-advection-diffusion equation with spatially varying coefficients. We rigorously establish existence of stationary pulse solutions by blending techniques from geometric singular perturbation theory with bounds derived from the theory of exponential dichotomies. Moreover, the spectral stability of these solutions is determined, using similar methods. It is found that, due to the break-down of translation invariance, the presence of spatially varying terms can stabilize or destabilize a pulse solution. In particular, this leads to the discovery of a pitchfork bifurcation and existence of stationary multi-pulse solutions.
\end{abstract}

\section{Introduction}
Since Alan Turing's revolutionary insight that patterns can emerge spontaneously in systems with multiple species if these diffuse at different rates~\cite{turing1952}, systems of reaction-diffusion equations have served as prototypical pattern forming models. Scientists have been using these reaction-diffusion models successfully to describe for instance animal markings~\cite{koch1994}, embryo development~\cite{meinhardt2008} and the faceted eye of {\it Drosophila}~\cite{maini2001}. Special interest has been given to localized solutions (e.g. pulses, fronts), that arise when the diffusivity of species involved is very different. The prototypical (two-component) model (in one spatial dimensional) is a singularly perturbed equation of the (scaled) form
\begin{equation}\label{eq:modelGeneric}
\left\{
\begin{array}{rcrcl}
	 \partial_t U & = & \partial_{x}^2 U & +&  \mathcal{H}_1\left(x,u,u_x,v,v_x;\tilde{\varepsilon}\right) ,\\
 	 \partial_t V & = & \tilde{\varepsilon}^2 \partial_{x}^2 V & +& \mathcal{H}_2\left(x,u,u_x,v,v_x;\tilde{\varepsilon}\right),
\end{array}
\right.
\end{equation}
where $0 < \tilde{\varepsilon} \ll 1$ is a measure for the ratio of diffusion constants, and $\mathcal{H}_1$, $\mathcal{H}_2$ are sufficiently smooth functions. Because of the singular perturbed nature of~\eqref{eq:modelGeneric}, it is possible to establish existence and determine (linear) stability of localized patterns in these models. In the past, this has been done successfully for the Gray-Scott model~\cite{DEK01, doelman1998, doelman2003semistrong, chen2009oscillatory,Kolokolnikov2005PS, Sun2005}, the Gierer-Meinhardt model~\cite{D01, veerman2013pulses, doelman2003semistrong, Sun2005}, and in several other settings~\cite{BjornRiccati, doelman2015explicit, rottschafer2017transition, moyles2016explicitly}. However, these studies are usually limited to models with constant coefficients. Some research has focused on the introduction of localized spatial inhomogeneities~\cite{van2010pinned, nishiura2007dynamics, nishiura2007dynamics2, xin2000front,yuan2007, doelman2016geometric}; also (often formal) research has been done on reaction-diffusion equations with (less restricted) spatially varying coefficients~\cite{brena2015, brena2014, avitabile2018, wei2017, wei2017-2, berestycki2014}. In this article, we aim to expand the knowledge of such systems, by studying a reaction-diffusion system with fairly generic spatially varying coefficients rigorously; motivated by its use in ecology (see Remark~\ref{remark:applicationKlausmeier}), we consider the following extended Klausmeier model with spatially varying coefficients~\cite{klausmeier1999,BD18}:
\begin{align}\label{eq:klausmeier_model}
\left\{
\begin{array}{rcrl}
 \partial_t U & = & \partial_{x}^2 U & +  f(x) \partial_{x}U + g(x)U + a - U - U V^2 \, ,\\[.2cm]
 \partial_t V & = & D^2 \partial_{x}^2 V & - \ m V + U V^2 \, ,
 \end{array}
 \right.
\end{align}
with $ x \in \mathbb{R}, t \geq 0, U = U(x,t), V = V(x,t) \in \mathbb{R} $, parameters $ D, a, m > 0 $ and functions $ f, g \in C^1_b(\mathbb{R})$. Certain conditions are imposed on the parameters and functions $f$ and $g$ -- these will be explained in section~\ref{sec:assumptions}.

\begin{remark}
	The model~\eqref{eq:klausmeier_model} can be brought into the form of~\eqref{eq:modelGeneric} by a series of scalings -- see section~\ref{sec:existence} and~\cite{doelman2003semistrong}.
\end{remark}

\begin{remark}[\underline{Application of the extended Klausmeier model}]\label{remark:applicationKlausmeier}
This system of equations is used as a model in ecology to describe the dynamics of vegetation ($ U $) and water ($ V $). The extended Klausmeier model \eqref{eq:klausmeier_model} takes into account the amount of rainfall ($ a > 0 $) and mortality rate of the vegetation ($m>0$) and goes beyond its classical version by modeling a smooth, spatially varying terrain $ h = h(x) $ which then enters \eqref{eq:klausmeier_model} as $ f(x) = h'(x), g(x) = h''(x) $ (see~\cite{BD18}). Variants of the Klausmeier model have been studied in various articles ranging from ecological studies~\cite{klausmeier1999, Bastiaansens2018} to mathematical analysis~\cite{BD18,siteur2014beyond,sherratt2013,sherratt2015}. The focus of all these studies are vegetation patterns, which have been found to play a crucial role in the process of desertification. A starting point for the analysis of more complicated patterns is a thorough understanding of their building blocks, namely, localized 
solutions. The present paper is motivated by observations -- both in numerical simulations and in real ecosystems~\cite{BD18,Bastiaansens2018} -- of the impact of nontrivial topographies on the dynamics of localized vegetation patterns.

\end{remark}

The focus of this article is to analyze existence, stability and (some) bifurcations of stationary pulse solutions to~\eqref{eq:klausmeier_model}. The presence of spatially varying coefficients, however, alters the approach that usually is taken in the case of constant coefficients models. For one, with spatially constant coefficients, \eqref{eq:klausmeier_model} possesses a uniform stationary state, with $V \equiv 0$, to which pulse solutions converge for $x \rightarrow \pm \infty$. In the case of spatially varying coefficients, however, typically such uniform stationary state does not exist; instead, a bounded solution $(u,v) = (u_b,0)$ exists and pulse solutions converge to this bounded solution for $x \rightarrow \pm \infty$ -- see Figure~\ref{fig:pulses}. Moreover, standard proofs using geometric singular perturbation theory typically rely on the availability of closed form expressions for orbits of subsystems of~\eqref{eq:klausmeier_model} -- see below. These are no longer available in case of generic spatially varying coefficients, and only bounds can be found. Indeed, the core contribution of the present work is to overcome these difficulties, which we do by blending geometric singular perturbation theory~\cite{Fen79} with the theory of exponential dichotomies~\cite{coppel1978stability} in a new way.\\

\begin{figure}[t!]
 \centering
	\begin{subfigure}[t]{0.32\textwidth}
		\includegraphics[width=\textwidth]{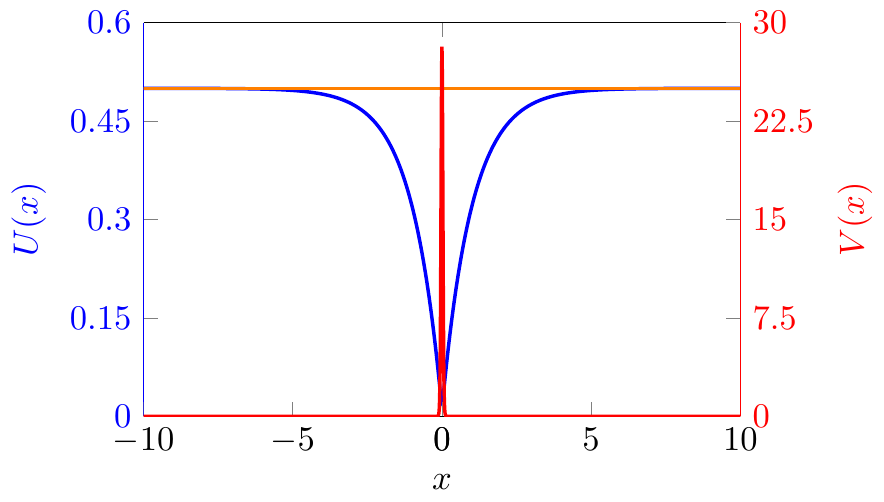}
	\caption{$h(x) = 0$}
	\label{fig:pulsesa}
	\end{subfigure}
~
	\begin{subfigure}[t]{0.32\textwidth}
		\includegraphics[width=\textwidth]{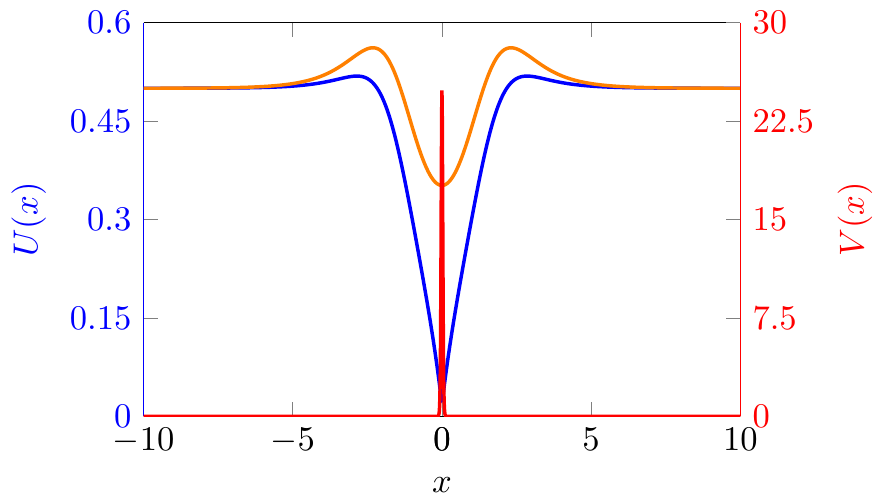}
	\caption{$h(x) = \exp(-x^2/2)$}
	\label{fig:pulsesb}
	\end{subfigure}
~
	\begin{subfigure}[t]{0.32\textwidth}
		\includegraphics[width=\textwidth]{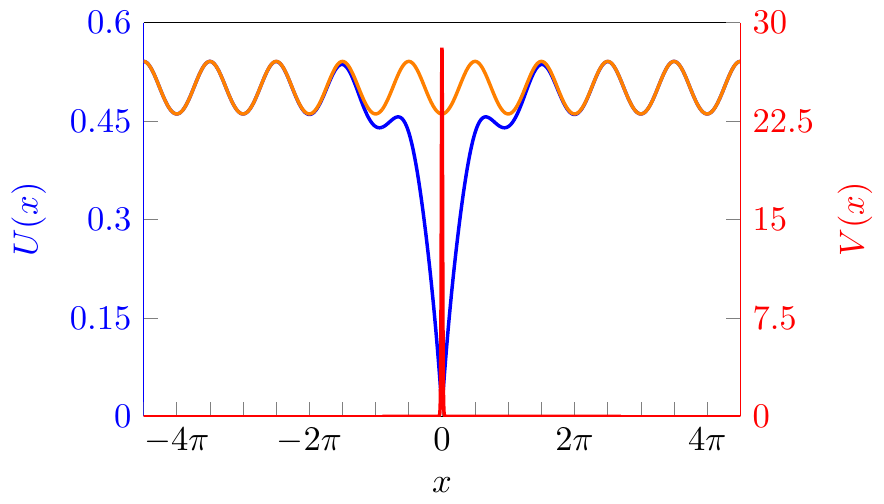}
	\caption{$h(x) = 0.1 \cos(2x) $}
	\label{fig:pulsesc}
	\end{subfigure}
 \caption{\small Numerical simulation resulting in a stationary pulse solution for \eqref{eq:klausmeier_model} with $f(x) = h'(x)$, $g(x) = h''(x)$, where $h(x) = 0$ (a), $h(x) = \exp(-x^2/2)$ (b) and $h(x) = 0.1 \cos(2x)$ (c). $ U, V $ components are blue and red respectively, while the orange curve depicts the bounded solution $ u_b $ to which the $ U $-component converges for $|x| \rightarrow \infty$.}\label{fig:pulses}
\end{figure}

In this article, we initially follow the `standard' approach of geometric singular perturbation theory. That is, we introduce a small parameter $\varepsilon := \frac{a}{m}$ -- see assumption (A1) in section~\ref{sec:assumptions}-- and construct a stationary pulse solution to~\eqref{eq:klausmeier_model} in the limit $\varepsilon = 0$, which present itself as a homoclinic orbit in the related stationary fast-slow ODE system -- in case of spatially varying coefficients it is homoclinic to the bounded solution. For this construction, the full system is split into a fast subsystem, and a (super)slow subsystem on a so-called slow manifold $\mathcal{M}$ that consists of fixed points of the fast subsystem. We establish fast connections to and from $\mathcal{M}$ that take off from submanifold $T_o \subset \mathcal{M}$ and touch down on submanifold $T_d \subset \mathcal{M}$. On $\mathcal{M}$, we construct stable and unstable submanifolds $W^{s/u}(u_b) \subset \mathcal{M}$ that consists of points on $\mathcal{M}$ that converge to the bounded solution for $x \rightarrow \infty$ respectively $x \rightarrow -\infty$. Intersections between these unstable/stable manifolds and take-off/touch-down submanifolds (and a symmetry assumption) then establish the existence of pulse solutions to~\eqref{eq:klausmeier_model}. Finally, persistence of these pulse solutions for $\varepsilon > 0$ is guaranteed by geometric singular perturbation theory~\cite{Fen79}.

Specifically, stationary solutions $ (U(x,t),V(x,t))= (\tilde{u}(x), \tilde{v}(x)) $ of \eqref{eq:klausmeier_model} fulfill the system of ODEs
\begin{align}\label{eq:klausmeier_model_ODE}
\left\{
\begin{array}{rcrl}
  0 & = & \tilde{u}_{xx} & +f(x)\tilde{u}_x + g(x) \tilde{u} + a - \tilde{u} - \tilde{u} \tilde{v}^2 \, ,\\[.1cm]
  0 & = & \frac{D^2}{m} \tilde{v}_{xx} & - \tilde{v} + \frac{1}{m} \tilde{u} \tilde{v}^2 \, .
 \end{array}
 \right.
\end{align}
After a sequence of (re)scalings, it can be seen that the associated fast subsystem is not affected by the spatially varying terms and can be studied using standard methods. However, the slow subsystem, on the slow manifold $\mathcal{M}$, is affected by the spatially varying terms. This subsystem is given (when rescaling $\hat{u} = a \tilde{u}$) by
\begin{equation}\label{eq:slowSubsystem}
	\left\{
	\begin{array}{rcl}
		\partial_x \hat{u} & = & \hat{p}, \\
		\partial_x \hat{p} & = & - f(x) \hat{p} - g(x) \hat{u} - 1 + \hat{u}.
	\end{array}
	\right.
\end{equation}
For $f$ and $g$ constant,~\eqref{eq:slowSubsystem} can be solved explicitly and the stable and unstable manifolds $W^{s,u}(u_b)$ are known explicitly. In case of (spatially) varying $f$ and $g$, typically no closed form solutions are available; however, when these varying coefficients are sufficiently small -- specifically, when $\delta := \sup_{x \in \mathbb{R}} \sqrt{f(x)^2+g(x)^2} < \frac{1}{4}$ (so $\delta$ can be $\mathcal{O}(1)$ with respect to $\varepsilon$); see section~\ref{sec:exp_dich} -- the dynamics of~\eqref{eq:slowSubsystem} can be related to the constant coefficient case $f,g \equiv 0$ using the theory of exponential dichotomies.

\begin{figure}
	\centering
	\includegraphics[width = 0.6\textwidth]{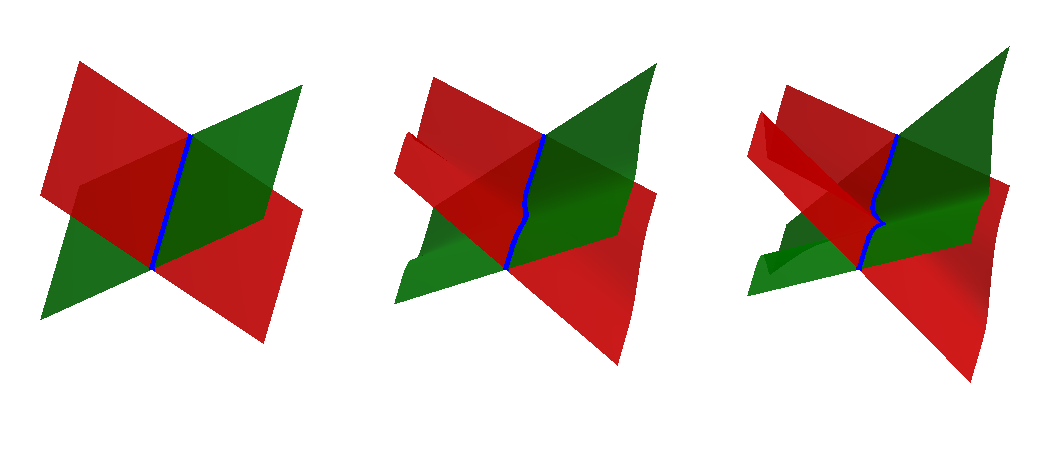}
	\caption{Sketches of the bounded solution (blue) and its stable (green) respectively unstable (red) manifolds in case of constant coefficients (left) and varying coefficients (center and right).}
	\label{fig:manifoldsOnSlowManifold}
\end{figure}

In particular, the saddle structure -- present for $f,g\equiv0$ -- persists as exponential dichotomy. Therefore,~\eqref{eq:slowSubsystem} possesses a $1D$ family of solutions that converge to the (unique) bounded solution to~\eqref{eq:slowSubsystem} for $x \rightarrow \infty$ and a $1D$ family of solutions that converge to the bounded solution for $x \rightarrow -\infty$. These families of solutions essentially form the stable and unstable manifolds $W^{s,u}(u_b)$. Due to the linear nature of~\eqref{eq:slowSubsystem}, these (un)stable manifolds are made up of straight lines, i.e. $W^{s,u}(u_b) = \cup_{x\in\mathbb{R}} (x,l^{s,u}(x))$ where $l^{s,u}(x)$ describes a straight line in $\mathbb{R}^2$. An important difference now arises between the cases of constant and varying coefficients: when $f,g\equiv 0$, the lines $l^{s,u}(x)$ do not depend on $x$; when $f$ and $g$ are spatially varying, they do. Hence, $W^{s,u}(u_b)$ appears wiggly in case of varying coefficients -- see~Figure~\ref{fig:manifoldsOnSlowManifold}. The theory of exponential dichotomies enables us to bound the variation of the lines $l^{s,u}(x)$; if $\delta$ is small enough (i.e. $\delta < \delta_c(a,m,D)$, where $\delta_c \leq 1/4$ is $\mathcal{O}(1)$ with respect to $\varepsilon$), these bounds are strict enough that a non-empty intersection $(0,l^u(0)) \cap T_o$ is guaranteed -- thus establishing existence of a (symmetric) pulse solution to~\eqref{eq:klausmeier_model}. See Figure~\ref{fig:existenceProofSketches} for a sketch.

\begin{figure}
	\centering
		\begin{subfigure}[t]{0.25\textwidth}
			\centering
			\includegraphics[width=\textwidth]{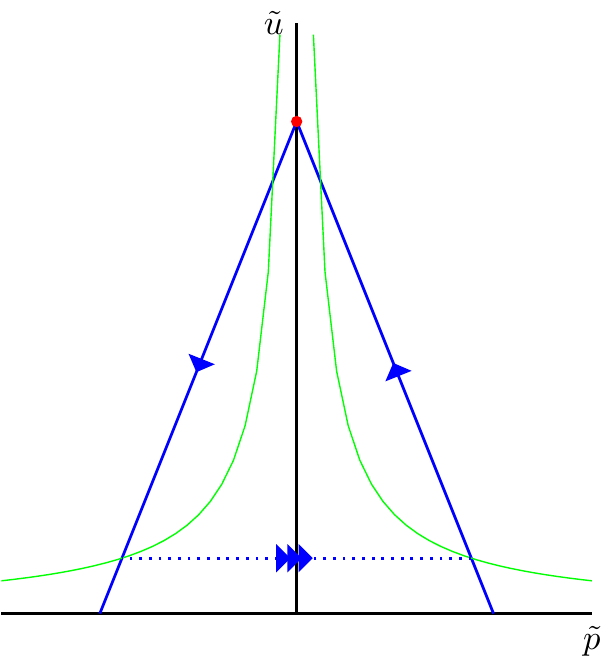}
			\caption{Constant coefficient case}
		\end{subfigure}
~
		\begin{subfigure}[t]{0.25\textwidth}
			\centering
			\includegraphics[width=\textwidth]{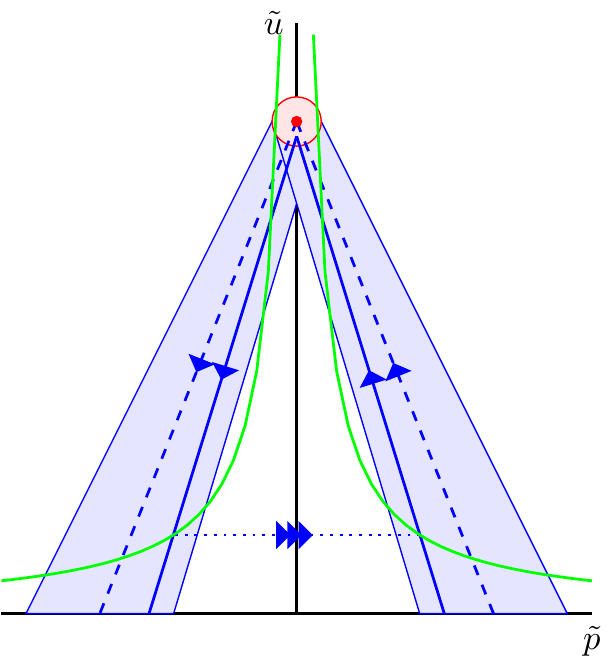}
			\caption{Strong enough bounds}
		\end{subfigure}
~
		\begin{subfigure}[t]{0.25\textwidth}
			\centering
			\includegraphics[width=\textwidth]{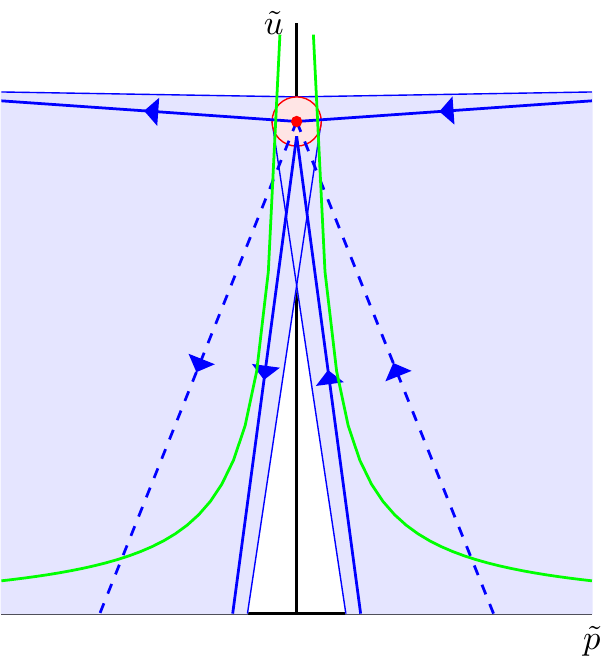}
			\caption{Too weak bounds}
		\end{subfigure}
	\caption{Sketches of a crosssection of $\mathcal{M}$ that illustrate the heart of the existence proof. In green the takeoff and touchdown curves are shown, the solid blue lines indicate (possible) $l^{s/u}(0)$, the dashed blue lines $l^{s/u}(0)$ for the constant coefficient case $f = 0, g = 0$. The shaded blue area indicates all possible locations of $l^{s/u}(0)$; the shaded red region the possible locations of the bounded solution. The existence proof works when bounds on $u_b$ and $l^{s/u}(0)$ are strong enough such that $l^{u}(0)$ necessarily intersects with $T_o(0)$ -- this happens when all straight lines that start from the red region and stay within the blue region intersect the green curves. If bounds are strong enough this is the case -- as illustrated in (b) -- but when bounds are too weak this is not the case and existence is not guaranteed by this method -- as illustrated in (c). In (a) the situation for the constant coefficient case is shown.}
	\label{fig:existenceProofSketches}
\end{figure}

Next, the spectral stability of the thus created pulse solutions is studied. Using similar bounds as in the existence problem, it is shown that eigenvalues are $\delta$-close to their counterparts in case of constant coefficients -- see Figure~\ref{fig:spectralBounds}. That is, under several conditions, typical for these systems, the `large' eigenvalues can be bounded to the stable half-plane $\{ \lambda \in \mathbb{C}: \mbox{Re} \lambda < 0 \}$. For the `small' eigenvalue -- located close to the origin -- it is more subtle. In case of $f,g \equiv 0$ this small eigenvalue is located precisely at the origin due to the translation invariance of~\eqref{eq:klausmeier_model}. The introduction of spatially varying coefficients to the system breaks this invariance and as a result the small eigenvalue moves to the stable or the unstable half-plane.

\begin{figure}
	\centering
	\includegraphics[width=0.355 \textwidth]{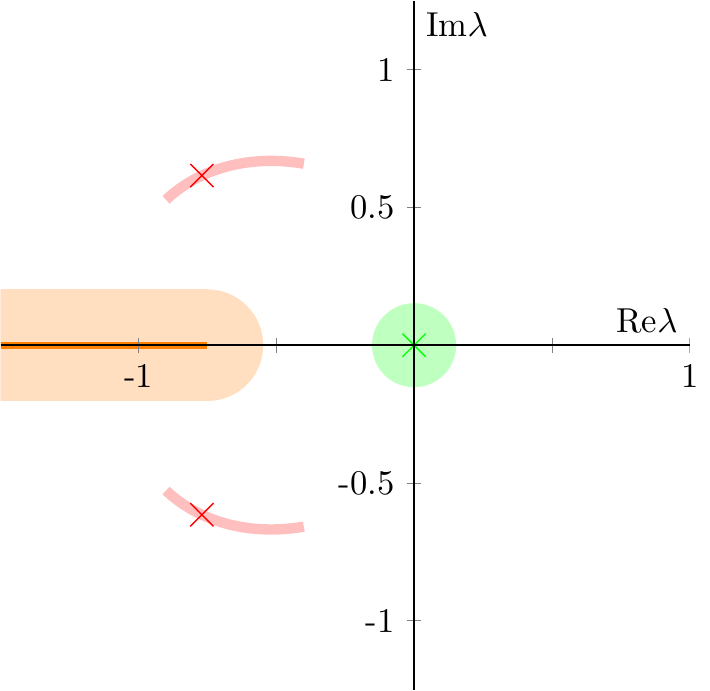}
	\caption{Sketch of the spectral bounds obtained in this paper. The shaded areas indicate the possible locations of spectra in the case of varying coefficients. The solid lines and crosses indicate the location of the essential and point spectra in the case of constant coefficients: the essential spectrum (orange), the `large' eigenvalues (red) and the `small' eigenvalue (green).}
	\label{fig:spectralBounds}
\end{figure}

Tracking of this eigenvalue indicates that it can, indeed, move to either half-plane, depending on the form of the functions $f$ and $g$. In particular, when taking $f = h'$, $g = h''$, the location of the small eigenvalue is related to the curvature $g=h''$ of $h$: when the curvature is weak, the pulse solution is stable if $g(0) = h''(0) < 0$ and unstable if $g(0) = h''(0) > 0$; for strong curvature, this is flipped, due to a pitchfork bifurcation.

Finally, the break-down of the translation invariance in~\eqref{eq:klausmeier_model} has another novel effect. In case of constant coefficients, stationary multi-pulse solutions -- solutions with multiple fast excursions -- do not exist, due to the presence of the translation invariance. If this invariance is broken, they can exist; the introduction of functions $f$ and $g$ now allows for these stationary multi-pulse solutions (under some conditions on $f$ and $g$) and their existence can be established (although we refrain from going in the details).

The set-up for the rest of this paper is as follows. In section~\ref{sec:existence}, we establish existence of stationary pulse solutions to~\eqref{eq:klausmeier_model}; here we first consider the case $f,g \equiv 0$ and subsequently the case of generic (bounded) $f$ and $g$. Then, using the theory of exponential dichotomies, both cases are related to each other, resulting in bounds for the generic case that allow us to prove existence. In section~\ref{sec:linstability} we study the spectral stability of found pulse solutions, again by relating the generic case to the constant coefficient case of $f,g\equiv0$. Then, in section~\ref{sec:pulseLocationODE} we consider the small eigenvalues more in-depth using formal and numerical techniques, focusing on the possible occurrence of bifurcations; we also present stationary multi-pulse solutions. We conclude with a discussion of the results in section~\ref{sec:discussion}.

\subsection{Assumption}\label{sec:assumptions}
We will make several assumptions throughout the manuscript. Some are crucial, while some serve to simplify the exposition.
\begin{align}
%
% ASSUMPTION 1
\mathbf{(A1):} \qquad & 
\varepsilon := \frac{a}{m} \ll 1 \label{eq:a_m_assumption}; \\
%
% ASSUMPTION 2
\mathbf{(A2):} \qquad & 
f(-x)= -f(x) \,, \quad g(-x)= g(x) \,, \qquad \mbox{for all $x \in \mathbb{R}$};\label{eq:f_g_assumptions_symmetry} \\
%
% ASSUMPTION 3
\mathbf{(A3):} \qquad & 
\sup_{x \in \mathbb{R}} \sqrt{ f(x)^2 + g(x)^2} < \frac{1}{4} \, ; \label{eq:f_g_assumptions_magnitude} \\
%
% ASSUMPTION 4
\mathbf{(A4):} \qquad & 
\lim_{x \rightarrow \pm \infty} f(x), g(x) = 0 \, ; \label{eq:f_g_asymptotics} \\
%
% ASSUMPTION 5
\mathbf{(A5):} \qquad &
 ||f||_{C_b} = \mathcal{O}(1) \, , \qquad  ||g||_{C_b} = \mathcal{O}(1) \qquad \left(w.r.t. \ \frac{a}{m}\right) \label{eq:f_g_wrt_epsilon} 
\end{align}
Assumption (A1) ensures the presence of a small parameter, necessary to use geometric singular perturbation theory~\cite{SD17,BD18}. (A2) is a symmetry assumption, that ensures~\eqref{eq:klausmeier_model} possesses a (point) symmetry in $x = 0$; this technicality significantly simplifies our rigorous proof; pulse solutions can also be found formally and/or numerically when (A2) does not hold (and we expect that their existence can be established rigorously by extending our methods). Then, assumption (A3) stems from the theory of exponential dichotomies: when this holds, solutions to~\eqref{eq:slowSubsystem} for generic $f$ and $g$ can be linked to solutions of~\eqref{eq:slowSubsystem} with $f,g\equiv 0$; when (A3) does not hold, this link is not provided by the theory of exponential dichotomies. Assumption (A4) is a technicality that is only needed in the stability section (specifically for the elephant-trunk method to work); for the existence theorems it is not necessary; in fact, it is suspected that even stability results continue to hold when (A4) is violated -- see also Remarks~\ref{remark:stability_fgLimits} and~\ref{remark:stability_fgBounded}. Finally, assumption (A5) is needed to pass limits in the treatment of the fast-slow system.

%%%%%%%%%%%%%%%%%%%%%%%%%%%%%%%%%%%%%%%%%%%%%%%%%%%%%
%% SECTION: Analysis of stationary pulse solutions %%
%%%%%%%%%%%%%%%%%%%%%%%%%%%%%%%%%%%%%%%%%%%%%%%%%%%%%

\section{Analysis of stationary pulse solutions}\label{sec:existence}

A crucial step for making the stationary ODE~\eqref{eq:klausmeier_model_ODE} amenable to analytic considerations is to find a parameter regime convenient for rigorous perturbation techniques. While there are various choices, we pick a specific one for clarity, since our focus is on novel phenomena due to the non-autonomous character of the system and not to classify all possible dynamics across parameter regimes.

Following~\cite{DEK01,chen2009oscillatory,BD18}, we rescale the spatial coordinate (motivated by the diffusivity of the $ v $-component) and the amplitudes of the unknowns by
\begin{align}\label{eq:scaling_1}
 \xi := \frac{\sqrt{m}}{D} x \, , \quad  \tilde{u} = \frac{m \sqrt{m} D}{a} u \, , \quad \tilde{v} = \frac{a}{\sqrt{m} D}  v \, ,
\end{align}
to get
\begin{align}\label{eq:klausmeier_model_constant_ODE_second_order}
\left\{
\begin{array}{rcl}
  u_{\xi \xi} & = & \frac{a^2}{m^2} \left[ \frac{D^2m}{a^2} u - \frac{Dm\sqrt{m}}{a^2} f\left( \frac{D}{\sqrt{m}} \xi\right) u_{\xi}- \frac{D^2m}{a^2} g\left( \frac{D}{\sqrt{m}} \xi\right) u  - \frac{D}{\sqrt{m}}  +  u v^2 \right]\, ,\\[.1cm]
  v_{\xi \xi} & = & v - u v^2 \, .
 \end{array}
 \right. 
\end{align}
It is now convenient to introduce
\begin{align}\label{eq:epsilon_mu}
  0 <  \varepsilon : = \frac{a}{m}  \, , \quad   0 <  \mu : = \frac{m \sqrt{m}D}{a^2} \, ,
\end{align}
and write the above ODEs as the first order system of ODEs
\begin{align}\label{eq:klausmeier_model_ODE_first_order}
\left\{
\begin{array}{rcl}
  \dot{u} & = & \varepsilon  p \, , \\[.1cm]
  \dot{p} & = & \varepsilon \left[ \varepsilon^2 \mu^2 u - \varepsilon \mu f\left( \varepsilon^2 \mu \xi\right) p - \varepsilon^2 \mu^2 g\left( \varepsilon^2 \mu \xi\right) u - \varepsilon^2 \mu  +  u v^2 \right] \, ,\\[.1cm]
  \dot{v} & = & q\, , \\[.1cm]
  \dot{q} & = & v - u v^2 \, .
 \end{array}
 \right. 
\end{align}
In order to use geometric singular perturbation theory, we make the customary assumption (A1), that is,
\begin{align}\label{eq:epsilon_small}
  0 <  \varepsilon \ll 1 \, .
\end{align}
and stipulate assumption (A5) so we can pass to limits.

In the autonomous case $ f \equiv 0 $ and $ g \equiv 0 $, system \eqref{eq:klausmeier_model_ODE_first_order} has a fixed point $ \left(1/\mu,0,0,0\right) $ and stationary pulse solutions of \eqref{eq:klausmeier_model} correspond to orbits that are homoclinic to $ \left(1/\mu,0,0,0\right) $; see Figure~\ref{fig:pulsesa} for an example. In the non-autonomous case $ f \neq 0, g \neq  0 $ there is no fixed point, but instead a unique bounded solution $ (u_b, p_b, 0, 0) $. In this case, stationary pulse solutions of \eqref{eq:klausmeier_model} correspond to orbits that are homoclinic to this bounded solutions; see Figures~\ref{fig:pulsesb} and \ref{fig:pulsesc} for examples. The existence of said unique bounded solution $(u_b,p_b,0,0)$ is established in the following proposition proven later in section~\ref{sec:exp_dich} (in the proof of Proposition~\ref{prop:roughness_closeness_general}).

%===========================================================================================================%
% Proposition (\underline{Existence of a bounded solution for \eqref{eq:klausmeier_model_ODE_first_order}}) %
%===========================================================================================================%

\begin{proposition}[\underline{Existence of a bounded solution for \eqref{eq:klausmeier_model_ODE_first_order}}]\label{proposition:bounded_solution}
Let assumptions (A3) and (A4) be fulfilled. Then \eqref{eq:klausmeier_model_ODE_first_order} has a unique bounded solution $ (u_b, p_b, 0, 0) $ that satisfies \begin{align}                                                                                                                                                          
 \lim_{\xi \leftarrow \pm \infty} (u_b, p_b, 0, 0) =   \left(1/\mu,0,0,0\right) \, .                                                                                                                                                       \end{align}
\end{proposition}

%=================================================%
% Remark (Orbits homoclinic to bounded solutions) %
%=================================================%

\begin{remark}[\underline{Orbits homoclinic to bounded solutions}]\label{rem:orbits_homoclinic_bounded_solutions}
Note that the assumption $ \lim_{x \rightarrow \pm \infty} f(x), g(x) = 0 $ in (A4) is not necessary for the existence proof, but will be used in the stability analysis. In case $ f, g $ are only bounded without approaching a constant state when $|x| \rightarrow \infty$, the corresponding constructed pulse solution is also a homoclinic to the respective bounded solution. An illustration of such a case is given in Figure~\ref{fig:pulsesc}, where, due to the periodicity of the coefficients $ f, g $, the bounded background solution is periodic and so is the pulse solution in its tails. 
\end{remark}

To highlight the novelty of the presented approach, we first briefly explain how the construction is carried out in the constant coefficient case $ f = g = 0 $, to then proceed to the non-autonomous case.

%%%%%%%%%%%%%%%%%%%%%%%%%%%%%%%%%%%%%%%%%%%%%%%%%%%%%%%%%%%%%%%%%%%%%%%%%%%
%% SUBSECTION: Stationary pulse solutions for $ f = 0 $ and/or $ g = 0 $ %%
%%%%%%%%%%%%%%%%%%%%%%%%%%%%%%%%%%%%%%%%%%%%%%%%%%%%%%%%%%%%%%%%%%%%%%%%%%%

\subsection{Stationary pulse solutions for $ f = 0 $ and $ g = 0 $}\label{sec:existence_f_g_zero}
The fast system reads
\begin{align}\label{eq:klausmeier_model_ODE_first_order_fast_autonomous}
\left\{
\begin{array}{rcl}
  \dot{u} & = & \varepsilon  p \, , \\[.1cm]
  \dot{p} & = & \varepsilon \left[ \varepsilon^2 \mu^2 u - \varepsilon^2 \mu  +  u v^2 \right] \, ,\\[.1cm]
  \dot{v} & = & q\, , \\[.1cm]
  \dot{q} & = & v - u v^2 \, .
 \end{array}
 \right. 
\end{align}
Note that this system possesses the symmetry $ (\xi, u, p, v, q) \rightarrow (-\xi, u, -p, v, -q) $. The corresponding slow system in the slow scaling $ \eta = \varepsilon \xi $ is given by
\begin{align}\label{eq:klausmeier_model_ODE_first_order_autonomous_slow}
\left\{
\begin{array}{rcl}
  u^\prime & = &  p \, , \\[.1cm]
  p^\prime & = & \varepsilon^2 \mu^2 u - \varepsilon^2 \mu  +  u v^2  \, ,\\[.1cm]
  \varepsilon v^\prime & = & q\, , \\[.1cm]
  \varepsilon q^\prime & = & v - u v^2 \, .
 \end{array}
 \right. 
\end{align}
Restricted to the invariant manifold
\begin{align}\label{eq:tilde_M}
  \widetilde{\mathcal{M}}:= \{ (u, p, 0, 0) ~|~ u > 0  \}
\end{align}
it reads
\begin{align}\label{eq:klausmeier_model_ODE_first_order_autonomous_on_M}
\left\{
\begin{array}{rcl}
  u^\prime & = &  p \, , \\[.1cm]
  p^\prime & = & \varepsilon^2 \mu^2 u - \varepsilon^2 \mu  \, ,
 \end{array}
 \right. 
\end{align}
which has a saddle structure around the fixed point $ \left(\frac{1}{\mu}, 0 \right) $ with stable and unstable eigenspaces given by
\begin{align}\label{eq:l_u_s_autonomous}
 \tilde{l}^{u/s}:= \left\{ (u, p) ~|~ p = \varepsilon \mu (u - \frac{1}{\mu}) \right\} \, .
\end{align}

%========%
% Remark %
%========%

\begin{remark}
Note that this step is much more intricate in the case of varying coefficients $ f,g $ where explicit solutions are possible only for very specific choices of coefficients. Therefore, one must resort to estimation techniques for the general case. Overcoming this difficulty using exponential dichotomies is the core contribution of the present work.
\end{remark}

The reduced fast system has the form
\begin{align}\label{eq:klausmeier_model_ODE_first_order_autonomous_fast_reduced}
\left\{
\begin{array}{rcl}
  \dot{u} & = & \ 0 \, , \quad \dot{p} \ = \ 0 \, ,\\[.1cm]
  \dot{v} & = & q\, , \\[.1cm]
  \dot{q} & = & v - u v^2 \, .
 \end{array}
 \right. 
\end{align}
A sketch of its planar subsystem $  \dot{v} = q, \dot{q} = v - u v^2 $ can be found in Figures~\ref{fig:fastReducedSystem}; this planar subsystem is a Hamiltonian system with Hamiltonian
\begin{align}\label{eq:hamiltonian_planar_subsystem}
 H(v, q; u) = \frac{1}{2} q^2 - \frac{1}{2} v^2  + \frac13 u v^3 \, . 
\end{align}
Its fixed point $ (v,q) = (0,0) $ features a saddle structure and a family of homoclinic orbits 
\begin{align}\label{eq:homoclinic}
\left\{
 \begin{array}{rcl}
  v_{hom}^{(0)}(\xi;u_0) &=& \frac{1}{u_0} \, \omega(\xi) \, , \quad \omega(\xi) : = \frac{3}{2} \, \mathrm{sech}^2(\xi/2) \, ,\\
  q_{hom}^{(0)}(\xi;u_0) &=& \dot{v}_{hom}(\xi;u_0) \, , \quad u_0 \in \mathbb{R}\backslash\{0\} \, , 
 \end{array}
\right.
\end{align}
connecting its stable and unstable manifolds. Hence, \eqref{eq:klausmeier_model_ODE_first_order_autonomous_fast_reduced} is a Hamiltonian system with Hamiltonian 
\begin{align}\label{eq:hamiltonian_autonomous}
  \widetilde{K}(u, p, v, q) = H(v, q; u)\, .
\end{align}
The invariant manifold $ \widetilde{\mathcal{M}} $ from \eqref{eq:tilde_M} is the collection of saddle points $ (u,p,0,0), u > 0, p \in \mathbb{R}, $ for \eqref{eq:klausmeier_model_ODE_first_order_autonomous_fast_reduced} and is, hence, normally hyperbolic. For its stable and unstable manifolds $ W_{0}^{s/u}(\widetilde{\mathcal{M}}) $ it holds true that $\dim [W_{0}^{s/u}(\widetilde{\mathcal{M}})] = 3 $ and, in fact, $ W_{0}^{s}(\widetilde{\mathcal{M}})$ and $W_{0}^{u}(\widetilde{M}) $ (partly) coincide, where the intersection is simply given by the family of homoclinic orbits. Moreover, we have that $ \widetilde{K}(u, p, v, q)|_{(u, p, v, q) \in \widetilde{\mathcal{M}}} = 0 $.

For $ \varepsilon > 0 $, we note that $ \widetilde{\mathcal{M}} $ is still an invariant manifold of the full system \eqref{eq:klausmeier_model_ODE_first_order_fast_autonomous}. It is a standard result in geometric singular perturbation theory (see, e.g. the classic articles \cite{Tik48, Fen79, Jon95} or, more recent, \cite{Kue15}) that, for $\varepsilon$ sufficiently small, its stable and unstable manifolds persist as $ W_{\varepsilon}^{s/u}(\widetilde{\mathcal{M}}) $ with $\dim[W_{\varepsilon}^{s/u}(\widetilde{\mathcal{M}})] = 3 $, but do not necessarily coincide anymore. In fact, they generically meet in a 2D intersection in $ \mathbb{R}^4 $.

In order to analyze the persistence of homoclinic orbits we measure the distance of $ W_{\varepsilon}^{s}(\widetilde{\mathcal{M}}) $ and $ W_{\varepsilon}^{u}(\widetilde{\mathcal{M}}) $ in the hyperplane $ \widetilde{R} = \{ (u,p,v,q) ~|~ q = 0 \} $, that is, we fix an even homoclinic orbit $ (u_{hom}, p_{hom}, v_{hom}, q_{hom}) $ with $ (u_{hom}(0), p_{hom}(0), v_{hom}(0), q_{hom}(0)) = (u_0, p_0, v_{max}, 0) $. To this end we use the Hamiltonian $ \widetilde{K} $ and analyze its difference during the jump of the orbit through the fast field
\begin{align}\label{eq:fast_field}
 I_f : = \left( - \frac{1}{\sqrt{\varepsilon}} \, , \frac{1}{\sqrt{\varepsilon}} \right) \, ,
\end{align}
by setting up
\begin{align}\label{eq:change_hamiltonian}
 \Delta_{I_f} \widetilde{K} = \widetilde{K}(1/\sqrt{\varepsilon}) - \widetilde{K}(-1/\sqrt{\varepsilon}) = \int_{I_f} \frac{d}{d \xi} \widetilde{K}(\xi) \, d \xi = \frac13 \varepsilon \int_{I_f} p(\xi) v_{hom}(\xi)^3 \, d \xi + h.o.t.
\end{align}
where we used that $ \frac{d}{d \xi} \widetilde{K} =  \frac{\partial}{\partial u} H(v,q;u) (\frac{du}{d \xi}) +  \frac{d}{d \xi}H(v,q;u) = \frac13 v^3 (\frac{du}{d \xi}) + 0  = \frac13 \varepsilon v^3 p$. We may set (using the fact that $p$ is constant to leading order)
$
 p(\xi)=  p^{(0)} + \varepsilon p^{(1)}(\xi) + h.o.t.
$ Therefore, in order to make this difference vanish to leading order, we evidently need that $p^{(0)} = 0$ and $p^{(1)}(0) = 0$.

Now that a departure and return mechanism from and back to $ \widetilde{\mathcal{M}} $ is established through the intersection $ W_{\varepsilon}^{s}(\widetilde{\mathcal{M}}) \cap W_{\varepsilon}^{u}(\widetilde{\mathcal{M}})  \cap R $, the remaining task is to determine possible take-off and touch-down points on $ \widetilde{\mathcal{M}} $ and investigate if these intersect the stable and unstable eigenspaces $ l^{s/u} $ appropriately to form a homoclinic. To this end we observe that
\begin{align}
% \Delta_{I_f} u
%
\Delta_{I_f} u &= u(1/\sqrt{\varepsilon}) - u(-1/\sqrt{\varepsilon}) = \int_{I_f} \frac{d}{d \xi} u(\xi) \, d \xi = \varepsilon^2 \int_{I_f} p^{(1)}(\xi) \, d \xi = \mathcal{O}(\varepsilon^{3/2})\, ,\\
% \Delta_{I_f} p 
%
\Delta_{I_f} p &= p(1/\sqrt{\varepsilon}) - p(-1/\sqrt{\varepsilon}) = \int_{I_f} \frac{d}{d \xi} p(\xi) \, d \xi = \varepsilon u_0 \int_{I_f} v_{hom}^{(0)}(\xi)^2 \, d \xi =  \frac{6}{u_0} \varepsilon + h.o.t. \, ,
\end{align}
so, to leading order, only the $ p $-variable changes during the fast jump, and therefore, the take-off and touch-down curves on $ \widetilde{\mathcal{M}}$ are to leading order given by
\begin{align}
 \widetilde{T}_{o/d}: = \left\{ \left. \left(u,p, 0, 0 \right) ~\right|~ p = \mp \frac{3\varepsilon}{u}, u > 0 \right\} \, ,
\end{align}
where we used that, by symmetry, to leading order
\begin{align}
 p(\pm 1/\sqrt{\varepsilon}) = p(0) \pm \frac12 \Delta_{I_f} p = \varepsilon \left( p^{(1)}(0) \pm \frac{3}{u_0} \right)\, . 
\end{align}
Finally, a straightforward computation of the intersection points of these with the stable and unstable eigenspaces $ l^{s/u} $ gives two possible homoclinics when $\mu \leq \frac{1}{12}$, with
\begin{align}\label{eq:intersection_points_aut}
 u_0^{\pm} = \frac{1 \pm \sqrt{1-12 \mu}}{2 \mu} \, \qquad \left(\mbox{for } \mu \leq \frac{1}{12}\right).
\end{align}

\begin{remark}\label{remark:u0_autonomous_mu_small}
	When $\mu \ll 1$, the expression for $u_0^\pm$,~\eqref{eq:intersection_points_aut}, can be expanded in terms of $\mu$; this yields for $u_0^\pm$ the following expansions
	\begin{equation}\label{eq:u0_autonomous_mu_small}
	\begin{array}{rcrcrcrcrc}
		u_0^- & = && & 3 &+& 9 \mu &+& \mathcal{O}(\mu^2) \\
		u_0^+ & = & \frac{1}{\mu} &-& 3 &-& 9 \mu &+& \mathcal{O}(\mu^2)
	\end{array}
	\end{equation}
\end{remark}

A conceptual sketch of the dynamics on $\widetilde{\mathcal{M}}$, along with an excursion through the fast field, is given in Figure~\ref{fig:manifoldPlusExcursion}. Moreover, in Figures~\ref{fig:const-3D} and~\ref{fig:const-UxU}, the evolution of a homoclinic solution is projected onto manifold $\widetilde{\mathcal{M}}$.

\begin{figure}
	\centering
		\begin{subfigure}[t]{0.33\textwidth}
			\centering
				\includegraphics[width=\textwidth]{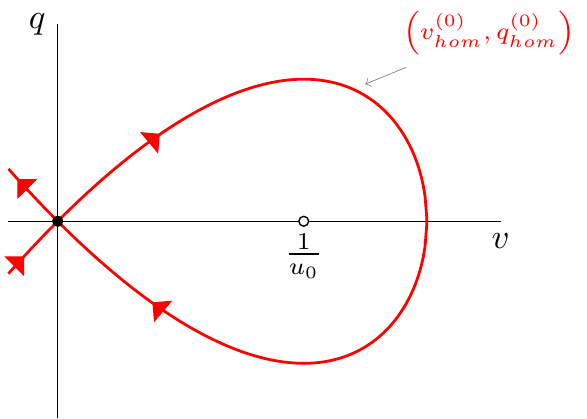}
			\caption{Sketch of fast reduced system}
			\label{fig:fastReducedSystem}
		\end{subfigure}
~
		\begin{subfigure}[t]{0.33\textwidth}
			\centering
				\includegraphics[width=\textwidth]{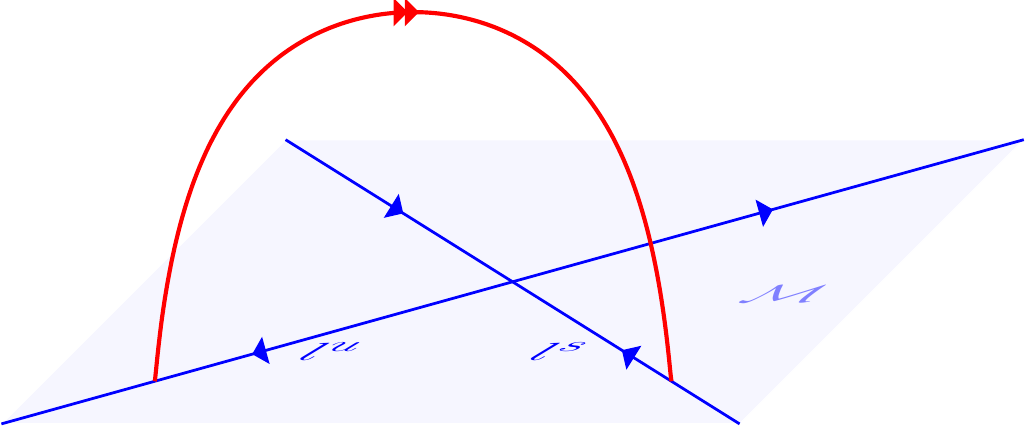}
			\caption{Sketch of homoclinic solution}
			\label{fig:manifoldPlusExcursion}
		\end{subfigure}
\caption{Sketches of the fast reduced system~\eqref{eq:klausmeier_model_ODE_first_order_autonomous_fast_reduced} (a) and the dynamics on the slow manifold $\mathcal{M}$ along with, in red, the excursion through the fast field (b).}
\end{figure}

\begin{figure}
	\centering
		\begin{subfigure}[t]{0.5\textwidth}
			\centering
				\includegraphics[width=\textwidth]{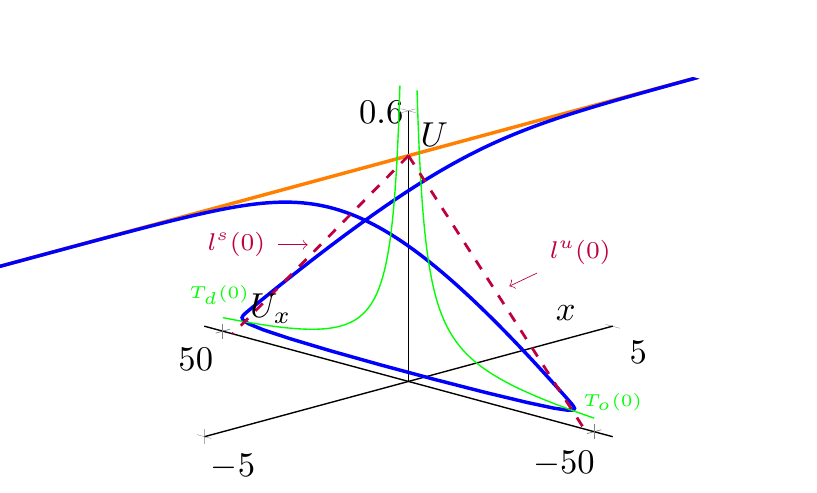}
			\caption{$(x,U,U_x)$-plane for $h(x) = 0$}
			\label{fig:const-3D}
		\end{subfigure}
~
		\begin{subfigure}[t]{0.35\textwidth}
			\centering
				\includegraphics[width=0.8\textwidth]{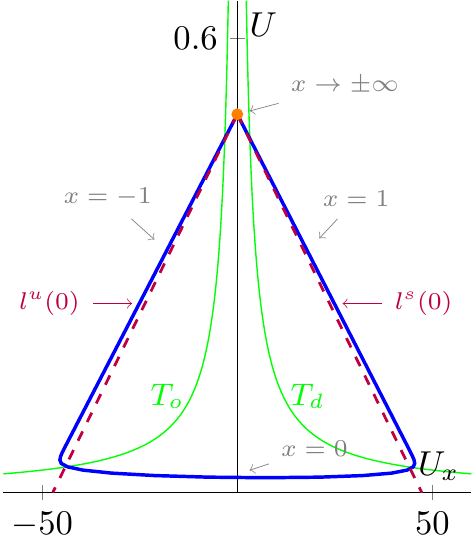}
			\caption{$(U,U_x)$-plane for $h(x) = 0$}
			\label{fig:const-UxU}
		\end{subfigure}
\\
		\begin{subfigure}[t]{0.5\textwidth}
			\centering
				\includegraphics[width=\textwidth]{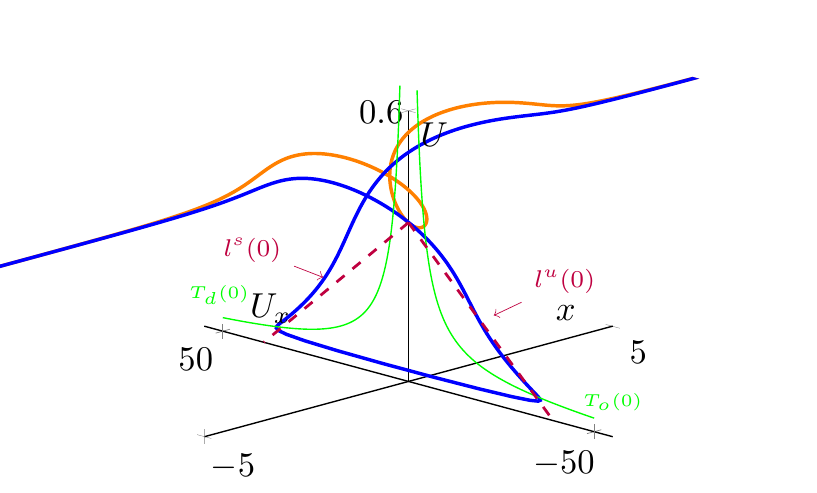}
			\caption{$(x,U,U_x)$-plane for $h(x) = \exp(-x^2/2)$}
			\label{fig:var-3D}
		\end{subfigure}
~
		\begin{subfigure}[t]{0.35\textwidth}
			\centering
				\includegraphics[width=0.8\textwidth]{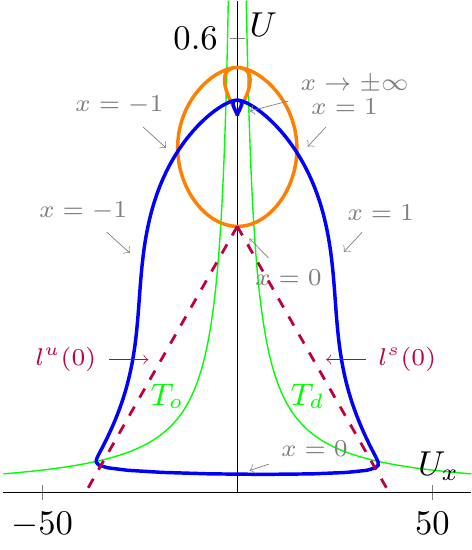}
			\caption{$(U,U_x)$-plane for $h(x) = \exp(-x^2/2)$}
			\label{fig:var-UxU}
		\end{subfigure}
\\
		\begin{subfigure}[t]{0.5\textwidth}
			\centering
				\includegraphics[width=\textwidth]{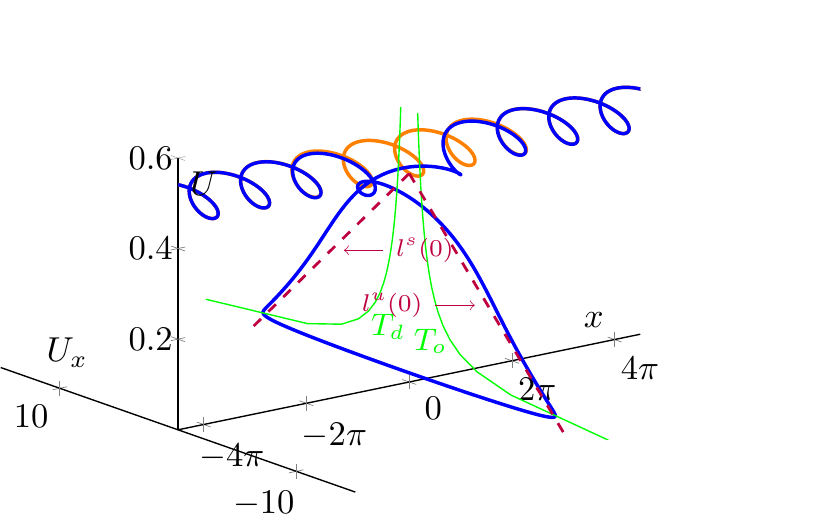}
			\caption{$(x,U,U_x)$-plane for $h(x) = 0.1 \cos(2x)$}
			\label{fig:cos-3D}
		\end{subfigure}
~
		\begin{subfigure}[t]{0.35\textwidth}
			\centering
				\includegraphics[width=0.8\textwidth]{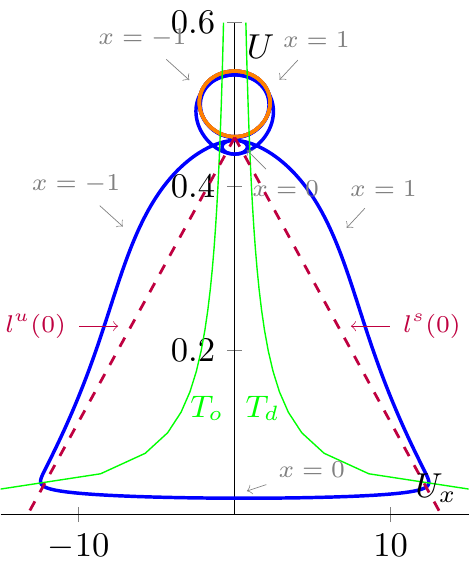}
			\caption{$(U,U_x)$-plane for $h(x) = 0.1 \cos(2x)$}
			\label{fig:cos-UxU}
		\end{subfigure}

\caption{Numerical simulations resulting in a stationary pulse solution for~\eqref{eq:klausmeier_model} with $f(x) = h'(x)$, $g(x) = h''(x)$, where $h(x) = 0$ (a,b), $h(x) = \exp(-x^2/2)$ (c,d) and $h(x) = 0.1 \cos(2x)$ (e,f). Shown are projections to the $(x,U,U_x)$-plane (a,c,e) and the $(U,U_x)$-plane (b,d,f) of a stationary pulse solution (blue) and the bounded solution $u_b$ to which the $U$-component converges for $|x| \rightarrow \infty$. Parts of the take-off and touch-down curves ($T_{o/d}$) along with stable and unstable manifolds at $x = 0$ are also sketched in green respectively red. Note that the plots in this figure correspond to the plots in Figure~\ref{fig:pulses}.}
\end{figure}

%%%%%%%%%%%%%%%%%%%%%%%%%%%%%%%%%%%%%%%%%%%%%%%%%%%%%%%%%%%%%%%%%%%%%%%%
%% SUBSECTION: Stationary pulse solutions for varying $ f $ and $ g $ %%
%%%%%%%%%%%%%%%%%%%%%%%%%%%%%%%%%%%%%%%%%%%%%%%%%%%%%%%%%%%%%%%%%%%%%%%%

\subsection{Stationary pulse solutions for varying $ f $ and $ g $}\label{sec:existence_varying_f_g}

First, we convert the non-autonomous system into an autonomous one by setting
\begin{align}
s(\xi): = \frac{D}{\sqrt{m}} \, \xi = \varepsilon^2 \mu \xi \, ,  
\end{align}
which gives the extended (autonomous) fast system
\begin{align}\label{eq:klausmeier_model_ODE_first_order_fast}
\left\{
\begin{array}{rcl}
  \dot{s} & = & \varepsilon^2 \mu \, , \\[.1cm]
  \dot{u} & = & \varepsilon  p \, , \\[.1cm]
  \dot{p} & = & \varepsilon \left[ \varepsilon^2 \mu^2 u - \varepsilon \mu f\left( s\right) p - \varepsilon^2 \mu^2 g\left( s\right) u - \varepsilon^2 \mu  +  u v^2 \right] \, ,\\[.1cm]
  \dot{v} & = & q\, , \\[.1cm]
  \dot{q} & = & v - u v^2 \, .
 \end{array}
 \right. 
\end{align}
It is important to note that the symmetry assumptions (A2) on $ f $ and $ g $ translate directly into a symmetry for \eqref{eq:klausmeier_model_ODE_first_order_fast} which is crucial for the construction of a homoclinic.

%======================================================================%
% Lemma (Symmetry of \eqref{eq:klausmeier_model_ODE_first_order_fast}) %
%======================================================================%

\begin{lemma}[\underline{Symmetry of \eqref{eq:klausmeier_model_ODE_first_order_fast}}]\label{lemma:symmetry_full_system}
Let the symmetry assumptions (A2) be fulfilled, that is, let $ f $ be an odd function and $ g $ be an even function. Then we have for \eqref{eq:klausmeier_model_ODE_first_order_fast} the symmetry $ (s, u, p, v, q) \rightarrow (-s, u, -p, v, -q) $.
\end{lemma}

The slow system corresponding to \eqref{eq:klausmeier_model_ODE_first_order_fast} in the slow variable $ \eta = \varepsilon \xi $ is given by
\begin{align}\label{eq:klausmeier_model_ODE_first_order_slow}
\left\{
\begin{array}{rcl}
  s^\prime & = & \varepsilon \mu \, , \\[.1cm]
  u^\prime & = &  p \, , \\[.1cm]
  p^\prime & = & \varepsilon^2 \mu^2 u - \varepsilon \mu f\left( s\right) p - \varepsilon^2 \mu^2 g\left( s\right) u - \varepsilon^2 \mu  +  u v^2 \, ,\\[.1cm]
  \varepsilon v^\prime & = & q\, , \\[.1cm]
  \varepsilon q^\prime & = & v - u v^2 \, .
 \end{array}
 \right. 
\end{align}
It possesses a three-dimensional invariant manifold
\begin{align}\label{eq:M}
 \mathcal{M}:=\{ (s, u, p, 0, 0) ~|~ u > 0, s, p \in \mathbb{R} \} \subset \mathbb{R}^5 \, ,
\end{align}
on which it takes the form
\begin{align}\label{eq:klausmeier_model_ODE_first_order_slow_on_M}
\left\{
\begin{array}{rcl}
  s^\prime & = & \varepsilon \mu \, , \\[.1cm]
  u^\prime & = &  p \, , \\[.1cm]
  p^\prime & = & \varepsilon^2 \mu^2 u - \varepsilon \mu f\left( s\right) p - \varepsilon^2 \mu^2 g\left( s\right) u - \varepsilon^2 \mu   \, . 
 \end{array}
 \right. 
\end{align}
which is an extension of the non-autonomous system
\begin{align}\label{eq:klausmeier_model_ODE_first_order_slow_on_M_nonautonomous}
\left\{
\begin{array}{rcl}
  u^\prime & = &  p \, , \\[.1cm]
  p^\prime & = & \varepsilon^2 \mu^2 u - \varepsilon \mu f\left( \varepsilon \mu \eta \right) p - \varepsilon^2 \mu^2 g\left(\varepsilon \mu \eta\right) u - \varepsilon^2 \mu   \, . 
 \end{array}
 \right. 
\end{align}
It is now convenient to introduce (or, actually, return to) the super-slow variable $ x = \varepsilon \mu \eta $. We set $ u(\eta) = \frac{1}{\mu} \hat{u}(\varepsilon \mu \eta) = \frac{1}{\mu} \hat{u}(x) $ and return to the second order non-autonomous setting
\begin{align}
\label{eq:klausmeier_model_ODE_first_order_slow_on_M_rescaled}
\left\{
\begin{array}{rcl}
  \frac{d}{dx} \hat{u} & = &  \hat{p} \, , \\[.1cm]
  \frac{d}{dx} \hat{p} & = & \hat{u} -  f\left( x\right) \hat{p} - g\left(x\right)  \hat{u} - 1\, . 
 \end{array}
 \right. 
\end{align}

%====================================================================================%
% Lemma (Symmetry of \eqref{eq:klausmeier_model_ODE_first_order_slow_on_M_rescaled}) %
%====================================================================================%

\begin{lemma}[\underline{Symmetry of \eqref{eq:klausmeier_model_ODE_first_order_slow_on_M_rescaled}}]\label{lemma:symmetry_reduced_system}
Let the symmetry assumptions (A2) be fulfilled, that is, let $ f $ be an odd function and $ g $ be an even function. Then we have for \eqref{eq:klausmeier_model_ODE_first_order_slow_on_M_rescaled} the symmetry $ (x, \hat{u}, \hat{p}) \rightarrow (-x, \hat{u}, -\hat{p}) $.
\end{lemma}

%========%
% Remark %
%========%

\begin{remark}
For conciseness, we note that we have three different scales:\\ 
\begin{center}
\parbox{12cm}{\bf fast scale $ \xi \, , \quad $ slow scale $ \eta = \varepsilon \xi \, , \quad $ super-slow scale $ x = \varepsilon \mu \eta = \varepsilon^2 \mu \xi $\\} 
\end{center}
The construction that we illustrate in this article therefore relies heavily on assumption (A1). The specific definition of the small parameter is convenient since the fast reduced system is an ODE which is known to have homoclinic solutions and the slow system on the critical manifold $ \mathcal{M} $ is a linear planar system.
\end{remark}

\begin{remark}\label{remark:scaling_p_phat}
Note the difference between $p = \frac{du}{d\eta}$ and $\hat{p} = \frac{d\hat{u}}{dx}$. Hence, $p = \varepsilon \hat{p}$.
\end{remark}

%=========================================%
% Proposition (Dynamics on $\mathcal{M}$) %
%=========================================%

\begin{proposition}[\underline{Dynamics on $\mathcal{M}$}]\label{prop:dynamics_slow_manifold}
Consider the slow system on $ \mathcal{M} $ \eqref{eq:klausmeier_model_ODE_first_order_slow_on_M} with $ f, g $ fulfilling (A3). Then there exists a unique bounded solution $ (\hat{u}_b, \hat{p}_b) $ of \eqref{eq:klausmeier_model_ODE_first_order_slow_on_M_rescaled} and corresponding connected set $\Gamma \subset \mathbb{R} \cup \{\infty\}$ such that the following holds true: For each fixed $ x \in \mathbb{R} $ there exists $ C^{s/u}(x) \in \Gamma $ and lines 
\begin{align}\label{eq:lines}
 l^{s/u}(x) := \{ (\hat{u},\hat{p}) ~|~ \hat{p} - \hat{u}_b'(x) = C^{s/u}(x) (\hat{u} - \hat{u}_b(x)) \} \, , 
\end{align}
such that the solution to the initial value problem \eqref{eq:klausmeier_model_ODE_first_order_slow_on_M_rescaled} with $ (\hat{u}(x), \hat{p}(x)) =  (\hat{u}_0,\hat{p}_0) \in l^{s}(x)$ converges to $ (\hat{u}_b, \hat{p}_b) $ for $ x \rightarrow \infty $, while with $ (\hat{u}(x), \hat{p}(x)) =  (\hat{u}_0,\hat{p}_0) \in l^{u}(x)$ it converges to $ (\hat{u}_b, \hat{p}_b) $ for $ x \rightarrow -\infty $. Moreover, if $f$ and $g$ fulfill the symmetry assumption (A2), $C^{s/u}$ posses the symmetry $C^s(x) = - C^u(-x)$ for all $x \in \mathbb{R}$. In particular, $C^s(0) = - C^u(0)$.
\end{proposition}

The proof of Proposition~\ref{prop:dynamics_slow_manifold} constitutes the contents of section~\ref{sec:dynamics_slow_manifold}. Also note the similarities with Proposition~\ref{proposition:bounded_solution}, since the bounded solutions mentioned in both Propositions are identical up ot the scaling $\hat{u}_b(x) = \mu u_b(\xi)$.

\begin{remark}
When $\lim_{x \rightarrow \pm \infty} f(x),g(x) = 0$ (i.e. assumption (A4)), the unique bounded solution $(\hat{u}_b,\hat{p}_b)$ limits to the fixed point of the autonomous equation. That is,
\begin{align}
 \lim_{x \rightarrow \pm \infty} (\hat{u}_b(x), \hat{p}_b(x))  = (1,0) \, .
\end{align}
\end{remark}

This result implies that there are trajectories on $ \mathcal{M} $ that lead to and away from the bounded solution $ (\hat{u}_b, \hat{p}_b) $. Hence, the only remaining construction steps are the analysis of persistence of orbits biasymptotic to $ \mathcal{M} $ and their touch-down/take-off locations. We therefore switch back to the fast system and examine the dynamics during the jump of an orbit through the fast field. In order to pass to the reduced fast system, we use the assumption (A5) so, in the limit $ \varepsilon \rightarrow 0 $, we get the reduced fast system
\begin{align}\label{eq:klausmeier_model_ODE_first_order_fast_reduced}
\left\{
\begin{array}{rcl}
  \dot{s} & = & 0 \, , \quad \dot{u} \ = \ 0 \, , \quad \dot{p} \ = \ 0 \, ,\\[.1cm]
  \dot{v} & = & q\, , \\[.1cm]
  \dot{q} & = & v - u v^2 \, .
 \end{array}
 \right. 
\end{align}
Note that in the reduced fast system the non-autonomous character of our problem is not visible. The only difference is the added trivial equation $ \dot{s} = 0 $. As alluded to in the constant coefficient case in section~\ref{sec:existence_f_g_zero}, the planar subsystem $ \dot{v} = q, \dot{q} = v - u v^2 $ is known to be Hamiltonian and features a homoclinic to the saddle point $ (v,q) = (0,0) $ which can be specified explicitly (see \eqref{eq:homoclinic}). As a result, also \eqref{eq:klausmeier_model_ODE_first_order_fast_reduced} is Hamiltonian with
\begin{align}\label{eq:hamiltonian_nonautonomous}
  K(s, u, p, v, q) = H(v, q; u)\, .
\end{align}
The invariant manifold $ \mathcal{M} $ from \eqref{eq:M} is the collection of saddle points $ (s,u,p,0,0), u > 0, s, p \in \mathbb{R}, $ for \eqref{eq:klausmeier_model_ODE_first_order_fast_reduced} and is, hence, normally hyperbolic. For its stable and unstable manifolds $ W_{0}^{s/u}(\mathcal{M}) $ it holds true that $\dim [W_{0}^{s/u}(\mathcal{M})] = 4 $ and, in fact, $ W_{0}^{s}(\mathcal{M})$ and $W_{0}^{u}(\mathcal{M}) $ (partly) coincide, where the intersection is simply given by the family of homoclinic orbits. Moreover, we have that $ K(s, u, p, v, q)|_{(s,u, p, v, q) \in {\mathcal{M}}} = 0 $.

The analogy with the constant coefficient case continues for $ \varepsilon > 0 $ sufficiently small; we still have that $ {\mathcal{M}} $ is an invariant manifold of the full system \eqref{eq:klausmeier_model_ODE_first_order_fast} and that its stable and unstable manifolds persist as $ W_{\varepsilon}^{s/u}(\mathcal{M}) $ with $\dim[W_{\varepsilon}^{s/u}(\mathcal{M})] = 4 $, but do not necessarily coincide anymore. In fact, they generically meet in a 3-D intersection in $ \mathbb{R}^5 $.

%======================================================%
% Proposition (Persistence of a homoclinic connection) %
%======================================================%

\begin{proposition}[\underline{Persistence of a homoclinic connection}]\label{prop:persistence}
Let $\varepsilon$ be sufficiently small.

\begin{enumerate}

% (1)
\item Define the hyperplane $ R = \{ (s, u, p, v, q) ~|~ q = 0 \} $. Then $\dim[W_{\varepsilon}^{s}(\mathcal{M}) \cap W_{\varepsilon}^{u}(\mathcal{M}) \cap R  ] = 2 $ and orbits in this intersection fulfill $ p(\xi) = \varepsilon p^{(1)}(\xi) + h.o.t. $, that is, the leading order constant term $ p^{(0)} $ vanishes.

% (2)
\item The take-off and touch-down surfaces on $ \mathcal{M} $ of orbits in the intersection $ W_{\varepsilon}^{s}(\mathcal{M}) \cap W_{\varepsilon}^{u}(\mathcal{M}) \cap R $ are to leading order given by
\begin{align}\label{eq:take_off_touch_down_surfaces}
 {T}_{o/d}(s): = \left\{ \left. \left(s, u, p, 0, 0 \right) ~\right|~ p = \mp \frac{3\varepsilon}{u}, u > 0 \right\} \, .
\end{align}

% (3)
\item For orbits in the intersection $ W_{\varepsilon}^{s}(\mathcal{M}) \cap W_{\varepsilon}^{u}(\mathcal{M}) \cap R $ the touch-down curve $ {T}_{d}(0) $ and stable line $ l^s(0) $ from \eqref{eq:lines} intersect in at most two points
\begin{align}\label{eq:intersection_points}
 u_0^{\pm} = \frac{u_b(0) \pm \sqrt{u_b(0)^2+ 12/ (\mu C^s(0))}}{2} \, ,
\end{align}
where $ C^s(0) $ is the slope of the stable line $ l^s(0) $ from \eqref{eq:lines} and $ \hat{u}_b = \mu u_b $ is the (rescaled) bounded background solution from Proposition~\ref{prop:dynamics_slow_manifold}. By symmetry, the analogous is true for the take-off curve $ {T}_{o}(0) $ and unstable line $ l^u(0) $ from \eqref{eq:lines}. In particular, the thus computed $u_0^\pm$-values coincide by the aforementioned symmetry $C^u(0) = -C^s(0)$ -- see Proposition~\ref{prop:dynamics_slow_manifold}.

% (4)
\item There are two even homoclinic orbits for \eqref{eq:klausmeier_model_ODE} with $u_0^\pm > 0$ in case $ u_b(0)^2+ 12/(\mu C^s(0)) > 0 $ and $ u_b(0) - \sqrt{u_b(0)^2+ 12/ (\mu C^s(0))} > 0 $.
\end{enumerate}

\end{proposition}

%========%
% Remark %
%========%

\begin{remark}
 If we set $ u_b(0) = \frac{1}{\mu} $ and $ C^s(0) = -1 $ in \eqref{eq:intersection_points}, we recover \eqref{eq:intersection_points_aut}.
\end{remark}

\begin{proof}
Measuring the distance of $ W_{\varepsilon}^{s}({\mathcal{M}}) $ and $ W_{\varepsilon}^{u}({\mathcal{M}}) $ in the hyperplane $ R $ can again be accomplished using the difference of the Hamiltonian $ K $ during the fast the jump of the orbit through the fast field \eqref{eq:fast_field}. We have exactly as in the constant coefficient case \eqref{eq:change_hamiltonian} where (using that $ p $ is constant to leading order) we have set
$
 p(\xi)=  p^{(0)} + \varepsilon p^{(1)}(\xi) + h.o.t. \, ,
$
and used that $ \frac{d}{d \xi} {K} = \frac{\partial}{\partial s} K(s, u, p, v, q)(\frac{ds}{d \xi}) + \frac{\partial}{\partial u} H(v,q;u) (\frac{du}{d \xi}) +  \frac{d}{d \xi}H(v,q;u) = 0 + \frac13 v^3 (\frac{du}{d \xi}) + 0  = \frac13 \varepsilon v^3 p$. In order to make this difference vanish to leading order, we evidently need that $ p^{(0)} = 0 $ and $p^{(1)}(0) = 0$. This proves the first statement.

In order to construct the take-off and touch-down curves, we again investigate the change of the fast variables during the jump through the fast field:
\begin{align}
% \Delta_{I_f} s
%
\Delta_{I_f} s &= s(1/\sqrt{\varepsilon}) - s(-(1/\sqrt{\varepsilon})) = \int_{I_f} \frac{d}{d \xi} s(\xi) \, d \xi = \frac{2}{\sqrt{\varepsilon}} \, \varepsilon^2 \mu = \mathcal{O}(\varepsilon^{3/2})\, ,\\
% \Delta_{I_f} u
%
\Delta_{I_f} u &= u(1/\sqrt{\varepsilon}) - u(-(1/\sqrt{\varepsilon})) = \int_{I_f} \frac{d}{d \xi} u(\xi) \, d \xi = \varepsilon^2 \int_{I_f} p^{(1)}(\xi) \, d \xi = \mathcal{O}(\varepsilon^{3/2})\, ,\\
% \Delta_{I_f} p 
%
\Delta_{I_f} p &= p(1/\sqrt{\varepsilon}) - p(-(1/\sqrt{\varepsilon})) = \int_{I_f} \frac{d}{d \xi} p(\xi) \, d \xi = \varepsilon u_0 \int_{I_f} v_{hom}^{(0)}(\xi)^2 \, d \xi =  \frac{6}{u_0} \varepsilon + h.o.t. \, ,
\end{align}
Hence, to leading order, only the $ p $-variable changes during the fast jump, and therefore, the take-off and touch-down curves on $ {\mathcal{M}}$ are to leading order given by \eqref{eq:take_off_touch_down_surfaces} where we used that, by symmetry,
$
 p(\pm 1/\sqrt{\varepsilon}) = p(0) \pm \frac12 \Delta_{I_f} p \, . 
$
This proves the second statement.

Equating~\eqref{eq:take_off_touch_down_surfaces} and~\eqref{eq:lines} (where we used that $p = \varepsilon \hat{p}$ -- see Remark~\ref{remark:scaling_p_phat}) gives the equality
\begin{align}
 \varepsilon \mu C^{s}(0)\left( u_0 - u_b(0) \right) = \frac{3 \varepsilon}{u_0};
\end{align}
the solutions of which give the claimed expression \eqref{eq:intersection_points} in the third statement. Finally, the fourth statement follows from inspecting~\eqref{eq:intersection_points}.
\end{proof}

Two examples of homoclinic solutions for varying $f$ and $g$ can be found in Figures~\ref{fig:var-3D}--\ref{fig:cos-UxU}. In these figures the evolution of a homoclinic solution is projected onto the manifold $\mathcal{M}$, which shows the essence of Proposition~\ref{prop:persistence}.

Proposition~\ref{prop:persistence} thus establishes existence of homoclinic solutions for~\eqref{eq:klausmeier_model_ODE} under the conditions stated in Proposition~\ref{prop:persistence}(4). However, in the case of varying coefficients, there typically are no explicit expressions available for the bounded solution $u_b(0)$ and the constant $C^s(0)$. To circumvent this, in the next section we derive bounds on these using the theory of exponential dichotomy, which simultaneously forms the proof of Proposition~\ref{prop:dynamics_slow_manifold}.

%%%%%%%%%%%%%%%%%%%%%%%%%%%%%%%%%%%%%%%%%%%%%%%%%%%%%%%%%%%%%%%%%%%%%%%%%%%%%%
% SUBSECTION (Some basic results from the theory of exponential dichotomies) %
%%%%%%%%%%%%%%%%%%%%%%%%%%%%%%%%%%%%%%%%%%%%%%%%%%%%%%%%%%%%%%%%%%%%%%%%%%%%%%

\subsection{Some basic results from the theory of exponential dichotomies}\label{sec:exp_dich}
When $f$ and/or $g$ are non-constant, generically it is not possible to capture the dynamics on manifold $\mathcal{M}$ in explicit expressions. Instead, our main tools for constructing a saddle-like structure on $\mathcal{M}$ are from the theory of exponential dichotomies. To fix notation and keep the exposition self-contained, we state (following \cite{coppel1978stability}) the definition of exponential dichotomies along with a selection of results that we use here. 

%====================================%
% Definition (Exponential Dichotomy) %
%====================================%

\begin{definition}[\underline{Exponential Dichotomy}]\label{def:exp_dich}
Consider the planar ODE $\frac{d}{dx} Y = B(x) Y$ for the unknown $ Y: \mathbb{R} \rightarrow \mathbb{R}^2 $ and with $B: \mathbb{R} \rightarrow \mathbb{R}^{2 \times 2} $ a matrix-valued function which is continuous on $\mathbb{R}$. Let $\Phi= \Phi(x)$ be the associated canonical solution operator. This ODE is said to have an \emph{exponential dichotomy} if there is a projection matrix $P$ and positive constants $K$ and $\rho$ such that
\begin{align*}
 \| \Phi(x)\ P\ \Phi^{-1}(\tilde{x}) \| &\leq K e^{-\rho (x - \tilde{x})} \, , \quad x \geq \tilde{x} \, , \\
 \| \Phi(x)\ (I-P)\ \Phi^{-1}(\tilde{x}) \| & \leq K e^{+\rho (x-\tilde{x})} \, , \quad x \leq \tilde{x} \, .
\end{align*}
\end{definition}

In the next section we will be interested in first order ODEs of the form
\begin{equation}\label{eq:inhom_nonaut}
\frac{d}{dx} Y = [A_0 + A(x)] Y + F \, , 
\end{equation}
with $ x \in \mathbb{R}, Y : \mathbb{R} \rightarrow \mathbb{R}^2, A_0 \in \mathbb{R}^{2 \times 2}, A : \mathbb{R} \rightarrow \mathbb{R}^{2\times2}, F \in \mathbb{R}^2$. In particular, we would like to corroborate knowledge of the autonomous version (which is often available in terms of explicit solutions) to deduce qualitative results for the full non-autonomous one. For the sake of clarity, we assemble first all auxiliary systems in one place:
\begin{align}
\intertext{First, we have the homogeneous, autonomous system}
\label{eq:hom_aut}  \frac{d}{dx} Z_h &= A_0 Z_h.\\
\intertext{Then, there is the homogeneous, non-autonomous system}
\label{eq:hom_nonaut} \frac{d}{dx} Y_h &= [A_0 + A(x)] Y_h.\\
\intertext{Finally, we have the inhomogeneous, autonomous system}
\label{eq:inhom_aut}\frac{d}{dx} Z \ &= A_0 Z + F.
\end{align}

%=========================%
% Proposition (Roughness) %
%=========================%

\begin{proposition}[\underline{Roughness and closeness of bounded solutions}]\label{prop:roughness_closeness_general}
Let $ K_{aut}, \rho_{aut} > 0 $ be the exponential dichotomy constants of the homogeneous, autonomous ODE \eqref{eq:hom_aut} and $ \Phi_{aut}, P_{aut} $ the corresponding solution and projection operators. If
\begin{align}\label{eq:delta_general}
 \delta : = \sup_{x \in \mathbb{R}} ||| A(x) ||| < \frac{\rho_{aut}}{4K_{aut}^2} \, ,
\end{align} 
the non-autonomous ODE~\eqref{eq:hom_nonaut} has an exponential dichotomy for which the following holds true.
\begin{enumerate}
% Roughness
\item{\underline{(Roughness)}} The exponential dichotomy constants of the homogeneous, non-autonomous ODE \eqref{eq:hom_nonaut} are $ K = \frac{5}{2} K_{aut}^2 $ and ${\rho} = \rho_{aut} - 2 K_{aut} \delta$, and concerning the solution and projection operators ${\Phi}, {P}$ of \eqref{eq:hom_nonaut} we have upon defining
\begin{align}
 {Q}(x) := {\Phi}(x) {P} {\Phi}^{-1}(x)\, , \quad Q_{aut}(x) := {\Phi_{aut}}(x) {P_{aut}} {\Phi_{aut}}^{-1}(x)  \, 
\end{align}
the estimate
\begin{align}\label{eq:roughnessProjections_general_Q}
||| {Q}(x) - Q_{aut}(x) ||| \leq \frac{4 K_{aut}^3 \delta}{\rho_{aut}} \, , \quad x \in \mathbb{R} \, .
\end{align}
% Closeness of bounded solutions
\item{\underline{(Closeness of bounded solutions)}} There exist unique bounded solutions $ Z_{b, aut}, Y_{b} $ of the inhomogeneous, autonomous and non-autonomous ODEs \eqref{eq:inhom_aut} and \eqref{eq:inhom_nonaut}. In particular, they satisfy
\begin{align}\label{eq:bounded_solutions_estimate}
\sup_{x \in \mathbb{R}} |||Y_b(x)-Z_{b, aut}(x)||| \leq \frac{4 \delta K_{aut} {K}}{\rho_{aut} {\rho}} \,  \| F \| \, . 
\end{align}

\end{enumerate}
\end{proposition}

% PROOF
\begin{proof}
The first statement is the persistence of exponential dichotomies, known as ``roughness'', and is a standard result (see \cite[Ch.4, Prop.1]{coppel1978stability}). Moreover, another standard result from the theory of exponential dichotomies stipulates that inhomogeneous equations have unique bounded solutions, when the homogeneous equations have an exponential dichotomy and the inhomogeneous terms are bounded (see~\cite[Ch.8, Prop.2]{coppel1978stability}). Then, to demonstrate the rest of the second statement, we define $ W(x) = {Y}_b(x) - Z_{b,aut}(x) $ which gives $ W'(x) = A_0 W(x) + G(x) $ with $ G(x) = A(x) {Y}_b(x) $. The unique bounded solution $ W_b $ of this ODE satisfies the estimate
$
 \sup_{x \in \mathbb{R}} \| W_b(x)\| \leq \frac{2 {K}_{aut}}{{\rho_{aut}}} \,  \sup_{x \in \mathbb{R}} \| G(x)\|  \leq \frac{4 \delta K_{aut} {K}}{\rho_{aut} {\rho}} \,  \| F \| \, ,
$
where we used that
$
 \sup_{x \in \mathbb{R}} \| {Y}_b(x) \| \leq  \frac{2 {K}}{{\rho}}  \| F\| \, .
$

\end{proof}

%%%%%%%%%%%%%%%%%%%%%%%%%%%%%%%%%%%%%%%%%%
% SUBSECTION (Dynamics on $\mathcal{M}$) %
%%%%%%%%%%%%%%%%%%%%%%%%%%%%%%%%%%%%%%%%%%

\subsection{Dynamics on $\mathcal{M}$ (Proof of Proposition~\ref{prop:dynamics_slow_manifold})}\label{sec:dynamics_slow_manifold}
Let us introduce the more concise notation $Y = \left(\hat{u},\frac{d}{dx}\hat{u}\right)^T$ such that \eqref{eq:klausmeier_model_ODE_first_order_slow_on_M_rescaled} has the form of \eqref{eq:inhom_nonaut} from the previous section; that is,
\begin{equation}\label{eq:slow-system}
\frac{d}{dx} Y = [A_0 + A(x)] Y + F \, , 
\end{equation}
with
\begin{equation}\label{eq:definition_A_0_A_x_F}
A_0  = \left( \begin{array}{cc} 0 & 1 \\ 1  & 0\end{array} \right) \, , \quad 
A(x) = \left( \begin{array}{cc} 0 & 0 \\  - g(x) & -f(x) \end{array} \right) \, , \quad 
   F = \left( \begin{array}{c} 0 \\ -1 \end{array} \right) \, . 
\end{equation}

%=========================================%
% Lemma (Exponential dichotomy constants) %
%=========================================%

\begin{lemma}[\underline{Exponential Dichotomy Constants and Roughness}]\label{lemma:exp_dich_const_roughness}
With the notation of Proposition~\ref{prop:roughness_closeness_general}, let 
\begin{align}
 \delta = \sup_{x \in \mathbb{R}} \sqrt{f(x)^2 + g(x)^2} < \frac14 \, .
\end{align}
Then we have $ \rho_{aut}= K_{aut} = 1, \rho = 1-2 \delta, K = 5/2 $ and
\begin{align}
||| {Q}(x) - Q_{aut}(x) ||| \leq 4 \delta \, , \quad x \in \mathbb{R} \, .
\end{align}

\end{lemma}

% PROOF

\begin{proof}
We have the canonical solution operator $\Phi(x) = e^{A_0 x}$. The eigenvalues of the matrix $A_0$ are $\pm 1$ and the corresponding normed eigenvectors are $v =\frac{1}{\sqrt{2}}(1, 1)^T, w = \frac{1}{\sqrt{2}}(1, -1)^T$. Thus the fixed point $Y = (0,0)^T$ is a saddle. From this it is clear that we can choose
\begin{equation*}
 P = w w^T = \frac{1}{2} \left( \begin{array}{cc} 1 & -1 \\ -1 & 1 \end{array} \right) \, .
\end{equation*}
With the basis transformation matrix $B = \left(  v ~|~ w \right)$ and the diagonal matrix $D = \mbox{diag}(1,-1)$ we then get
\begin{align*}
\|\Phi(x) P \Phi^{-1}(s) \| 
 = \| B e^{Dx} B^{-1} P B e^{-Ds} B^{-1} \| 
 = \left\| \left( \begin{array}{cc} 1&-1\\-1&1 \end{array}\right) \right\| \frac{ e^{-(x-s)}}{2} 
 = e^{-(x-s)} \, .
\end{align*}
A similar reasoning -- where one can use that $I-P = v v^T$ -- gives
\begin{equation*}
	\|\Phi(x) (I-P) \Phi^{-1}(s)\| = e^{(x-s)} \, .
\end{equation*}
Thus we have the estimate for exponential dichotomies from Definition~\ref{def:exp_dich} with $\rho_{aut} = 1$ and $K_{aut} = 1$. The remaining statements can now be read off Proposition~\ref{prop:roughness_closeness_general}. 
\end{proof}

The roughness of exponential dichotomies established in Lemma~\ref{lemma:exp_dich_const_roughness} provides a bound on the projection operator $Q(x)$ of the non-autonomous system. However, this bound cannot be used directly to prove existence of homoclinic solutions using geometric singular perturbation theory, as geometric properties need to be derived. In particular, we need to find the stable and unstable manifolds for the unique bounded solution $Y_b$ of~\eqref{eq:slow-system}. These can be defined as
\begin{align}
 W^s(Y_b) &: = \left\{ (x,Y^s(x)) ~|~ Y^s(x) = Y_b(x)+ \Phi(x)P\Phi^{-1}(x)r \, , r \in \mathbb{R}^2 \right\} \, ,\\
 W^u(Y_b) &: = \left\{ (x,Y^u(x)) ~|~ Y^u(x) = Y_b(x)+ \Phi(x)(Id-P)\Phi^{-1}(x)r \, , r \in \mathbb{R}^2 \right\} \, ,
\end{align}
where $ \Phi, P $ are the solution and projection operator for \eqref{eq:slow-system}. For the construction that we have in mind, it is convenient to notice that
\begin{align}
 W^{s/u}(Y_b) = \bigcup_{x \in \mathbb{R}} (x,l^{s/u}(x)) \, , 
\end{align}
with lines
\begin{align}
 l^{s}(x) &= \left\{ Y^s(x) ~|~ Y^s(x) = Y_b(x)+ \Phi(x)P\Phi^{-1}(x)r \, , r \in \mathbb{R}^2 \right\} \, , \\
 l^{u}(x) &= \left\{ Y^u(x) ~|~ Y^u(x) = Y_b(x)+ \Phi(x)(I-P)\Phi^{-1}(x)r \, , r \in \mathbb{R}^2 \right\} \, .
\end{align}
While, in general, it is not possible to find explicit expressions for these objects, we can derive estimates for their locations. For this we first observe that the line $l^s$ can be written equivalently as
\begin{align}
	l^s(x) = \left\{ (\hat{u},\hat{p}) ~|~ \hat{p} - \hat{u}_b'(x) = C(x) (\hat{u}-\hat{u}_b(x)) \right\} \, ,
\end{align}
where $C(x)$ is the slope of the line. Starting from the bound on the projection operator $Q(x) = \Phi(x)P\Phi^{-1}(x)$ derived in Lemma~\ref{lemma:exp_dich_const_roughness}, a bound on the projection lines will be established in Lemma~\ref{lemma:closeness_projection_lines}, which is then subsequently used to find a bound on the slope $C(x)$ via the angle $\theta(x)$ of the line in Lemma~\ref{lemma:closeness_slopes}. 

In particular, for the case of~\eqref{eq:klausmeier_model_ODE_first_order_slow_on_M_rescaled}, we thus obtain
\begin{align}
 l^{s}(x) = \left\{ (\hat{u},\hat{p}) ~|~ \hat{p} - \hat{u}_b'(x) = (-1 + \tilde{C}(x))(\hat{u} - \hat{u}_b(x)) \right\} \, ,
\end{align}
with $ \tilde{C}(x) $ as in Lemma \ref{lemma:closeness_slopes} taking into account that the projection operator depends on $ x $, that is, $Q = Q(x)$ and so does the angle $ \theta = \theta(x) $, which defines $ C = C(x) $ and, hence, also $ \tilde{C} = \tilde{C}(x) $.

The rest of this section consists of the two technical lemmas that ultimately derive a bound for $\tilde{C}$.

%=======================================%
% Lemma (Closeness of projection lines) %
%=======================================%

\begin{lemma}[\underline{Closeness of projection lines}]\label{lemma:closeness_projection_lines}
Let $Q$ and ${Q}_{aut}$ be the projection matrices with rank $1$ as defined in Proposition~\ref{prop:roughness_closeness_general}(i), i.e. there are unit vectors $q$ and $q_{aut}$ such that $Q = q q^T$ and ${Q}_{aut} = {q}_{aut} {q}_{aut}^T$, and $\|Q - {Q}_{aut}\| < 4 \delta$ holds true. Then either $\|q - {q}_{aut} \| < \sqrt{ 8 \delta }$ or $\|q + {q}_{aut} \| < \sqrt{ 8 \delta }$.
\end{lemma}

\begin{proof}
We prove the equivalent statement that from $\|q - q_{aut} \| \geq \sqrt{8 \delta}$ and $\|q + q_{aut} \| \geq \sqrt{8\delta}$ it follows that $\|Q - Q_{aut}\| \geq 4 \delta$. First we observe that
\begin{align}
 (q-q_{aut})(q^T+q_{aut}^T)(q+q_{aut}) &= (q q^T - q_{aut}q_{aut}^T)(q+q_{aut}) + (q q_{aut}^T - q_{aut} q^T) (q+q_{aut}) \nonumber\\ &= 2 (q q^T- q_{aut} q_{aut}^T)(q+q_{aut}) = 2 (Q-Q_{aut})(q+q_{aut}) \, .
\end{align}
Therefore, by assumption
\begin{align*}
& \|Q - Q_{aut}\| \, \|q+q_{aut} \| \geq \| (Q-Q_{aut})(q+q_{aut})\| \\
   &=   \frac12 \| (q-q_{aut})(q^T+q_{aut}^T)(q+q_{aut}) \| 
   =   \frac12 \| q+q_{aut}\|^2 \|q - q_{aut}\| 
  \geq 4 \delta \|q+q_{aut}\| \, ,
\end{align*}
from which it follows that $ \|Q - Q_{aut}\| \geq 4 \delta  $.
\end{proof}

The previous lemma establishes closeness of projection lines of the autonomous and the non-autonomous case. The thus obtained bounds on norms can be transferred to bounds on the slope $C$ by use of elementary geometry. Note that transforming the norm bounds in this way leads to singularities when a projection line passes the vertical axis (which also leads to a seemingly disjoint set of admittable slopes). A visualisation of the results of Lemma~\ref{lemma:closeness_slopes} are given in Figure~\ref{fig:closenessSlopeVisualisation}. In particular, the resulting bounds for the slope are shown.

%=============================%
% Lemma (Closeness of slopes) %
%=============================%

\begin{lemma}[\underline{Closeness of slopes}] \label{lemma:closeness_slopes}
Let $Q$ and ${Q}_{aut}$ be the projection matrices with rank $1$ as defined in Proposition~\ref{prop:roughness_closeness_general}(i), i.e. there are unit vectors $q$ and $q_{aut}$ such that $Q = q q^T$ and ${Q}_{aut} = {q}_{aut} {q}_{aut}^T$, and $\|Q - Q_{aut}\| < 4 \delta$ holds true. Furthermore, let $ \theta, \theta_{aut} \in [-\pi, \pi) $ be defined by $ q =: (\cos(\theta), \sin(\theta)), q_{aut} = (\cos(\theta_{aut}), \sin(\theta_{aut})) $ such that the slopes of the lines spans by $ q $ and $ q_{aut} $ are given by
\begin{align}
 C := \tan(\theta) \, , \qquad C_{aut} := \tan(\theta_{aut}) \, .
\end{align}
Then there exist constants $C_{min/max}(\delta,C_{aut})$ defined by
\begin{align}
	C_{\mathrm{min}}(\delta,C_{aut}) & := 
\begin{cases} -(1+C_{aut}^2) \frac{2 \sqrt{2}\sqrt{\delta} \sqrt{1-2\delta}}{(1-4\delta) + 2 C_{aut} \sqrt{2}\sqrt{\delta}\sqrt{1-2\delta}}, & \mbox{if }\delta \neq \frac{1}{4} \left( 1 + \frac{C_{aut}}{\sqrt{1+C_{aut}^2}}\right)\\
-\infty, & \mbox{if }\delta = \frac{1}{4} \left( 1 + \frac{C_{aut}}{\sqrt{1+C_{aut}^2}}\right)
\end{cases}\label{eq:Cmin} \\
	C_{\mathrm{max}}(\delta,C_{aut}) & := 
\begin{cases} +(1+C_{aut}^2) \frac{2 \sqrt{2}\sqrt{\delta} \sqrt{1-2\delta}}{(1-4\delta) - 2 C_{aut} \sqrt{2}\sqrt{\delta}\sqrt{1-2\delta}}, & \mbox{if }\delta \neq \frac{1}{4} \left( 1 - \frac{C_{aut}}{\sqrt{1+C_{aut}^2}}\right);\\
+ \infty, & \mbox{if }\delta = \frac{1}{4} \left( 1 - \frac{C_{aut}}{\sqrt{1+C_{aut}^2}}\right),
\end{cases}\label{eq:Cmax}
\end{align}
such that $C - C_{aut} \in \Gamma\left(\delta,C_{aut}\right)$, where
\begin{align}
	\Gamma\left(\delta,C_{aut}\right) := 
\begin{cases}
	\Big( C_{\mathrm{min}}\left(\delta,C_{aut}\right), C_{\mathrm{max}}\left(\delta,C_{aut}\right) \Big), & \mbox{if } C_{\mathrm{min}}\left(\delta,C_{aut}\right) < C_{\mathrm{max}}\left(\delta,C_{aut}\right); \\
	\Big(-\infty, C_{\mathrm{max}}\left(\delta,C_{aut}\right)\Big) \cup \Big(C_{\mathrm{min}}\left(\delta,C_{aut}\right),+\infty\Big), & \mbox{if } C_{\mathrm{max}}\left(\delta,C_{aut}\right) < C_{\mathrm{min}}\left(\delta,C_{aut}\right).
\end{cases}\label{eq:definitionSlopeBounds}
\end{align}
In particular, for $ q_{aut} = \frac{1}{\sqrt{2}}(1, -1)^T $ we have $ C_{aut} = -1 $ and, hence,
\begin{align}\label{eq:slope_estimate}
 C  = -1 + \tilde{C} \, ,  \qquad \tilde{C} \in \Gamma(\delta,-1) \, .
\end{align}
\end{lemma} 

% PROOF

\begin{proof}
For technical reasons we assume that $\|q - q_{aut}\| \leq \|q + q_{aut}\|$; if this inequality does not hold, we can scale $q \rightarrow -q$ without changing the projection matrix $Q$. Then, with
\begin{align}
 \Delta \theta : = \theta - \theta_{aut} \, ,
\end{align}
we have
\begin{align}\label{eq:slope_difference}
 C - C_{aut} = \tan(\theta)- \tan(\theta_{aut}) = \tan(\Delta \theta + \theta_{aut})- \tan(\theta_{aut}) = (1+C_{aut}^2) \, \left( \frac{\tan(\Delta \theta)}{1-C_{aut}\tan(\Delta \theta)} \right) \, .
\end{align}
From $\|Q - {Q}_{aut}\| < 4 \delta$ we know by the previous lemma that $\|q - {q}_{aut} \| < \sqrt{ 8 \delta }$ and, hence, since $ q $ and $ q_{aut} $ are unit vectors, we have
\begin{align}
 0 \leq 2(1- q^T q_{aut}) = \|q - q_{aut}\|^2  < 8 \delta \qquad \Longrightarrow \qquad 1- 4 \delta < q^T q_{aut} \, .
\end{align}
Since $ \mathrm{arccos}(z) $ is monotonically decreasing, we hence get from $ | \Delta \theta | =  \mathrm{arccos}(q^T q_{aut}) $ that
\begin{align}
 -\mathrm{arccos}(1-4\delta) < \Delta \theta < \mathrm{arccos}(1-4\delta)  \, .
\end{align}
Furthermore, since $ \frac{\tan(z)}{1-C_{aut}\tan(z)} $ is monotonically increasing in $z$, we have the claimed result by using  
\[ \tan(\pm \mathrm{arccos}(z)) = \pm \frac{\sqrt{1-z^2}}{z} \]
and some simplifications in \eqref{eq:slope_difference}.
\end{proof}

\begin{figure}
	\centering
		\begin{subfigure}[t]{0.33\textwidth}
			\centering
				\includegraphics[width = \textwidth]{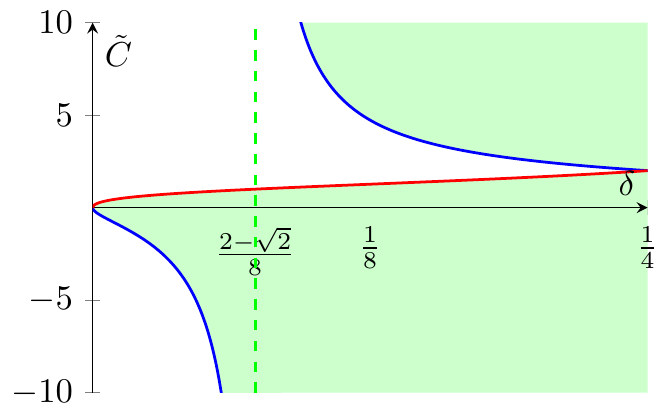}
			\caption{Plots of $C_{\mathrm{min}}$ (blue) in~\eqref{eq:Cmin} and $C_{\mathrm{max}}$ (red) in~\eqref{eq:Cmax} as functions of $\delta$ for $C_{aut} = -1$.}
		\end{subfigure}
~
		\begin{subfigure}[t]{0.2 \textwidth}
			\centering
				\includegraphics[width = \textwidth]{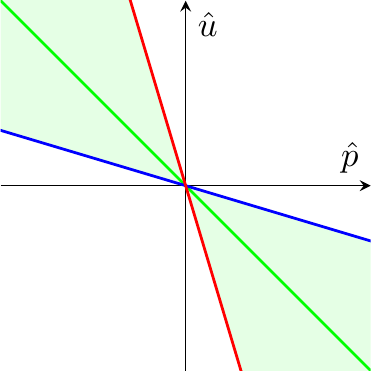}
			\caption{Bounds on slope for $C_{aut} = -1$ and some $\delta <  \frac{2-\sqrt{2}}{8}$.}
		\end{subfigure}
~
		\begin{subfigure}[t]{0.2 \textwidth}
			\centering
				\includegraphics[width = \textwidth]{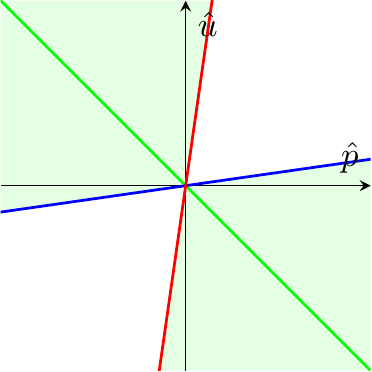}
			\caption{Bounds on slope for $C_{aut} = -1$ and some $\delta >  \frac{2-\sqrt{2}}{8}$.}
		\end{subfigure}

	\caption{Visualisation of the results of Lemma~\ref{lemma:closeness_slopes}. In (a) plots of $C_{\mathrm{min}}$ (blue) and $C_{\mathrm{max}}$ (red) are shown as function of $\delta$ for $C_{aut} = -1$, i.e. the set $\Gamma(\delta,-1)$. The green region indicates all possible values for the difference between slopes, $C-C_{aut}$. In (b) and (c) plots of the possible slopes $C$ are shown for some $\delta < \frac{2-\sqrt{2}}{8}$ (b) and $\delta > \frac{2-\sqrt{2}}{8}$ (c). The green line indicates the slope value $C_{aut} = -1$.}
	\label{fig:closenessSlopeVisualisation}
\end{figure}

%%%%%%%%%%%%%%%%%%%%%%%%%%%%%%%%%%
% SUBSECTION (Existence results) %
%%%%%%%%%%%%%%%%%%%%%%%%%%%%%%%%%%

\subsection{Existence results}\label{sec:existenceResults}

Here, we first state our main existence results in detail. Their proofs are given in section~\ref{sec:existenceResultsProofs}.

%========================%
% THEOREM (general f, g) %
%========================%

\begin{theorem}[\underline{Existence for general $ f,g $}]\label{theorem:fg_general} Let assumptions (A1), (A2), (A3) and (A5) be satisfied. Then there is a $ \mu^* $ with $ 0 < \mu^* < \frac{1}{12} $ and corresponding $ \varepsilon^* = \varepsilon^*(\mu) >0, 0 < \delta^* = \delta^*(\mu) < \frac{2-\sqrt{2}}{8} $ such that the following holds true: For any $ \varepsilon, \mu, \delta $ with
\begin{align} \label{eq:varepsilon_mu_delta_conditions_general_h}
0 < \mu < \mu^* \, , \quad 0 < \varepsilon < \varepsilon^* = \varepsilon^*(\mu) \,, \quad \delta = \sup_{x \in \mathbb{R}} \sqrt{ f(x)^2 + g(x)^2 } < \delta^* = \delta^*(\mu) \, , 
\end{align}
the stationary wave ODE \eqref{eq:klausmeier_model_ODE_first_order_fast} has (two) orbits $\left(s_p(\xi), u_p(\xi), p_p(\xi) ,v_p(\xi), q_p(\xi)\right)$, that are homoclinic to the bounded solution $\left(\xi,\frac{\hat{u}_b(\varepsilon^2\mu \xi)}{\mu}, \varepsilon \hat{u}_b'(\varepsilon^2\mu \xi), 0,0\right)$, with $\left(u_p(\xi),v_p(\xi)\right) $ to leading order given by
{\footnotesize
\begin{align}\label{eq:leading_order_orbit_general_fg}
\left[
 \begin{array}{c}
  \frac{ (\hat{u}_b(\varepsilon^2\mu\xi)-(\hat{u}_b(0)-\mu u_0)\, \hat{u}_-(\varepsilon^2\mu\xi))}{\mu} \\ 0
 \end{array}
\right]
\chi_{s}^{-}(\xi)
+
\left[
 \begin{array}{c}
  u_0 \\ \frac{3}{2 u_{0}} \, \mathrm{sech}\left(\frac{\xi}{2}\right)^2
 \end{array}
\right]
\chi_{f}(\xi)
+
\left[
 \begin{array}{c}
  \frac{(\hat{u}_b(\varepsilon^2\mu\xi)-(\hat{u}_b(0)-\mu u_0)\, \hat{u}_+(\varepsilon^2\mu\xi))}{\mu} \\ 0
 \end{array}
\right]
\chi_{s}^{+}(\xi)
\end{align}
}
with $ u_0 = u_0^- $ or $u_0 = u_0^+$ from \eqref{eq:intersection_points}, i.e.
\begin{equation}\label{eq:definitionu0}
 u_0 = \frac{\hat{u}_b(0) \pm \sqrt{\hat{u}_b(0)^2+ 12 \mu / C^s(0)}}{2\mu};
\end{equation}
$\hat{u}_b $ the bounded solution from Proposition~\ref{prop:dynamics_slow_manifold} and where the indicator functions
\begin{align}\label{eq:slow_fast_fields}
 \chi_{s}^{-}(\xi) = \chi_{\left(-\infty,-1/\sqrt{\varepsilon}\right)} \, , \qquad 
 \chi_{f}(\xi)  = \chi_{\left(-1/\sqrt{\varepsilon},1/\sqrt{\varepsilon}\right)} \, , \qquad 
 \chi_{s}^{+}(\xi) = \chi_{\left(1/\sqrt{\varepsilon}, \infty \right)}
\end{align}
distinguishes the behavior of the solution in the fast and super-slow fields. Furthermore, for $ \hat{u}_{\pm} $ we have the estimates
\begin{align*}
  | \hat{u}_{\pm}(x) | \leq C e^{-(1-2\delta) |x| } \, , \quad x \gtrless 0 \, , 
\end{align*}
for some $ C>0 $, the bounded solution $u_b$ obeys
\begin{equation*}
\sup_{x \in \mathbb{R}}
\sqrt{ (\hat{u}_b(x) - 1)^2 + \hat{u}_b'(x)^2 } \leq \frac{10\delta}{1 - 2 \delta} \, .
\end{equation*}
Finally, this homoclinic orbit gives rise to a stationary pulse solution 
\begin{align}\label{eq:front_general_h}
\left[
 \begin{array}{c}
  U_p(x,t)\\[.2cm] V_p(x,t)
 \end{array}
\right]
=
\left[
 \begin{array}{c}
  \frac{m \sqrt{m} D}{a} u \left( \frac{\sqrt{m}}{D} x \right)\\[.2cm]
  \frac{a}{D \sqrt{m}} v \left( \frac{\sqrt{m}}{D} x \right)
 \end{array}
\right]
\end{align}
for the Klausmeier model \eqref{eq:klausmeier_model} that is biasymptotic to the bounded state $\left(a \hat{u}_b\left(\frac{\sqrt{m}}{D} x\right),0\right)$.
\end{theorem}

%=======================%
% COROLLARY (f,g =0) %
%=======================%

\begin{corollary}[\underline{Existence for $f, g = 0$}]\label{cor:fg_equal_zero}
Let $ f, g = 0$, and the conditions from Theorem~\ref{theorem:fg_general} be fulfilled. Then
\begin{align*}
\hat{u}_\pm (x) = e^{\mp x} \, , \qquad \hat{u}_b \equiv  1 \, .
\end{align*}
\end{corollary}

%=====================%
% COROLLARY f,g small %
%=====================%

\begin{corollary}[\underline{Existence for small $f,g$}]\label{cor:fg_small}
Let the conditions from Theorem~\ref{theorem:fg_general} be fulfilled and $ f = \delta \tilde f $, $ g = \delta \tilde g $ where $ \tilde{f}, \tilde{g} = \mathcal{O}(1), 0 < \delta \ll 1 $ (i.e. $\sup_{x \in \mathbb{R}} \sqrt{\tilde{f}(x)^2 + \tilde{g}(x)^2} = 1$. Then 
{\footnotesize
\begin{align*}
\hat{u}_+ (x) &= e^{-x} +\frac{\delta}{2} \left[ - e^{x} \int_{x}^{\infty} (\tilde f(z) - \tilde g(z))e^{-2z} dz + e^{-x} \left(\int_0^{\infty} (\tilde f(z) - \tilde g(z))e^{-2z} dz + \int_0^{x} (\tilde f(z) - \tilde g(z)) ds \right) \right] + h.o.t. \, ,\\
\hat{u}_- (x) &= e^{x} +\frac{\delta}{2} \left[ e^{-x} \int_{-\infty}^{x} (\tilde f(z) + \tilde g(z))e^{-2z} ds - e^{-x} \left(\int_{-\infty}^0 (\tilde f(z) + \tilde g(z))e^{-2z} dz + \int_0^{x} (\tilde f(z) + \tilde g(z)) dz \right) \right] + h.o.t. \, ,\\
\hat{u}_b(x) &=  1 + \frac{\delta}{2} \left[e^{x}  \int_{x}^{\infty} \tilde g(z) e^{-z} \, dz + e^{-x}  \int_{-\infty}^{x} \tilde g(z) e^{z} \, dz \right]+ h.o.t. \,.
\end{align*}}
Moreover, $u_0$ as in~\eqref{eq:definitionu0} can be expressed in terms of $\delta$ as
\begin{equation}
	u_0 = u_{00} + \delta u_{01} + h.o.t.,
\end{equation}
where $u_{00}$ corresponds to the $u_0$-value for the autonomous case, i.e. $u_{00}$ is given by~\eqref{eq:intersection_points_aut}.

\end{corollary}

%==========================================%
% COROLLARY (h(x) = -2 \ln \cosh(\beta x)) %
%==========================================%

\begin{corollary}[\underline{Existence for $h(x) = -2 \ln \cosh(\beta x)$}]\label{cor:h_example}
Let $ h(x) = -2 \ln \cosh(\beta x), \beta > 0, f = h', g = h'' $, and the conditions from Theorem~\ref{theorem:fg_general} be fulfilled. Then
\begin{align*}
\hat{u}_\pm (x) &= e^{\mp\sqrt{1+\beta^2}x} \cosh(\beta x) \, ,\\
\hat{u}_b(x) &=  \frac{u_-(x)}{2\sqrt{1+\beta^2}} \int_{x}^\infty e^{-\sqrt{1+\beta^2}z} \sech(\beta z)\ dz + \frac{u_+(x)}{2\sqrt{1+\beta^2}} \int_{-\infty}^{x} e^{\sqrt{1+\beta^2}z} \sech(\beta z)\ dz \, .
\end{align*}
\end{corollary}

\begin{remark}
	Pulses solutions as in Corollary~\ref{cor:h_example} exist for any $\beta > 0$ without the need of the general assumption on $\delta$ as in Theorem~\ref{theorem:fg_general}; since the flow on $\mathcal{M}$ can be solved explicitly for these functions $f$ and $g$, no condition on $\delta$ is needed.
\end{remark}

\begin{remark}
	Since the flow on $\mathcal{M}$ can be solved explicitly for the functions $f$ and $g$ as in Corollary~\ref{cor:h_example}, it is also possible to prove existence of symmetric, stationary $2$-pulse solutions (and, in fact, any symmetric, stationary $N$-pulse solution). Note that normally, for $f,g \equiv 0$, these do not exist, since pulses in~\eqref{eq:klausmeier_model} repel each other~\cite{dek1siam,BD18}; this repulsive force can only be overcome by driving forces due to the spatially varying functions $f$ and $g$. We come back to these multi-pulse solutions in section~\ref{sec:multiPulses}.
\end{remark}

\subsection{Proof of existence results}\label{sec:existenceResultsProofs}

The proofs of the existence results in section~\ref{sec:existenceResults} follow from the theory developed in the preceding sections. The heart of these proofs is formed by Proposition~\ref{prop:persistence} and the bounds on the bounded solution $u_b$ and the slopes $C^{s/u}$ as found in Proposition~\ref{prop:dynamics_slow_manifold}. Ultimately, it boils down to taking $\delta$ small enough such that an intersection between $l^s(0)$ and $T_o(0)$ is guaranteed. A sketch of this idea is given in Figure~\ref{fig:existenceProofSketches}; the rest of this section is devoted to the rigorous proof of the existence theorem and the corollories in section~\ref{sec:existenceResults}.

% %===============================%
% % PROOF FOR EXISTENCE THEOREM
% %===============================%

\begin{proof}[Proof of Theorem~\ref{theorem:fg_general}]
	Existence of the homoclinic orbits is established by Proposition~\ref{prop:persistence} if the conditions in Proposition~\ref{prop:persistence}(4) are satisfied. Since $u_b(0) = \hat{u}_b(0)/\mu$, these hold if and only if the following three bounds hold true:
\begin{itemize}
	\item[(i)] $\hat{u}_b(0) > 0$;
	\item[(ii)] $C^s(0) < 0$;
	\item[(iii)] $\hat{u}_b(0)^2 + 12 \mu / C^s(0) > 0$.
\end{itemize}
By Proposition~\ref{prop:roughness_closeness_general} and Lemma~\ref{lemma:exp_dich_const_roughness}, we have
\begin{equation}
	\hat{u}_b(0) > \frac{1 - 12 \delta}{1-2\delta},
\end{equation}
and by Lemma~\ref{lemma:closeness_slopes} we have
\begin{equation}
	C^s(0) = -1 + \tilde{C}, \qquad \tilde{C} \in \Gamma(\delta,-1),
\end{equation}
where $\Gamma$ is as in~\eqref{eq:definitionSlopeBounds}. Using these, bound (i) is satisfied when $\delta < \frac{1}{12}$ and bound (ii) when $\delta < \frac{2 - \sqrt{2}}{8}$. Since the bound (iii) holds true when $\delta = 0$ and $\mu < \frac{1}{12}$, continuity of mentioned bounds on $\hat{u}_b(0)$ and $C^s(0)$ guarantees the existence of the critical value $0 < \delta^*(\mu) < \frac{2 - \sqrt{2}}{8}$.
%The given bounds follow from Proposition~\ref{prop:roughness_closeness_general} and Lemma~\ref{lemma:exp_dich_const_roughness}.
\end{proof}

\begin{proof}[Proof of Corollary~\ref{cor:fg_equal_zero}]
	This follows immediately from solving~\eqref{eq:klausmeier_model_ODE_first_order_slow_on_M_rescaled} with $f,g \equiv 0$, and is also carried out in more detail in section~\ref{sec:existence_f_g_zero}.
\end{proof}

\begin{proof}[Proof of Corollary~\ref{cor:fg_small}]
	The super-slow system on $\mathcal{M}$ in~\eqref{eq:klausmeier_model_ODE_first_order_slow_on_M_rescaled} can be solved using a regular expansion in $0 < \delta \ll 1$. By requiring that $\lim_{x \rightarrow \infty} \hat{u}_+(x)$ and $\lim_{x \rightarrow -\infty} \hat{u}_-(x)$ exist, the results follow by a straightforward calculation.
\end{proof}

\begin{proof}[Proof of Corollary~\ref{cor:h_example}]
	One can easily verify that $\hat{u}_\pm$ solve~\eqref{eq:klausmeier_model_ODE_first_order_slow_on_M_rescaled}, and that $\lim_{x \rightarrow \pm \infty} \hat{u}_\pm(x) = 0$. The bounded solution $\hat{u}_b$ follows from a standard variation of constants method.
\end{proof}

%%%%%%%%%%%%%%%%%%%%%%%%%%%%%%%%%%%%%%%%
%% SECTION: Linear Stability Analysis %%
%%%%%%%%%%%%%%%%%%%%%%%%%%%%%%%%%%%%%%%%

\section{Linear stability analysis}\label{sec:linstability}

In the previous section, we proved the existence of stationary $1$-pulse solutions to~\eqref{eq:klausmeier_model}. In this section we study the linear stability of these solutions. For $ (U_p, V_p) $ a pulse solution from Theorem~\ref{theorem:fg_general} we define the linear operator  
\begin{equation}\label{eq:linearization_operator}
	\mathcal{L}\left(\begin{array}{c} \bar{U} \\ \bar{V} \end{array}\right) =
	\left(
	\begin{array}{c}
	\partial_x^2 \bar{U} + f(x) \partial_x \bar{U} + g(x) \bar{U} - \bar{U} - V_p^2 \bar{U} - 2 U_p V_p \bar{V} \\
	D^2 \partial_x^2 \bar{V} - m \bar{V} + V_p^2 \bar{U} + 2 U_p V_p \bar{V}.
	\end{array}
	\right) \, ,
\end{equation}
with $ \mathcal{L}: H^2(\mathbb{R}) \times H^2(\mathbb{R}) \subset L^2(\mathbb{R}) \times L^2(\mathbb{R}) \rightarrow L^2(\mathbb{R}) \times L^2(\mathbb{R}) $ and its spectrum by $ \Sigma(\mathcal{L}) $, where we distinguish between the point spectrum $ \Sigma_\mathrm{pt}(\mathcal{L}) $ and the essential spectrum $  \Sigma_\mathrm{ess}(\mathcal{L}) =  \Sigma(\mathcal{L})\setminus \Sigma_\mathrm{pt}(\mathcal{L}) $ -- we denote the elements of $\Sigma_\mathrm{ess}(\mathcal{L})$ by $\underline{\lambda}$. As customary, we say that $ (U_p, V_p) $ is linearly stable if there is no spectrum in the right half plane. In order to keep the exposition at reasonable length, we will concentrate here on characterizing parameter regimes where the only instability that can occur is through the (translational) zero eigenvalue which starts moving due to the introduction of spatially varying $ f $ and/or $ g $. In particular, there are no essential instabilities:

%============================%
% LEMMA (essential spectrum) %
%============================%

\begin{lemma}[\underline{Essential spectrum}]\label{lemma:essential_spectrum}
Let the conditions of Theorem~\ref{theorem:fg_general} and assumption (A4) be fulfilled, and let $(U_p,V_p)$ be a pulse solution to \eqref{eq:klausmeier_model} as in Theorem~\ref{theorem:fg_general}. Then the essential spectrum of $ \mathcal{L} $ from \eqref{eq:linearization_operator} is
\begin{equation}
  \Sigma_\mathrm{ess}(\mathcal{L}) = (-\infty, \max\{-m,-1\}] \, ,
\end{equation}
and, hence, lies in the left half-plane.
\end{lemma}

% PROOF

\begin{proof}
The limiting operator of $ \mathcal{L} $ at $ x \rightarrow \pm \infty $ is $ \mathcal{L}_{\infty}:= \mathrm{diag}[\partial_x^2 - 1, D^2 \partial_x^2 - m] $ (note that we thus explicitly use assumption (A4)). Therefore, we have that the boundaries of the essential spectrum are $ \underline{\lambda}_1(k) = -(k^2 +1)$, $\underline{\lambda}_2(k) = -(D^2k^2 +m)$, $k \in \mathbb{R} $, which immediately gives the claimed result. 
\end{proof}

The assumptions on $ f, g $ allow (again through the use of exponential dichotomies) the derivation of	 bounds on the location of the point spectrum, which, under the assumption that $ f,g $ are chosen `small', can be further refined to track the one small eigenvalue that can possibly lead to bifurcations. The proof of the following statements will be the subject of the next sections.

%==========================%
% THEOREM (Point spectrum) %
%==========================%

\begin{theorem}[\underline{Point spectrum}]\label{theorem:point_spectrum}
Let the conditions of Theorem~\ref{theorem:fg_general} and assumption (A4) be fulfilled, and let $(U_p,V_p)$ be a pulse solution to \eqref{eq:klausmeier_model} with $u_0 = u_0^-$ as in \eqref{eq:definitionu0}. Then there exist constants $m_c, \mu^*, \nu^* > 0$ such that if either (i) $m < m_c$ and $\mu < \mu^*$ or (ii) $m > m_c$ and $\mu \sqrt{m} < \nu^*$, then there exists a $\delta_c > 0$ such that if $0 \leq \delta < \delta_c$ precisely one eigenvalue $\underline{\lambda}_0$ is $\mathcal{O}(\varepsilon)$-close to $0$ and all other eigenvalues of $\mathcal{L}$ lie in the left-half plane.
\end{theorem}

\begin{proof} The statement is demonstrated in section~\ref{sec:point_spectrum_general} by combining the setup of an Evans function and the theory of exponential dichotomies.
\end{proof}

\begin{remark}
Note that Theorem~\ref{theorem:point_spectrum} only holds for pulse solutions with $u_0 = u_0^-$; pulse solutions with $u_0 = u_0^+$ are always unstable. See also Remark~\ref{remark:varTer_positiveu0}.
\end{remark}

\begin{remark}
The constants $m_c$, $\mu^*$ and $\nu^*$ in Theorem~\ref{theorem:point_spectrum} can be computed explicitly (see Lemma~\ref{lemma:rootsOft22ComputedBetter}).
\end{remark}

\begin{theorem}[\underline{Small eigenvalue close to $\underline{\lambda} = 0$ for small $f$, $g$}] \label{theorem:point_spectrum_small} Assuming that $f = \delta \tilde{f}$, $g = \delta \tilde{g}$ with $0 < \delta \ll 1$, $\tilde{f}, \tilde{g} = \mathcal{O}(1)$ (i.e. $\sup_{x \in \mathbb{R}} \sqrt{\tilde{f}(x)^2+\tilde{g}(x)^2} = 1$), there exists a constant $\tau^* > 0$ such that if $\tau:= \varepsilon^4 \mu m < \tau^*$ the small eigenvalue $\underline{\lambda}_0$ close to $\underline{\lambda} = 0$ is located, to leading order, at
\begin{equation}\label{eq:smallEigenvalue}
  \underline{\lambda}_0 = \frac{2 \tau \delta}{u_0 - \tau (1 - \mu u_0)} \int_0^{+\infty} e^{-2x} \left(\tilde{f}'(x)(1-\mu u_0) + \tilde{g}'(x) [e^{x}+\mu u_0 - 1] \right) \, dx,
\end{equation}
where $u_0$ is as in~\eqref{eq:definitionu0} and Corollary~\ref{cor:fg_small}.
\end{theorem}

\begin{proof} This statement is derived in section~\ref{sec:smallEigenvalue} by employing a regular expansion in $\delta$.
\end{proof}

\begin{corollary}\label{cor:small_eigenvalue_double_limit}
	Let the conditions of Theorem~\ref{theorem:point_spectrum_small} be fulfilled. Then, in the double asymptotic limit $\mu \ll 1$ and $\tau := \varepsilon^4 \mu m \ll 1$ the leading order expression for $\underline{\lambda}_0$ becomes
\begin{equation}
	\underline{\lambda}_0 = \frac{2}{3} \tau \int_0^\infty e^{-2x} \left( \tilde{f}'(x) + \tilde{g}'(x)[e^x-1]\right)\ dx.
\end{equation}
\end{corollary}

%========%
% Remark %
%========%

\begin{remark}
When the term $\tau = \varepsilon^4 \mu m = \frac{a^2 D}{m \sqrt{m}}$ in \eqref{eq:smallEigenvalue} becomes too large (larger than $\tau^*$), the pulse becomes unstable due to a traveling wave bifurcation/drift instability~\cite{chen2009oscillatory, DEK01}.
\end{remark}

%%%%%%%%%%%%%%%%%%%%%%%%%%%%%%%%%%%%%%%%%%%%%%%%%%%%%%%%%%%%%%%%%%%%%
%% SUBSECTION (Qualitative description of point spectrum location) %%
%%%%%%%%%%%%%%%%%%%%%%%%%%%%%%%%%%%%%%%%%%%%%%%%%%%%%%%%%%%%%%%%%%%%%

\subsection{Qualitative description of the point spectrum location (Proof of Theorem~\ref{theorem:point_spectrum})}\label{sec:point_spectrum_general}
This section is devoted to finding the point spectrum of the operator $\mathcal{L}$. For that, we use a decomposition method for the Evans function, first developed in~\cite{AGJ90, D01}, which is supplemented by the theory of exponential dichotomies to treat the varying coefficients in~\eqref{eq:klausmeier_model}. As before, the following computations will again heavily rely on the singularly perturbed structure. Therefore, we introduce for the eigenvalue problem $ (\mathcal{L} - \underline{\lambda} I) (\bar{U}, \bar{V})^T = 0 $, that is,
\begin{align}\label{eq:evp_original}
\left\{
\begin{array}{rcl}
 \underline{\lambda} \bar{U} & = & \frac{d^2}{dx^2} \bar{U} + f(x) \frac{d}{dx} \bar{U} + g(x) \bar{U} - \bar{U} - V_p^2 \bar{U} - 2 U_p V_P \bar{V} \, ,\\[.2cm]
  \frac{1}{m} \underline{\lambda} \bar{V} & = & \frac{D^2}{m} \frac{d^2}{dx^2} \bar{V} - \bar{V} + \frac{1}{m} V_p^2 \bar{U} + \frac{2}{m} U_p V_p \bar{V} \, ,
 \end{array}
 \right.
\end{align}
and the scalings (analogous to \eqref{eq:scaling_1} and \eqref{eq:epsilon_mu})  
\begin{equation}
 \xi = \frac{D}{\sqrt{m}} = \varepsilon^2 \mu x \, , \quad  \bar{U} = m \varepsilon \mu \bar{u} \, , \quad U_p = m \varepsilon \mu  u_p \, , \quad \bar{V} = \frac{1}{\varepsilon \mu} \bar{v} \, , \quad V_p = \frac{1}{\varepsilon \mu} v_p,
\end{equation}
to get the fast eigenvalue problem
\begin{equation}\label{eq:evp_original_scaled_}
	\left\{
	\begin{array}{rcl}
	\varepsilon^4 \mu^2 \underline{\lambda}  \bar{u}
& =
& \ddot{\bar{u}} - \varepsilon^2 [ 2 u_p v_p \bar{v} + v_p^2 \bar{u} ] - \varepsilon^4 \mu^2 \bar{u} + \varepsilon^2 \mu f(\varepsilon^2 \mu \xi) \dot{\bar{u}} + \varepsilon^4 \mu^2 g(\varepsilon^2 \mu \xi) \bar{u} \, , \\
	\frac{1}{m} \underline{\lambda}  \bar{v} 
&=
& \ddot{\bar{v}} - \bar{v} + [ 2 u_p v_p \bar{v} + v_p^2 \bar{u} ]  \, ,
	\end{array}\right.
\end{equation}
which suggests (just as in \cite{BD18, chen2009oscillatory, DEK01}) the introduction of the scaled eigenvalue parameter 
\begin{align}\label{eq:scaling_eigenvalue}
 \underline{\lambda}  = m \lambda \, ,
\end{align}
so, finally,
\begin{equation}
	\left\{
	\begin{array}{rcl}
	\varepsilon^4 \mu^2 m \lambda \bar{u}
& =
& \ddot{\bar{u}} - \varepsilon^2 [ 2 u_p v_p \bar{v} + v_p^2 \bar{u} ] - \varepsilon^4 \mu^2 \bar{u} + \varepsilon^2 \mu f(\varepsilon^2 \mu \xi) \dot{\bar{u}} + \varepsilon^4 \mu^2 g(\varepsilon^2 \mu \xi) \bar{u}  \, , \\
	\lambda \bar{v} 
&=
& \ddot{\bar{v}} - \bar{v} + [ 2 u_p v_p \bar{v} + v_p^2 \bar{u} ]  \, .
	\end{array}\right.
	\label{eq:eigenvalueProblem}
\end{equation}
It is convenient to introduce $\phi := \left( \bar{u}, \dot{\bar{u}} / (\varepsilon^2 \mu), \bar{v}, \dot{\bar{v}}\right)$ and to write the above ODEs as the system of first order ODEs
\begin{equation}
	\dot{\phi} = A(\xi; \lambda, \varepsilon, \mu, m) \phi ,\label{eq:eigenvalueProblemMatrixForm}
\end{equation}
where
\begin{equation}
	A(\xi; \lambda, \varepsilon, \mu, m) =
\left( \begin{array}{cccc}
	0 & \varepsilon^2 \mu & 0 & 0 \\
	v_p^2 / \mu + \varepsilon^2 \mu \left[1+m \lambda - g(\varepsilon^2 \mu \xi) \right]& - \varepsilon^2 \mu f(\varepsilon^2 \mu \xi) &2 u_p v_p / \mu & 0 \\
	0 & 0 & 0 & 1 \\
	- v_p^2 & 0 & 1 + \lambda - 2 u_p v_p & 0
\end{array}\right).
\end{equation}
From the existence analysis in section~\ref{sec:existence}, we have seen that the real line $\mathbb{R}$ can be split in one fast region, $I_f$, near the pulse location and two super slow fields $I_s^\pm$ to both sides of the fast field:
\begin{equation*}
	I_s^-  := \left(-\infty,-\frac{1}{\sqrt{\varepsilon}}\right),\, \quad
	I_f  := \left[ - \frac{1}{\sqrt{\varepsilon}}, \frac{1}{\sqrt{\varepsilon}} \right],\, \quad
	I_s^+  := \left( \frac{1}{\sqrt{\varepsilon}}, \infty \right).
\end{equation*}
Since we know that $ v_p $ vanished to leading order in the slow fields, we have in those regions the system matrix
\begin{equation}
	A_s(\xi;\lambda,\varepsilon,\mu,m) := 
\left( \begin{array}{cccc}
	0 & \varepsilon^2 \mu & 0 & 0 \\
	\varepsilon^2 \mu \left[1+m \lambda - g(\varepsilon^2 \mu \xi) \right]& - \varepsilon^2 \mu f(\varepsilon^2 \mu \xi) &0 & 0 \\
	0 & 0 & 0 & 1 \\
	0 & 0 & 1 + \lambda & 0
\end{array}\right) \, ,
\end{equation}
that is, the dynamics for slow and fast variables are decoupled.
Any value $ \lambda \in \mathbb{C} $ for which this system of ODEs has a non-trivial solution in $ L^2(\mathbb{R}) \times L^2(\mathbb{R}) $ corresponds to an eigenvalue $\underline{\lambda} = m \lambda$ of $ \mathcal{L} $. A mechanism (that is by now standard) for detecting eigenvalues is the construction of an Evans function, whose roots coincide with the eigenvalues of $ \mathcal{L} $. Although the Evans function can also be extended into the essential spectrum, we do not need this in the present work and rather restrict $\lambda$ to
\begin{equation}\label{eq:definitionCe}
\mathcal{C}_e := \mathbb{C} \setminus \left\{ \lambda \in \mathbb{R} : \lambda \leq \max\{-1,-1/m\} \right\} = \left\{ \lambda = \frac{\underline{\lambda}}{m} : \underline{\lambda} \notin \Sigma_\mathrm{ess}(\mathcal{L})\right\} ,
\end{equation}
on which the Evans function is analytic.

%%%%%%%%%%%%%%%%%%%%%%%%%%%%%%%%%%%%%%%%%%%%%%%%%%
%% SUBSUBSECTION (Evans function construction ) %%
%%%%%%%%%%%%%%%%%%%%%%%%%%%%%%%%%%%%%%%%%%%%%%%%%%

\subsubsection{Evans function construction}\label{sec:evans_function_construction}
By (conditions and results of) Theorem~\ref{theorem:fg_general} and assumption (A4), we know that the limiting matrix for  $ |\xi| \rightarrow \infty $ is given by
\begin{equation}
	A_\infty(\lambda,\varepsilon,\mu,m) :=
\left( \begin{array}{cccc}
	0 & \varepsilon^2 \mu & 0 & 0 \\
	\varepsilon^2 \mu \left[1+m \lambda \right]& 0 &0 & 0 \\
	0 & 0 & 0 & 1 \\
	0 & 0 & 1 + \lambda & 0
\end{array}\right) \, .
\end{equation}
Its eigenvalues $\Lambda_{1,2,3,4}$ and eigenvectors $E_{1,2,3,4}$ are
\begin{equation}
	\begin{array}{ll}
\Lambda_{1,4}(\lambda) = \pm \sqrt{1+\lambda},
& \Lambda_{2,3}(\lambda) = \pm \varepsilon^2 \mu \sqrt{1+ m \lambda} \\
E_{1,4}(\lambda) = \left(0,0,1,\Lambda_{1,4}\right)^T,
& E_{2,3}(\lambda) = \left(1, \pm \sqrt{1 + m \lambda},0,0\right)^T.
	\end{array}
\end{equation}
where $\mbox{Re}\left( \Lambda_1(\lambda) \right) < \mbox{Re}\left( \Lambda_2(\lambda) \right)< 0 < \mbox{Re}\left( \Lambda_3(\lambda) \right)< \mbox{Re}\left( \Lambda_4(\lambda) \right)$ for $\lambda \in \mathcal{C}_e$.

The system $\dot{\phi}_\infty = A_{\infty}(\lambda,\varepsilon,\mu,m) \phi_\infty$ admits exponential dichotomies on $\mathcal{C}_e$. Since $A_\infty$ is exponentially close to $A$ for large $|\xi|$, the stable and unstable subspaces of $\dot{\phi} = A(\xi;\lambda,\varepsilon,\mu,m)\phi$ and $\dot{\phi}_\infty = A_\infty(\lambda,\varepsilon,\mu,m) \phi_\infty$ are similar when $|\xi| \rightarrow \infty$. In particular, for all $\lambda \in \mathcal{C}_e$ there is a two-dimensional family of solutions, $\Phi_\infty^-(\lambda)$, to $\dot{\phi}_\infty = A_\infty(\lambda,\varepsilon,\mu,m) \phi_\infty$ such that $\lim_{\xi \rightarrow -\infty} \phi_\infty^-(\xi) = 0$ for all $\phi_\infty^- \in \Phi_\infty^-(\lambda)$, and a two-dimensional family of solutions, $\Phi_\infty^+(\lambda)$, to $\dot{\phi}_\infty = A_\infty(\lambda,\varepsilon,\mu,m) \phi_\infty$ such that $\lim_{\xi \rightarrow \infty} \phi_\infty^+(\xi) = 0$ for all $\phi_\infty^+ \in \Phi_\infty^+(\lambda)$, which implies that the system $\dot{\phi} = A(\xi;\lambda,\varepsilon,\mu,m)\phi$ also possesses two two-dimensional families of solutions, $\Phi^-(\lambda)$ and $\Phi^+(\lambda)$ with the same properties. 

For the system $\dot{\phi} = A(\xi;\lambda,\varepsilon,\mu,m) \phi$, however, it is possible that the intersection $\Phi^+(\lambda) \cap \Phi^-(\lambda)$ is nonempty. The values $\lambda \in \mathcal{C}_e$ for which this happens correspond to $\underline{\lambda} = m \lambda$ in the point spectrum $\Sigma_\mathrm{pt}$. To find these, we use a Evans function~\cite{AGJ90,D01}, which is defined as
\begin{equation}\label{eq:definitionEvansFunction}
	\mathcal{D}(\lambda) = \det\left[ \phi_1(0;\lambda), \phi_2(0;\lambda), \phi_3(0;\lambda), \phi_4(0;\lambda)\right] \, ,
\end{equation}
where $\{\phi_1(\cdot;\lambda),\phi_2(\cdot;\lambda)\}$ spans the space $\Phi^-(\lambda)$ and $\{\phi_3(\cdot;\lambda),\phi_4(\cdot;\lambda)\}$ spans the space $\Phi^+(\lambda)$. For notational clarity we have suppressed the dependence on the other parameters. Essentially, the Evans function $\mathcal{D}(\lambda)$ measures the linear independence of the solution functions $\phi_{1,\ldots,4}$. Therefore, zeros of $\mathcal{D}(\lambda)$ correspond to values of $\lambda$ for which $\Phi^+(\lambda) \cap \Phi^-(\lambda) \neq \varnothing$, and thus to eigenvalues in the point spectrum~\cite{AGJ90}.

In~\eqref{eq:definitionEvansFunction} the solutions $\phi_{1,\ldots,4}$ are not uniquely defined, and any choice leads to the same eigenvalues. However, for singularly perturbed partial differential equations a specific choice enables the use of the scale separation in these equations, which in turn makes it possible to determine the eigenvalues.

%%%%%%%%%%%%%%%%%
%% Evans Function Decomposition
%%%%%%%%%%%%%%%%%

\begin{lemma}\label{lemma:EvansFunctionDecomposition}
	Let the conditions of Theorem~\ref{theorem:fg_general} be fulfilled and let $(U_p,V_p)$ be a pulse solution to \eqref{eq:klausmeier_model} as in Theorem~\ref{theorem:fg_general}. Then all eigenvalues $\lambda \in \Sigma_\mathrm{pt}$ associated to~\eqref{eq:eigenvalueProblem} are roots of the Evans function
\begin{equation}
	\mathcal{D}(\lambda) = t_{11}(\lambda) t_{22}(\lambda) (1 + m \lambda) (1 + \lambda) \exp\left(\int_0^\infty f(x)\ dx\right),
\end{equation}
where $t_{11}$ and $t_{22}$ are analytic (transmission) functions of $\lambda$, defined by
\begin{align}
	\lim_{\xi \rightarrow \infty} \phi_1(\xi;\lambda) e^{-\Lambda_1(\lambda) \xi} &= t_{11} E_1; \\
	\lim_{\xi \rightarrow \infty} \phi_2(\xi;\lambda) e^{-\Lambda_2(\lambda) \xi} & = t_{22} E_2,
\end{align}
where $\phi_1$ is the (unique) solution to~\eqref{eq:eigenvalueProblemMatrixForm} for which
\begin{align}
	\lim_{t \rightarrow -\infty} \phi_1(\xi;\lambda) e^{-\Lambda_1(\lambda)\xi} &= E_1;
\intertext{and $\phi_2$ is the (unique) solution to~\eqref{eq:eigenvalueProblemMatrixForm} (if $t_{11}(\lambda) \neq 0$) for which}
	\lim_{t \rightarrow -\infty} \phi_2(\xi;\lambda)e^{-\Lambda_2(\lambda)\xi} &= E_2;\\
	\lim_{t \rightarrow \infty} \phi_2(\xi;\lambda) e^{\Lambda_1(\lambda)\xi} & = 0
\end{align}
\end{lemma}
\begin{proof}
The proof is heavily based on~\cite[Section 3.2]{D01}. Therefore, we present here only an outline of the proof and refer the interested reader to~\cite{D01} for more details.

The heart of the proof is based on choosing $\phi_{1,\ldots,4}$ in such way that the scale separation of~\eqref{eq:klausmeier_model} can be exploited. Because $A$ and $A_{\infty}$ are exponentially close when $\xi \rightarrow -\infty$, there is a unique solution $\phi_1$ such that $\phi_1$ closely follows $E_1(\lambda) e^{\Lambda_1(\lambda)\xi}$ as $\xi \rightarrow - \infty$. More precisely, we define $\phi_1$ uniquely such that $\lim_{\xi \rightarrow -\infty} \phi_1(\xi;\lambda) e^{-\Lambda_1(\lambda)\xi} = E_1(\lambda)$. For $\xi \rightarrow \infty$, we do not know the precise form of $\phi_1$, but we do know that, asymptotically, it is a combination of the eigenfunctions of the system $\dot{\phi}_\infty = A_\infty \phi_\infty$. That is, $\phi_1(\xi;\lambda) \rightarrow t_{11}(\lambda) E_1 e^{\Lambda_1(\lambda)\xi} + t_{12}(\lambda) E_2 e^{\Lambda_2(\lambda)\xi} + t_{13}(\lambda) E_3 e^{\Lambda_3(\lambda)\xi} + t_{14}(\lambda) e^{\Lambda_4(\lambda)\xi}$ as $\xi \rightarrow \infty$, where $t_{11},\ldots,t_{14}$ are analytic transmission functions.

Next, $\phi_2$ must be chosen such that $\{\phi_1(\cdot,\lambda),\phi_2(\cdot,\lambda)\}$ spans $\Phi^-(\lambda)$. As this does not determine $\phi_2$ uniquely, we may, additionally, require that $\phi_2$ grows, at most, as $E_2(\lambda) e^{\Lambda_2(\lambda)\xi}$ for $\xi \rightarrow \infty$. More precisely, we define $\phi_2$ uniquely such that $\lim_{\xi \rightarrow -\infty} \phi_2(\xi;\lambda) e^{-\Lambda_2(\lambda)\xi} = E_2$ and $\lim_{\xi \rightarrow +\infty} \phi_2(\xi;\lambda) e^{-\Lambda_1(\lambda)\xi} = 0$ (note that this construction is based on insight in $t_{11}$ -- that may not be $0$ -- that is obtained by the `elephant trunk procedure', see~\cite{D01, gardner1991stability} and Remark~\ref{remark:stability_fgLimits}). For $\xi \rightarrow \infty$, $\phi_2$ is then asymptotically given by $\phi_2(\xi;\lambda) \rightarrow t_{22}(\lambda) E_2(\lambda) e^{\Lambda_2(\lambda)\xi} + t_{23}(\lambda) E_3(\lambda) e^{\Lambda_3(\lambda)\xi} + t_{24}(\lambda) e^{\Lambda_4(\lambda)\xi}$ as $\xi \rightarrow \infty$, where $t_{21}$, $t_{23}$, $t_{24}$ are analytical transmission functions.

In a similar vein the solutions $\phi_3$ and $\phi_4$ can be defined such that $\lim_{\xi \rightarrow \infty} \phi_4(\xi;\lambda) e^{-\Lambda_4(\lambda)} = E_4(\lambda)$ and $\lim_{\xi \rightarrow \infty} \phi_3(\xi;\lambda) e^{-\Lambda_3(\lambda)} = E_3(\lambda)$.

Then, using that $\sum_{j=1}^4 \Lambda_j(\lambda) = 0$ and by Liouville's formula, the Evans function~\eqref{eq:definitionEvansFunction} can be rewritten:
\begin{align*}
\mathcal{D}(\lambda)
& = \lim_{\xi \rightarrow \infty} \det\left[ \phi_1(\xi;\lambda), \phi_2(\xi;\lambda), \phi_3(\xi;\lambda), \phi_4(\xi;\lambda)\right] \exp\left(-\int_0^\xi \mathrm{Tr} A(z)\ dz\right) \\
& = \lim_{\xi \rightarrow \infty} \det\left[ \phi_1(\xi;\lambda) e^{-\Lambda_1(\lambda)\xi}, \phi_2(\xi;\lambda)e^{-\Lambda_2(\lambda)\xi}, \phi_3(\xi;\lambda)e^{-\Lambda_3(\lambda)\xi}, \phi_4(\xi;\lambda)e^{-\Lambda_4(\lambda)\xi}\right] \exp\left(-\int_0^\xi \mathrm{Tr} A(z)\ dz\right) \\
& = \det\left[ t_{11}(\lambda)E_1(\lambda), t_{22}(\lambda)E_2(\lambda),E_3(\lambda),E_4(\lambda)\right] \exp\left(\int_0^\infty f(x)\ dx\right) \\
& = t_{11}(\lambda) t_{22}(\lambda) (1+ m \lambda)(1+\lambda) \exp\left(\int_0^\infty f(x)\ dx\right).
\end{align*}
\end{proof}

The roots $\lambda \in \mathcal{C}_e$ of $\mathcal{D}(\lambda)$ thus correspond to the roots of $t_{11}(\lambda) t_{22}(\lambda)$. The next goal, therefore, is to determine the roots of these transmission functions.

%%%%%%%%%%%%%%%%%%%%%%%%%%%%%%%%%%%%%%%%%%%%%%%%%%%%%%
%% SUBSECTION (Fast transmission function $t_{11}$) %%
%%%%%%%%%%%%%%%%%%%%%%%%%%%%%%%%%%%%%%%%%%%%%%%%%%%%%%

\subsubsection{Fast transmission function $t_{11}$}

The transmission function $t_{11}$ is closely related to the linearization around the pulse in the fast field,
\begin{equation}\label{eq:eigenvaluesFastReduced}
	(\mathcal{L}^\mathrm{r} - \lambda) v= 0, \, \quad \mathcal{L}^\mathrm{r} v := \partial_\xi^2 v - [1 - 3 \sech(\xi/2)^2] v.
\end{equation}
The eigenvalues of $\mathcal{L}^\mathrm{r}$ are well-known to be $\lambda_0^\mathrm{r} = 5/4$, $\lambda_1^\mathrm{r} = 0$ and $\lambda_2^\mathrm{r} = - 3 / 4$. By a standard winding number argument, it follows that roots of $t_{11}$ lie $\mathcal{O}(\varepsilon)$-close to these eigenvalues $\lambda_0^\mathrm{r}$, $\lambda_1^\mathrm{r}$ and $\lambda_2^\mathrm{r}$.

\begin{lemma}[Properties of $t_{11}$]\label{lemma:propertiesOft11}
	Let the conditions of Proposition~\ref{lemma:EvansFunctionDecomposition} be fulfilled. The roots of $t_{11}$ lie $\mathcal{O}(\varepsilon)$ close to the eigenvalues (counting multiplicity) of $\mathcal{L}^\mathrm{r}$, i.e. close to $\lambda_0^\mathrm{r} = 5/4$, $\lambda_1^\mathrm{r} = 0$ and $\lambda_2^\mathrm{r} = -3/4$.
\end{lemma}
\begin{proof}
See~\cite[Lemma 4.1]{D01}.
\end{proof}

Although $t_{11}$ has a root (with multiplicity $1$) close to $\lambda^r_0 = 5/4$, this does not mean that $\mathcal{D}(\lambda)$ has a root for the same value of $\lambda$, since -- as will be discussed in the next section -- the transmission function $t_{22}$ has a pole of order $1$ for the same $\lambda$, thus preventing it from being an eigenvalue of $\mathcal{L}$ -- in the literature, this is known as the `NLEP paradox'.

In studies of autonomous systems, the root of $t_{11}$ close to $\lambda = 0$ is actually located precisely at $\lambda = 0$ because of the translation invariance of those autonomous systems. However,~\eqref{eq:klausmeier_model} is non-autonomous and therefore this reasoning no longer holds and the eigenvalue close to $\lambda^r_1 = 0$ can have negative or positive real part. As $t_{22}$ does \emph{not} have a pole for this $\lambda$ -- as will be discussed in the next section -- the Evans function $\mathcal{D}(\lambda)$ has a root for this value; it thus corresponds to an eigenvalue of $\mathcal{L}$. To our best knowledge, it is, in general, not possible to determine the precise location of this eigenvalue; in section~\ref{sec:smallEigenvalue} we compute its location using standard regular perturbation techniques when the non-autonomous terms are small.

%%%%%%%%%%%%%%%%%%%%%%%%%%%%%%%%%%%%%%%%%%%%%%%%%%%%%%
%% SUBSECTION (Slow transmission function $t_{22}$) %%
%%%%%%%%%%%%%%%%%%%%%%%%%%%%%%%%%%%%%%%%%%%%%%%%%%%%%%

\subsubsection{Slow transmission function $t_{22}$}

To determine the transmission function $t_{22}$, we focus on the function $\phi_2$, as defined in Proposition~\ref{lemma:EvansFunctionDecomposition}. Per construction, we know that $\phi_2(\xi;\lambda) \rightarrow t_{22}(\lambda) E_2(\lambda) e^{\Lambda_2(\lambda)\xi} + t_{23}(\lambda) E_3(\lambda) e^{\Lambda_3(\lambda)\xi} + t_{24}(\lambda) e^{\Lambda_4(\lambda)\xi}$ as $\xi \rightarrow \infty$. As $|\Lambda_4(\lambda)| \gg |\Lambda_{2,3}(\lambda)|$ for $\lambda \in \mathcal{C}_e$, the term $e^{\Lambda_4(\lambda)\xi}$ is exponentially small in the slow fields $I_s^\pm$. Therefore, we have $\phi_2(\xi;\lambda) \approx t_{22}(\lambda) E_2(\lambda) e^{\Lambda_2(\lambda)\xi} + t_{23}(\lambda) E_3(\lambda) e^{\Lambda_3(\lambda)\xi}$ for $\xi \in I_s^+$ sufficiently large. In this way, $\phi_2$ in the slow fields is related to the properties of the exponentially asymptotic constant-coefficient system $\dot{\phi}_\infty = A_\infty(\lambda,\varepsilon,\mu,m) \phi_\infty$. However, we need to relate $\phi_2$ in the slow fields to the exponentially asymptotic non-autonomous system $\dot{\phi}_s = A_s(\xi;\lambda,\varepsilon,\mu,m) \phi_s$ to determine $t_{22}$.

In the slow fields the system $\dot{\phi}_s = A_s(\xi;\lambda,\varepsilon,\mu,m) \phi_s$ has the dynamics for the $(\bar{u},\bar{p})$ part completely separated from the dynamics of the $(\bar{v},\bar{q})$ part. 
The $(\bar{u},\bar{p})$ part is governed by the non-autonomous ODE
\begin{equation}
	\left( \begin{array}{c} \dot{\bar{u}} \\ \dot{\bar{p}} \end{array} \right)
		= \varepsilon^2 \mu \left[ B_0(\lambda) + B_1(\xi) \right] \left( \begin{array}{c} \bar{u} \\ \bar{p} \end{array} \right),
\label{eq:ODEforlinstab}
\end{equation}
where
\begin{equation*}
	B_0(\lambda)  = \left( \begin{array}{cc} 0 & 1 \\ 1+ m\lambda & 0 \end{array} \right); \qquad
	B_1(\xi)  =  \left( \begin{array}{cc} 0 & 0 \\ - g(\varepsilon^2 \mu \xi) & -f(\varepsilon^2 \mu \xi) \end{array} \right).
\end{equation*}
Here, only the matrix $B_1$ carries the non-autonomous part of the differential equation and the system without $B_1$ corresponds to the $(\bar{u},\bar{p})$ part of the system $\dot{\phi}_\infty = A_\infty(\lambda,\varepsilon,\mu,m) \phi_\infty$, which has spatial eigenvalues $\Lambda_{2,3} = \pm \varepsilon^2\mu\ \sqrt{1+m\lambda}$. When $\lambda \in \mathcal{C}_e$ this autonomous system admits an exponential dichotomy on $\mathbb{R}$ and, therefore, by roughness the non-autonomous system~\eqref{eq:ODEforlinstab} does so as well, provided that $\delta  = \sup_{x \in \mathbb{R}} \sqrt{f(x)^2+g(x)^2} = \sup_{x \in \mathbb{R}} \|B_1(x)\|$ is sufficiently small. Under these conditions, there exist $\tilde{\psi}_2(\xi;\lambda) = (u_2(\xi;\lambda),p_2(\xi;\lambda),0,0)^T$ and $\tilde{\psi}_3(\xi;\lambda) = (u_3(\xi;\lambda),p_3(\xi;\lambda),0,0)^T$  such that $\tilde{\psi}_2(\xi;\lambda) \rightarrow E_2(\lambda) e^{\Lambda_2(\lambda)\xi}$ and $\tilde{\psi}_3(\xi;\lambda) \rightarrow E_3(\lambda) e^{\Lambda_3(\lambda) \xi}$ as $|\xi| \rightarrow \infty$. The same reasoning as before can now be used to deduce that $\phi_2(\xi;\lambda) \approx \tilde{\psi}_2(\xi;\lambda)$ for $\xi \in I_s^-$ and $\phi_2(\xi;\lambda) \approx t_{22}(\lambda) \tilde{\psi}_2(\xi;\lambda) + t_{23}(\lambda) \tilde{\psi}_3(\xi;\lambda)$ for $\xi \in I_s^+$.\\

To compute $t_{22}$ we need to track the changes of $\bar{u}$ and $\bar{p}$ during the fast transition when $\xi \in I_f$. From~\eqref{eq:eigenvalueProblem}, it follows that $\bar{u}$ stays constant to leading order. Hence, matching $\phi_2$ at the ends of both super-slow fields $I_s^\pm$ gives the leading order matching condition
\begin{equation} \label{eq:matchingUbar}
	u_2(0;\lambda) = t_{22}(\lambda) u_2(0;\lambda) + t_{23}(\lambda) u_3(0;\lambda).
\end{equation}
The $\bar{p}$ component changes in the fast field. On the one hand, this change is given by the difference of $\bar{p}$ values at both ends of the slow fields $I_s^\pm$, i.e.
\begin{equation}
	\Delta_\mathrm{s}\ \bar{p} = t_{22}(\lambda) p_2(0;\lambda) + t_{23}(\lambda) p_3(0;\lambda) - p_2(0;\lambda).
\end{equation}
On the other hand, the accumulated jump over the fast field is
\begin{equation} \label{eq:fastJumpPBar}
	\Delta_\mathrm{f}\ \bar{p} = \frac{1}{\mu} \int_{I_f} \left( v_p(\xi)^2 u_2(0;\lambda) + 2 u_p(\xi) v_p(\xi) \bar{v}(\xi;\lambda) \right)\ d\xi,
\end{equation}
where $\bar{v}$ satisfies $\left( \mathcal{L}^r - \lambda \right) \bar{v} = - u_2(0;\lambda) v_p(\xi)^2$. We recall that, in the fast field, to leading order, $u_p = u_0$ and $v_p = \frac{\omega}{u_0}$, where $\omega(\xi) = \frac{3}{2} \sech(\xi/2)^2$. We rescale $\bar{v}(\xi;\lambda) = - \frac{u_2(0;\lambda)}{u_0^2} V_\mathrm{in}(\xi;\lambda)$. Then~\eqref{eq:fastJumpPBar} becomes
\begin{equation}
	\Delta_\mathrm{f}\ \bar{p} = \frac{1}{\mu} \frac{u_2(0;\lambda)}{u_0^2} \int_{I_f} \left( \omega(\xi)^2 - 2 \omega(\xi) V_\mathrm{in}(\xi;\lambda) \right)\ d\xi = \frac{1}{\mu} \frac{u_2(0;\lambda)}{u_0^2} \left( 6 - 2 \mathcal{R}(\lambda) \right) + h.o.t.
\end{equation}
where 
\begin{equation}
	\mathcal{R}(\lambda) := \int_{-\infty}^\infty \omega(\xi) V_{in}(\xi;\lambda)\ d\xi
\end{equation}
and $V_\mathrm{in}$ satisfies 
\begin{equation}
	\left( \mathcal{L}^r - \lambda \right) V_\mathrm{in}(\xi;\lambda) = \omega(\xi)^2.
\end{equation}
Equating $\Delta_\mathrm{s}\ \bar{p} = \Delta_\mathrm{f}\ \bar{p}$ and by~\eqref{eq:matchingUbar} one readily derives (at leading order in $\varepsilon$)
\begin{equation}
	t_{22}(\lambda) = 1 + \frac{1}{\mu} \frac{1}{u_0^2} \frac{6 - 2 \mathcal{R}(\lambda)}{ \frac{p_2(0;\lambda)}{u_2(0;\lambda)} - \frac{p_3(0;\lambda)}{u_3(0;\lambda)}}.
\end{equation}
Because of the symmetry $f(x) = f(-x)$, $g(x) = - g(-x)$, it follows that $u_2(0;\lambda) = u_3(0;\lambda)$ and $p_2(0;\lambda) = - p_3(0;\lambda)$. Hence
\begin{equation}
	t_{22}(\lambda) = 1 + \frac{1}{\mu} \frac{1}{u_0^2} \frac{3 - \mathcal{R}(\lambda)}{ \frac{p_2(0;\lambda)}{u_2(0;\lambda)}}.
\end{equation}
The inhomogeneous ODE $\left(\mathcal{L}^r - \lambda\right) V_{in} = \omega^2$ admits bounded solutions for all $\lambda$ that are not eigenvalues of $\mathcal{L}^r$. When $\lambda$ is an eigenvalue, though, a bounded solution only exists if the following Fredholm condition is satisfied:
\begin{equation}
	\int_{-\infty}^\infty \omega^2 v^* d\xi = 0,
\end{equation}
where $v^*$ is the corresponding eigenfunction. Therefore, by Sturm-Liouville theory, it is clear that there is a bounded solution for $\lambda^r_1 = 0$, but not for $\lambda_0^r = 5/4$ or $\lambda_2^r = -3/4$. That is, $\mathcal{R}(\lambda)$, and therefore $t_{22}$, has poles of order $1$ at $\lambda_0^r$ and $\lambda_2^r$.\\

We have, hence, demonstrated the following:

%========================%
% LEMMA (Evans function) %
%========================%

\begin{lemma}[\underline{Evans function}]\label{lemma:evans_function}
Let the conditions of Theorem~\ref{theorem:fg_general} and assumption (A4) be fulfilled, and let $ (U_p, V_p) $ be a pulse solution to \eqref{eq:klausmeier_model} as described in Theorem~\ref{theorem:fg_general}. It then holds true that the eigenvalues of the operator $ \mathcal{L} $ in \eqref{eq:linearization_operator} arising from linearization around the pulse solution $ (U_p, V_p) $ coincide on $ \mathcal{C}_e $ with the roots of the Evans function
\begin{align}\label{eq:evans_function_}
 \mathcal{D}(\lambda) = t_{11}(\lambda)t_{22}(\lambda)\widetilde{\mathcal{D}}(\lambda) \, ,
\end{align}
with $ \widetilde{\mathcal{D}}(\lambda) \neq 0, \lambda \in \mathcal{C}_e$ and where the so-called fast transmission function is given by
\begin{align}\label{eq:t_11_fast}
 t_{11}(\lambda) = C_{1} \left(\lambda - \lambda_0^f\right) \left(\lambda - \lambda_1^f \right) \left(\lambda - \lambda_2^f \right) \, ,
\end{align}
with $ \lambda_1^f = \mathcal{O}(\varepsilon) $, while the so-called slow transmission function is given by
\begin{align}\label{eq:t_22_slow}
 t_{22}(\lambda) = C_{2} \frac{\widetilde{t}_{22}(\lambda)}{\left(\lambda - \lambda_0^f \right) \left(\lambda - \lambda_2^f \right)} \, ,
\end{align}
with some $ C_1, C_2, \lambda_0^f, \lambda_2^f \in \mathbb{R}\setminus\{ 0 \} $ and $ \widetilde{t}_{22} $ an analytic function on $ \mathcal{C}_e $. In particular, 
\begin{equation}\label{eq:NLEPt22expression}
 t_{22}(\lambda) = 1 + \frac{1}{u_0^2\mu} \left( \frac{3 - \mathcal{R}(\lambda)}{p_2(0;\lambda)/u_2(0;\lambda)} \right) \, ,
\end{equation}
where $p_2(0;\lambda)/u_2(0;\lambda)$ is the slope of the unstable manifold of the trivial solution to~\eqref{eq:ODEforlinstab} at $x = 0$, and $\mathcal{R}$ is given (at leading order in $\varepsilon$) by
\begin{equation}\label{eq:definitionRfunction}
 \mathcal{R}(\lambda) = \int_{-\infty}^\infty \frac{3}{2} \sech(\xi/2)^2 V_\mathrm{in}(\xi;\lambda)\ d\xi \, ,
\end{equation}
where $V_\mathrm{in}$ satisfies $\left(\mathcal{L}^r - \lambda \right) V_\mathrm{in} = \frac{9}{4} \sech(\xi/2)^4$.
\end{lemma}

%================================%
% REMARKS (\mathcal{R}(\lambda)) %
%================================%

\begin{remark}
The function $\mathcal{R}$ has been extensively studied in~\cite[Section 3.1.1]{BD18},~\cite[Section 4.1]{DP02} and~\cite[Section 5]{DRS12}. We would like to stress, however, that $\mathcal{R}$ in this article has a different factor in front of it and is defined in terms of $\lambda$, whereas in~\cite{DP02,DRS12} it is defined as function of $P:= 2 \sqrt{1+\lambda}$. A plot of $\mathcal{R}$ has been included in Figure~\ref{fig:Rlambda}.
\end{remark}

\begin{figure}
	\centering
	\includegraphics[width=0.3\textwidth]{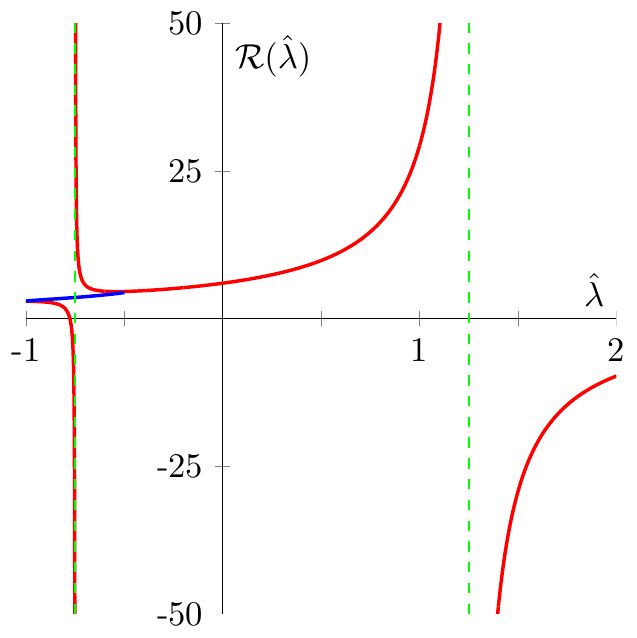}
	\caption{A plot of the function $\mathcal{R}(\lambda)$. The red lines show the form of $\mathcal{R}(\lambda)$ for real-valued $\lambda$, whereas the blue lines also show the complex $\lambda$ for which $\mathcal{R}(\lambda)$ is real-valued; the green, dashed lines indicate the poles of the $\mathcal{R}(\lambda)$.}
\label{fig:Rlambda}
\end{figure}

\begin{remark}
The eigenvalue problem is often written as a nonlocal eigenvalue problem (NLEP). This can be achieved via the transformation
\begin{equation*}
	V_\mathrm{in}(\xi;\lambda) = \frac{3 - \mu u_0^2 \frac{p_2(0;\lambda)}{u_2(0;\lambda)}}{\int_{-\infty}^{\infty} \omega(\xi) f(\xi;\lambda)\ d\xi} z(\xi;\lambda),
\end{equation*}
which results in the NLEP
\begin{equation*}
	\left( \mathcal{L}^r - \lambda \right) z = \frac{\omega^2 \int_{-\infty}^\infty \omega z \ d\xi}{3 - \mu u_0^2 \frac{p_2(0;\lambda)}{u_2(0;\lambda)}}.
\end{equation*}
\end{remark}

%%%%%%%%%%%%%%%%%%%%%%%%%%%%%%%%%%%%%%%%%%%%%%%%%%%%%%%%%%
%% SUBSECTION (Roots of transmission function $t_{22}$) %%
%%%%%%%%%%%%%%%%%%%%%%%%%%%%%%%%%%%%%%%%%%%%%%%%%%%%%%%%%%

\subsubsection{Roots of transmission function $t_{22}$}

In the constant coefficient case $ f,g \equiv 0 $, we have that $p_2(0;\lambda)/u_2(0;\lambda) = \sqrt{1+m\lambda}$ and so $ t_{22}(\lambda) = 0 $ reduces to
\begin{equation}\label{eq:stabilityConditionGeneral}
	\mu u_0^2 = \frac{\mathcal{R}(\lambda)-3}{\sqrt{1+m\lambda}},
\end{equation}
with $u_0$ as in~\eqref{eq:definitionu0}, and eigenvalues can be readily extracted from this condition -- see~\cite{BD18}; in Figure~\ref{fig:stabilityCondition}, we show plots of the right-hand side for various $m$. With additional asymptotic approximations, $m \ll 1$ and $m \gg 1$, this can be reduced even further, to leading order to,
\begin{equation}
	\begin{array}{rcll}
		\mu u_0^2 & = & \mathcal{R}(\lambda) - 3, & \mbox{ when $m \ll 1$;} \\
		\nu u_0^2 & = & \frac{\mathcal{R}(\lambda)-3}{\sqrt{\lambda}}, & \mbox{ when $m \gg 1$;}
	\end{array}
\end{equation}
where 
\begin{align}\label{eq:nu}
 \nu = \frac{m^2 D}{a^2} = \mu \sqrt{m} \, .
\end{align}
Now, when $\mu \ll 1$, respectively $\nu \ll 1$, the left-hand side of these expressions becomes asymptotically small (since $u_0 = u_0^- = \mathcal{O}(1)$, see~\eqref{eq:definitionu0} and Remark~\ref{remark:u0_autonomous_mu_small}), but stays positive. Hence solutions $\lambda$ accumulate at points for which $\mathcal{R}(\lambda) - 3 \approx 0$, which happens to be at the tip of the essential spectrum, i.e. $\lambda = \underline{\lambda}/m \approx -1$, see Figure~\ref{fig:stabilityCondition} and~\cite{BD18}. Certainly, no eigenvalues with positive real parts are found.

\begin{figure}
	\centering
	\begin{subfigure}[t]{0.3\textwidth}
		\centering
			\includegraphics[width=\textwidth]{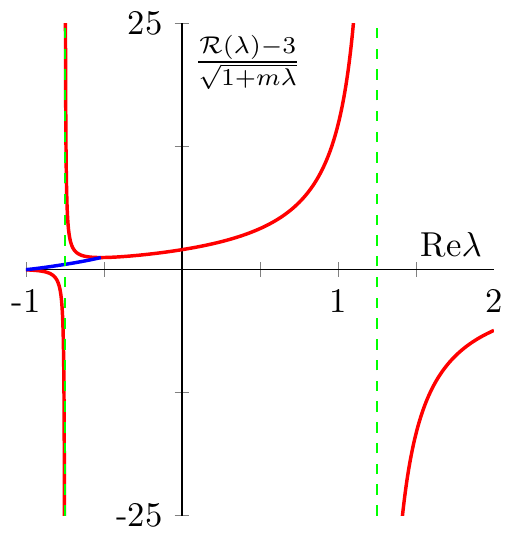}
		\caption{$m=0.45$}
	\end{subfigure}
	\begin{subfigure}[t]{0.3\textwidth}
		\centering
			\includegraphics[width=\textwidth]{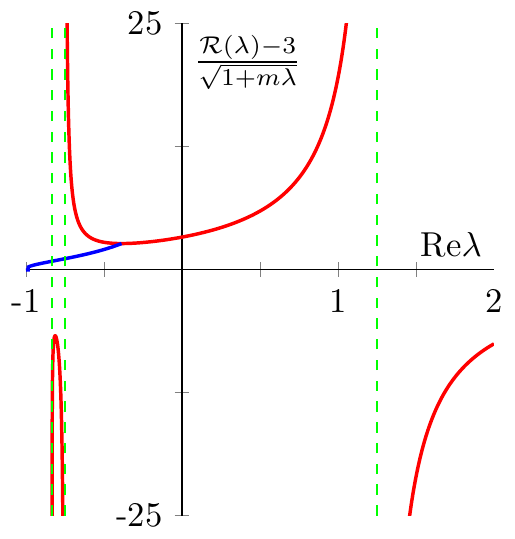}
		\caption{$m=1.2$}
	\end{subfigure}
	\begin{subfigure}[t]{0.3\textwidth}
		\centering
			\includegraphics[width=\textwidth]{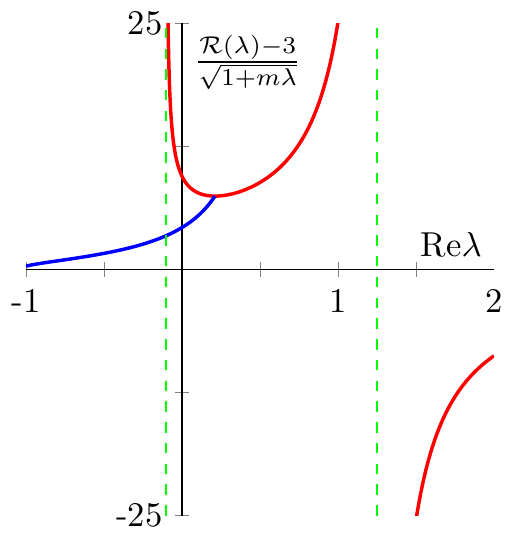}
		\caption{$m=10$}
	\end{subfigure}
\caption{Plots of the right-hand side of \eqref{eq:stabilityConditionGeneral} for various $m$. The red lines indicate the values for real-valued $\lambda$, whereas the blue lines indicate complex $\lambda$ for which the right-hand side of~\eqref{eq:stabilityConditionGeneral} is real-valued; in green the poles are shown; see~\cite{BD18} for more details.}
\label{fig:stabilityCondition}
\end{figure}

This idea can be expanded to include the non-autonomous cases. For this, as in the existence problem, we relate the non-autonomous equation to the autonomous equation. Here, it is useful to rescale \eqref{eq:ODEforlinstab} such that it has the form of \eqref{eq:slow-system}. Specifically, we set $\tilde{x} = \varepsilon^2 \mu |\sqrt{1 + m \lambda}| \xi$ and $\bar{p} = |\sqrt{1+m\lambda}| \tilde{p}$, under which \eqref{eq:ODEforlinstab} turns into the system
\begin{equation}\label{eq:stabilitySystemStandardForm}
	\left(\begin{array}{c} \bar{u}' \\ \tilde{p}' \end{array}\right) = 
\left[ \left( \begin{array}{cc}
	0 & 1 \\ 1 & 0
\end{array}\right) 
+\left(
\begin{array}{cc}
0 & 0 \\
- \frac{g(\tilde{x}/|\sqrt{1+m\lambda}|)}{|1+m\lambda|} & - \frac{f(\tilde{x}/|\sqrt{1+m\lambda}|)}{|\sqrt{1+m\lambda}|}
\end{array}
\right)
\right]
\left(\begin{array}{c} \bar{u} \\ \tilde{p} \end{array}\right).
\end{equation}
The autonomous part of this equation corresponds to the autonomous part for the existence problem -- see section~\ref{sec:dynamics_slow_manifold} -- and thus possesses an exponential dichotomy with constants $K = 1$ and $\rho = 1$. Therefore, for a given $\lambda \in \mathcal{C}_e$, by roughness (Proposition~\ref{prop:roughness_closeness_general}) it follows that the full non-autonomous equation has an exponential dichotomy as well when
\begin{equation}
	\sup_{x \in \mathbb{R}}  \frac{1}{|\sqrt{1+m\lambda}|} \sqrt{ \frac{g(x)^2}{|1+m\lambda|} + f(x)^2} < \frac{1}{4}.
\end{equation}
It is easily verified that this condition is satisfied when
\begin{equation}
	\delta = \sup_{x \in \mathbb{R}} \sqrt{f(x)^2+g(x)^2} < \delta_c(\lambda) := \frac{1}{4} |\sqrt{1+m\lambda}| \left| \sqrt{ \frac{1+m\lambda}{2+m\lambda}}\right|.
\end{equation}
Thus, for all $\lambda \in \mathcal{C}_e$, we obtain a (different) bound $\delta_c(\lambda)$. Since $\delta_c(\lambda) \downarrow 0$ as $|\sqrt{1+m\lambda}| \downarrow 0$ -- i.e. when $\lambda$ approaches $- 1/ m$ -- we cannot take the infimum over the region $\mathcal{C}_e$. Instead, we further restrict $\lambda$ to $\lambda \in \tilde{C}_e := \mathcal{C}_e \cap \left\{\lambda \in \mathbb{C}: |\lambda + \frac{1}{m}| > \frac{1}{2m}\right\}$. Note that $\mathbb{C}^+ \subset \tilde{C}_e$. Then the infimum of $\delta_c(\lambda)$ over this region exists, and we define it as $\delta_c := \inf_{\lambda \in \tilde{C}_e} \delta_c(\lambda) = \frac{\sqrt{6}}{24} \approx 0.102$. Thus, if $\delta < \delta_c$, \eqref{eq:stabilitySystemStandardForm} possesses an exponential dichotomy for all $\lambda \in \tilde{C}_e$.

Moreover, for all $\lambda \in \tilde{C}_e$ and $\delta < \delta_c$, the slope $p_2(0;\lambda)/u_2(0;\lambda)$ of the non-autonomous case can be related to that of the autonomous case, along the same lines as in the existence proof in section~\ref{sec:dynamics_slow_manifold} (specifically, as in Lemma~\ref{lemma:closeness_slopes}). That is, there are $\mathcal{O}(1)$ constants $0 < C_-(\delta) \leq 1 \leq C_+(\delta)$ such that $\tilde{p}(0;\lambda) = C \bar{u}(0;\lambda)$ for some $C \in \left( C_-(\delta), C_+(\delta) \right)$. Rescaling back to the original variables then yields $p_2(0;\lambda) / u_2(0;\lambda) = C \sqrt{1 + m \lambda}$. Therefore $t_{22}(\lambda) = 0$ reduces to
\begin{equation}
	C\mu u_0^2 = \frac{\mathcal{R}(\lambda)-3}{\sqrt{1+m\lambda}}.
\label{eq:NLEPnonAutonomous}
\end{equation}
The asymptotic arguments for the autonomous case can now be repeated and it readily follows that no solutions are found with $\lambda \in \tilde{C}_e$. In particular $ t_{22}(\lambda) = 0 $ does not have solutions with $\mbox{Re} \lambda > 0$.  We, hence, have the following result.

%=======================================================%
% PROPOSITION (Roots of the slow transmission function) %
%=======================================================%

\begin{proposition}[\underline{Roots of the slow transmission function}]\label{proposition:slow_transmission_function}
Let $ t_{22} $ be the slow transmission function from Lemma~\ref{lemma:evans_function}. Then, for $\lambda \in \left\{\lambda \in \mathcal{C}_e: \| \lambda + \frac{1}{m} \| > \frac{1}{2m}\right\}$,
\begin{equation}\label{eq:NLEPt22expression_details}
 t_{22}(\lambda) = 1 + \frac{1}{u_0^2\mu} \left( \frac{3-\mathcal{R}(\lambda)}{C \sqrt{1+m \lambda}} \right) \, ,
\end{equation}
with $u_0 = u_0^-$ as in \eqref{eq:definitionu0} and for some $ C \in \mathbb{R} $ with
\begin{align}
 0 < C_{\mathrm{min}}(\delta) < C < C_{\mathrm{max}}(\delta) < \infty
\end{align}
and $ C_{\mathrm{min}/\mathrm{max}}(\delta) $ defined as in Lemma~\ref{lemma:closeness_slopes}. \\
Moreover, if either of the following two asymptotic approximations hold true,
\begin{itemize}
	\item[(i)] $m \ll 1$ and $\mu \ll 1$;
	\item[(ii)] $m \gg 1$ and $\nu \ll 1$,
\end{itemize}
then $t_{22}(\lambda) = 0$ does not have any solution $\lambda \in \mathcal{C}_e$ with $\mbox{Re} \lambda > 0$.

\end{proposition}

Combining Lemma~\ref{lemma:evans_function} with Proposition~\ref{proposition:slow_transmission_function} readily demonstrates Theorem~\ref{theorem:point_spectrum}.

%%%%%%%%%%%%%%%%%%%%%%%%%%%%%%%%%%%%%%%%%%%%%%%%%%%%%%%%
%% SUBSECTION (Further remarks) %%
%%%%%%%%%%%%%%%%%%%%%%%%%%%%%%%%%%%%%%%%%%%%%%%%%%%%%%%%

\subsubsection{Further remarks}

If the asymptotic conditions  on $m$, $\mu$ and $\nu$ from Proposition~\ref{proposition:slow_transmission_function} do not hold, equation~\eqref{eq:NLEPnonAutonomous} still holds. By restricting $\delta$ further (i.e. taking a lower bound $\delta_c$) stronger bounds on the constant $C_+$ can be enforced that guarantee all roots of $t_{22}$ lie to the left of the imaginary axis. The proof of this heavily relies on the proof for the autonomous case (see e.g.~\cite{BD18}) and a careful estimation of the constant $C_+$. Specifically, the following lemma can be established:

%==========================%
% LEMMA (better estimates) %
%==========================%

\begin{lemma}\label{lemma:rootsOft22ComputedBetter}
	Let the conditions of Proposition~\ref{lemma:EvansFunctionDecomposition} be fulfilled. Then there exists critical values $m_c = 3$, $0 < \mu^*(m) < \frac{1}{12}$ (see Theorem~\ref{theorem:fg_general}) and $\nu^*(m) > 0$ such that if either of the following holds
	\begin{itemize}
		\item[(i)] $m < m_c$ and $\mu < \mu^*(m)$;
		\item[(ii)] $m > m_c$ and $\nu < \nu^*(m)$;
		\item[(iii)] $m = m_c$ and $\mu < \mu^*(m)$ and $\nu < \nu^*(m)$,
	\end{itemize}
then there exists a $\delta_c > 0$ such that if $\delta < \delta_c$ the condition~\eqref{eq:NLEPnonAutonomous} has no solutions with $\mbox{Re} \lambda > 0$; that is, $t_{22}$ has not roots with positive real part.
\end{lemma}

\begin{remark}
	In~\eqref{eq:NLEPnonAutonomous}, the left-hand side is always real-valued. Hence, only $\lambda \in \mathbb{C}$ for which the right-hand side is real-valued can satisfy~\eqref{eq:NLEPnonAutonomous}. Due to this, eigenvalues can only appear on a skeleton in $\mathbb{C}$, of which the form only depends on $m$. In Figure~\ref{fig:eigenvalueSkeletons} we show several skeletons for different $m$. Note that this is the reason for (the shape of) the bounds on the `large' eigenvalues shown in Figure~\ref{fig:spectralBounds} (in red).
\end{remark}

\begin{figure}
	\centering
	\begin{subfigure}[t]{0.3\textwidth}
		\centering
			\includegraphics[width=\textwidth]{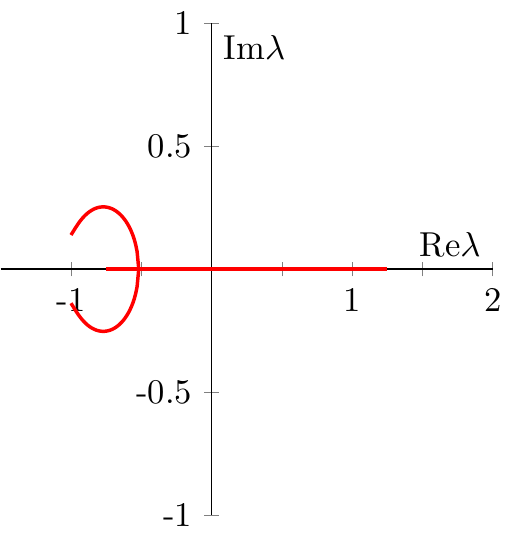}
		\caption{$m=0.45$}
	\end{subfigure}
~
	\begin{subfigure}[t]{0.3\textwidth}
		\centering
			\includegraphics[width=\textwidth]{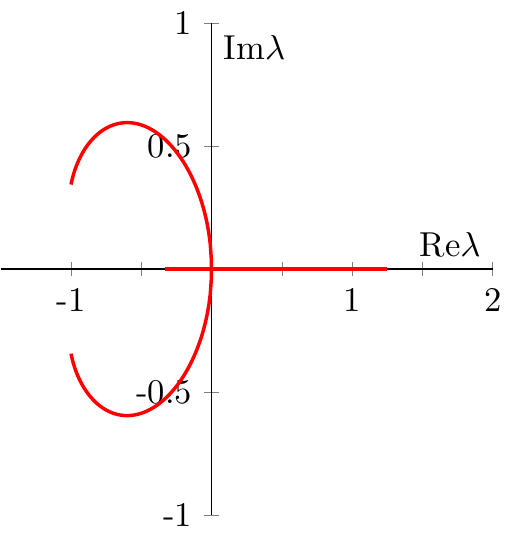}
		\caption{$m=3$}
	\end{subfigure}
~
	\begin{subfigure}[t]{0.3\textwidth}
		\centering
			\includegraphics[width=\textwidth]{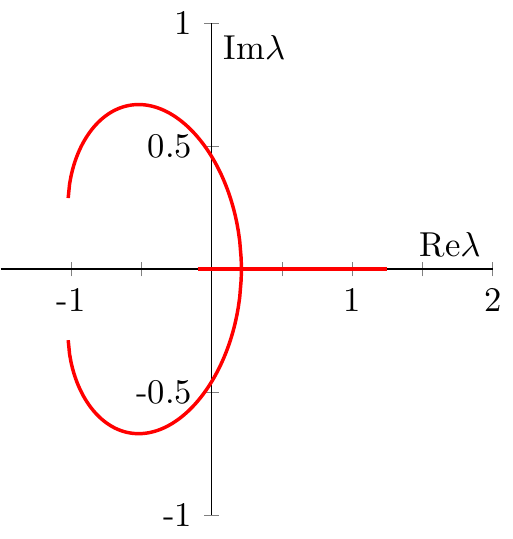}
		\caption{$m=10$}
	\end{subfigure}
	\caption{Plots of skeletons on which $\lambda$ that satisfy~\eqref{eq:NLEPnonAutonomous} necessarily need to lie.}
	\label{fig:eigenvalueSkeletons}
\end{figure}

\begin{remark}\label{remark:varTer_positiveu0}
	The arguments in this section have been applied to pulse solutions with $u_0 = u_0^-$ (see \eqref{eq:definitionu0}; $u_0^-$ as in \eqref{eq:intersection_points} and \eqref{eq:intersection_points_aut}). There also exist pulse solutions with $u_0 = u_0^+$ (with $u_0^+$ as in \eqref{eq:intersection_points} and \eqref{eq:intersection_points_aut}) and the reasoning also holds for these, up to equation~\eqref{eq:NLEPnonAutonomous}. However, $u_0^+ = \mathcal{O}\left(\frac{1}{\mu}\right)$ for these solutions (see Remark \ref{remark:u0_autonomous_mu_small}) and as an effect the left-hand side of~\eqref{eq:NLEPnonAutonomous} thus is asymptotically large (for $\mu \ll 1$). As result, eigenvalues accumulate around the poles of the right-hand side. In particular, because of this, these alternative pulse solution necessarily have an eigenvalue close to $\lambda =5/4 > 0$, making these pulse solutions unstable.
\end{remark}

\begin{remark}
	If $\delta \ll 1$, a direct application of roughness of exponential dichotomies can be used to directly prove that eigenvalues of~\eqref{eq:eigenvalueProblemMatrixForm} necessarily lie $\mathcal{O}(\delta)$ close to eigenvalues of the problem with $f \equiv 0$, $g \equiv 0$.
\end{remark}

\begin{remark}\label{remark:stability_fgLimits}
	If $\lim_{x \rightarrow \pm \infty} f(x),g(x)$ exist but are not (all) equal to zero, a similar result can be found with minor changes to the proof -- provided that the essential spectrum lies to the left of the imaginary axis.
\end{remark}

\begin{remark}\label{remark:stability_fgBounded}
	If $\lim_{x \rightarrow \pm \infty} f(x),g(x)$ do not exists, the outlined proof fails because the `elephant trunk' procedure used in the proof of Lemma~\ref{lemma:EvansFunctionDecomposition} does no longer work. If $f$ and $g$ approach (possibly different) period functions for $x \rightarrow \infty$ a variant of this proof using a Ricatti transformation such as in~\cite{BjornRiccati} seems possible. 
\end{remark}

%%%%%%%%%%%%%%%%%%%%%%%%%%%%%%%%%%%%%%%%%%%%%%%%%%%%%%%%%%
%% SUBSECTION (Small eigenvalue close to $\lambda = 0$) %%
%%%%%%%%%%%%%%%%%%%%%%%%%%%%%%%%%%%%%%%%%%%%%%%%%%%%%%%%%%

\subsection{Small eigenvalue close to $\lambda = 0$ (Proof of Theorem~\ref{theorem:point_spectrum_small})}\label{sec:smallEigenvalue}
In this section we assume that
\begin{align}
 f(x) = \delta \widetilde{f}(x) \, , \quad  g(x) = \delta \widetilde{g}(x) \, , \qquad 0 < \delta \ll 1 \,,  \widetilde{f},  \widetilde{g} = \mathcal{O}(1) \, , \quad \sup_{x \in \mathbb{R}} \sqrt{\widetilde{f}(x)^2+\widetilde{g}(x)^2} = 1 ,
\end{align}
which will ease the derivation of a more detailed estimate (as given in Theorem~\ref{theorem:point_spectrum_small}) of the location of the small eigenvalue around $ \lambda = 0 $ (in terms of $ \delta $), so we set
\begin{align}
 \lambda = \delta \widetilde{\lambda} \, .
\end{align}
The strategy to derive such an estimate is to relate the eigenvalue and existence problems in an appropriate way and then use the Fredholm alternative. To this end, let us write the eigenvalue problem in the fast field \eqref{eq:eigenvalueProblem} in the more concise form
\begin{align}\label{eq:evp_concise}
 \delta \tilde \lambda
 \left(
 \begin{array}{cc}
  \varepsilon^4 \mu^2 m & 0\\
  0 & 1
 \end{array}
 \right)
   \left(
 \begin{array}{c}
  \bar{u}\\
  \bar{v}
 \end{array}
 \right)
  = \mathbb{L}_{u_p, v_p}
  \left(
 \begin{array}{c}
  \bar{u}\\
  \bar{v}
 \end{array}
 \right) \, ,
\end{align}
and the existence problem in the fast field \eqref{eq:klausmeier_model_constant_ODE_second_order} as
\begin{align}
 0 = L_{h}   \left(
 \begin{array}{c}
  u_p\\
  v_p
 \end{array}
 \right) +
 \delta
 L_{in}(\xi)   \left(
 \begin{array}{c}
  u_p\\
  v_p
 \end{array}
 \right) +
 N  \left(
 \begin{array}{c}
  u_p\\
  v_p
 \end{array}
 \right) +
 \left(\begin{array}{c} a \\ 0 \end{array}\right)
\, ,
\end{align}
with (the linear part with constant coefficients)
\begin{align}
 L_h = 
 \left(
 \begin{array}{cc}
  \partial_\xi^2 - \varepsilon^4 \mu^2 & 0 \\
   0 & \partial_\xi^2 -1 
 \end{array}\right) \, , \quad 
\end{align}
and 
\begin{align}
 L_{in}(\xi) = 
 \left(
 \begin{array}{cc}
  \varepsilon^2 \mu\ \widetilde{f}(\varepsilon^2 \mu \xi) \partial_{\xi} + \varepsilon^4 \mu^2\ \widetilde{g}(\varepsilon^2 \mu \xi) & 0   \\
   0 & 0
 \end{array}\right) \, , \quad 
\end{align}
and $ N $ the nonlinear terms. Recall that in the autonomous case the derivative of the pulse solution is an eigenfunction for the zero eigenvalue. Motivated by this, we take a derivative w.r.t. $\xi$ of the non-autonomous existence problem which gives
\begin{align}\label{eq:diff_existence}
 0 = \underbrace{[L_{h} + \delta L_{in}(\xi) + DN(u_p, v_p)]}_{= \mathbb{L}_{u_p, v_p}}
 \left(
 \begin{array}{c}
  \dot{u}_p\\
  \dot{v}_p
 \end{array}
 \right) +
 \delta
  \left(\frac{d}{d\xi} L_{in}(\xi)\right)   \left(
 \begin{array}{c}
  u_p\\
  v_p
 \end{array}
 \right) \, ,
\end{align}
and plug into the above eigenvalue problem \eqref{eq:evp_concise} the ansatz
\begin{align}
\left(
 \begin{array}{c}
   \bar{u} \\ \bar{v}
 \end{array}
\right)
 = 
\left(
 \begin{array}{c}
   \dot{u}_p \\ \dot{u}_p
 \end{array}
\right)
+
\delta
\left(
 \begin{array}{c}
   \widetilde{u} \\ \widetilde{v}
 \end{array}
\right)
 \, ,
\end{align}
which results in
\begin{align}
 \delta \tilde \lambda
 \left(
 \begin{array}{cc}
  \varepsilon^4 \mu^2 m & 0\\
  0 & 1
 \end{array}
 \right)
\left(
 \begin{array}{c}
   \dot{u}_p \\ \dot{u}_p
 \end{array}
\right)
+
 \delta^2 \tilde \lambda
 \left(
 \begin{array}{cc}
  \varepsilon^4 \mu^2 m & 0\\
  0 & 1
 \end{array}
 \right)
   \left(
 \begin{array}{c}
  \widetilde{u}\\
  \widetilde{v}
 \end{array}
 \right) 
= 
 \mathbb{L}_{u_p, v_p}
 \left(
 \begin{array}{c}
   \dot{u}_p \\ \dot{u}_p
 \end{array}
\right)
+
\delta
\mathbb{L}_{u_p, v_p}
   \left(
 \begin{array}{c}
  \widetilde{u}\\
  \widetilde{v}
 \end{array}
 \right) \, .
\end{align}
Upon using \eqref{eq:diff_existence} to replace the term featuring $  \mathbb{L}_{u_p, v_p}( \dot{u}_p, \dot{v}_p)^T $, we get
\begin{align}
 \delta \tilde \lambda
 \left(
 \begin{array}{cc}
  \varepsilon^4 \mu^2 m & 0\\
  0 & 1
 \end{array}
 \right)
\left(
 \begin{array}{c}
   \dot{u}_p \\ \dot{v}_p
 \end{array}
\right)
+
 \delta^2 \tilde \lambda
 \left(
 \begin{array}{cc}
  \varepsilon^4 \mu^2 m & 0\\
  0 & 1
 \end{array}
 \right)
   \left(
 \begin{array}{c}
  \widetilde{u}\\
  \widetilde{v}
 \end{array}
 \right) 
= 
- \delta
  \left(\frac{d}{d\xi} L_{in}(\xi) \right)  \left(
 \begin{array}{c}
  u_p\\
  v_p
 \end{array}
 \right) 
+
\delta
\mathbb{L}_{u_p, v_p}
   \left(
 \begin{array}{c}
  \widetilde{u}\\
  \widetilde{v}
 \end{array}
 \right) 
\end{align}
For the perturbation analysis to follow we will use the notation $ u_{p,0}, v_{p,0}, \bar{u}_{0}, \bar{v}_{0} $ to indicate the leading order in $ \delta $ of the corresponding terms. In particular, $ u_{p,0}, v_{p,0} $ are the pulse solutions for the homogeneous case $ f=g=0 $ as described in Corollary~\ref{cor:fg_equal_zero}. We, hence, arrive at the leading order in $ \delta $ of the previous equation
\begin{align}\label{eq:perturbation_first_order}
 \mathbb{L} \left(
 \begin{array}{c}
  \widetilde{u}_0\\
  \widetilde{v}_0
 \end{array}
 \right)
 =
 \left(
 \begin{array}{c}
  \alpha\\
  \beta
 \end{array}
 \right)
\end{align}
with
\begin{align}
 \mathbb{L}:=  \mathbb{L}_{u_{p,0}, v_{p,0}} = 
 \left(
 \begin{array}{cc}
  \partial_\xi^2 - \varepsilon^4 \mu^2 - \varepsilon^2 v_{p,0}^2 & - 2 \varepsilon^2 u_{p,0} v_{p,0} \\
   v_{p,0}^2 & \partial_\xi^2 -1 + 2 u_{p,0} v_{p,0}
 \end{array}\right) \, ,
\end{align}
and
{\footnotesize
\begin{align}
 \left(
 \begin{array}{c}
  \alpha\\
  \beta
 \end{array}
 \right)
 :=
  \tilde \lambda
 \left(
 \begin{array}{cc}
  \varepsilon^4 \mu^2 m & 0\\
  0 & 1
 \end{array}
 \right)
\left(
 \begin{array}{c}
   \dot{u}_{p,0} \\ \dot{v}_{p,0}
 \end{array}
\right)
+
  \left(\frac{d}{d\xi} L_{in}(\xi)\right)   \left(
 \begin{array}{c}
  u_{p,0}\\
  v_{p,0}
 \end{array}
 \right)  
 =
 \left(
 \begin{array}{c}
  \varepsilon^4 \mu^2 m \tilde{\lambda} \dot{u}_{p,0} + \varepsilon^4 \mu^2 \tilde{f}'(\varepsilon^2 \mu \xi) \dot{u}_{p,0} + \varepsilon^6    \mu^3 \tilde{g}'(\varepsilon^2 \mu \xi) u_{p,0} \\
  \tilde{\lambda} \dot{v}_{p,0}
 \end{array}
 \right)  \, .
 \end{align}
}
In order to find an expression for the eigenvalue correction $ \widetilde{\lambda} $, we will make use of the Fredholm alternative for \eqref{eq:perturbation_first_order}. Hence, we first need to study the kernel of the adjoint operator
\begin{equation*}
\mathbb{L}^* = \left(\begin{array}{cc} 
\partial_\xi^2 - \varepsilon^4 \mu^2 - \varepsilon^2 v_{p,0}^2 
	& v_{p,0}^2 \\
- 2 \varepsilon^2 u_{p,0} v_{p,0} 
	& \partial_\xi^2 - 1 + 2 u_{p,0} v_{p,0}
\end{array}\right) \, ,
\end{equation*}
that is, to find $ (u^*,v^*)^T $ with 
\begin{align}\label{eq:adjoint_problem}
 \mathbb{L}^* 
 \left(
 \begin{array}{c}
  u^*\\
  v^*
 \end{array}
 \right)
 =
 0 \, ,
\end{align}
and rearrange the solvability condition
\begin{equation} \label{eq:smallEigenvalueFredholmCondition}
\left\langle \left( \begin{array}{c} u^* \\ v^* \end{array} \right), \left(\begin{array}{c} \alpha \\ \beta \end{array} \right) \right\rangle_{L^2 \times L^2} = 0,
\end{equation}
to get an expression for $ \tilde{\lambda} $. Since \eqref{eq:adjoint_problem} is again a singularly perturbed problem (in $ \varepsilon $), we split this problem into three regions: two slow regions, $I_s^\pm$, and one fast region, $I_f$. As described in Theorem~\ref{theorem:fg_general} and Corollary~\ref{cor:fg_equal_zero}, we have
\begin{equation}\label{eq:smallEigenvalueExistence}
	u_{p,0,0}(\xi) =
\begin{cases}
	\frac{1}{\mu} \left[1-(1-\mu u_0)e^{+\varepsilon^2 \mu \xi} \right] \, , & \xi \in I_s^-; \\
	u_0, & \xi \in I_f; \\
	\frac{1}{\mu} \left[1-(1-\mu u_0)e^{-\varepsilon^2 \mu \xi} \right] \, , & \xi \in I_s^+,
\end{cases}\, \quad
	v_{p,0,0}(\xi) =
\begin{cases}
	0, & \xi \in I_s^-;\\
	\frac{1}{u_0} \omega(\xi), & \xi \in I_f;\\
	0, & \xi \in I_s^+,
\end{cases},
\end{equation}
where $\omega(\xi) = \frac{3}{2} \sech(\xi/2)^2$ and the notation ``$p,0,0$'' indicates that this the leading order in both, $ \delta $ and $ \varepsilon $. In the slow regions we have $v_{p,0,0} = 0$ to leading order and therefore (again to leading order)
\begin{equation}
 u^*(\xi) = 
\begin{cases}
C^- e^{\varepsilon^2 \mu \xi}, & \xi \in I_s^-;\\
C^+ e^{-\varepsilon^2 \mu \xi}, & \xi \in I_s^+;
\end{cases}\quad \, \quad
 v^*(\xi) =
\begin{cases}
D^- e^{\xi}, & \xi \in I_s^-;\\
D^+ e^{-\xi}, & \xi \in I_s^+,
\end{cases}
\end{equation}
where $C^\pm$ and $D^\pm$ are constants that need to be found via matching with the fast field at $\xi = \pm 1/\sqrt{\varepsilon}$. In the fast region, the adjoint problem is to leading order given by
\begin{equation*}
	\left\{
	\begin{array}{rcl}
		0 & = & \ddot{u}^* + \frac{1}{u_0^2} \omega^2 v^* \, , \\
		0 & = & \ddot{v}^* - v^* + 2 \omega v^* \, .
	\end{array}
	\right.
\end{equation*}
Up to a multiplicative constant, the only bounded solution to the $v^*$-equation is $v^* = \frac{1}{u_0} \omega'$. Matching with the slow fields indicates $D^\pm = 0$. The expression for $u^*$ in $I^f$ can be found by integrating twice, which reveals
\begin{align*}
u^*(\xi) 
 = - \frac{1}{3 u_0^3} \int^\xi \omega^3(z)\ dz + C_2 
 = - \frac{1}{3 u_0^3} \frac{9}{20} \left[ 6 \cosh(\xi) + \cosh(2\xi) + 8 \right] \tanh(\xi/2) \sech(\xi/2)^4 + C_2
 =: \sigma(\xi) \,  . 
\end{align*}
The value of $C_2$ turns out to be irrelevant and therefore we choose $C_2 = 0$ for simplicity of presentation. Matching with the slow fields then gives $C^- = \frac{6}{5 u_0^3}$ and $C^+ = -\frac{6}{5 u_0^3}$. In summary, we have to leading order in $ \varepsilon $
\begin{equation}\label{eq:solution_adjoint_problem}
	u^*(\xi) =
\begin{cases}
	+\frac{6}{5u_0^3} \, e^{+\varepsilon^2 \mu \xi} \, , & \xi \in I_s^-; \\
	\sigma(\xi) \, , & \xi \in I_f; \\
	-\frac{6}{5u_0^3} \, e^{-\varepsilon^2 \mu \xi} \, , & \xi \in I_s^+,
\end{cases}\, \quad
	v^*(\xi) =
\begin{cases}
	0 \, , & \xi \in I_s^-;\\
	\frac{1}{u_0} \omega'(\xi)\, , & \xi \in I_f;\\
	0 \, , & \xi \in I_s^+,
\end{cases},
\end{equation}
and 
\begin{equation}\label{eq:alpha_leading_order}
	\alpha(\xi) =
\begin{cases}
	\varepsilon^6 \mu^2 \, e^{+\varepsilon^2 \mu \xi} 
	\left[-m \widetilde{\lambda}(1-\mu u_0)- \widetilde{f}'(\varepsilon^2 \mu \xi) (1-\mu u_0)
	      + \widetilde{g}'(\varepsilon^2 \mu \xi)\left( e^{-\varepsilon^2 \mu \xi} + \mu u_0 - 1 \right) \right] \, , & \xi \in I_s^-; \\
	\varepsilon^6 \mu^3 \tilde{g}'(\varepsilon^2 \mu \xi) u_0 \, , & \xi \in I_f; \\
	\varepsilon^6 \mu^2 \, e^{-\varepsilon^2 \mu \xi} 
	\left[\ \ m \widetilde{\lambda}(1-\mu u_0)+ \widetilde{f}'(\varepsilon^2 \mu \xi) (1-\mu u_0)
	      + \widetilde{g}'(\varepsilon^2 \mu \xi)\left( e^{+\varepsilon^2 \mu \xi} + \mu u_0 - 1 \right) \right] \, , & \xi \in I_s^+,
\end{cases}
\end{equation}

\begin{equation}\label{eq:beta_leading_order}
	\beta(\xi) =
\begin{cases}
	0 \, , & \xi \in I_s^-;\\
	\frac{\widetilde{\lambda}}{u_0} \omega'(\xi)\, , & \xi \in I_f;\\
	0 \, , & \xi \in I_s^+; \, .
\end{cases},
\end{equation}
We can now assemble the different terms for the solvability condition
\begin{align} \label{eq:smallEigenvalueFredholmCondition_details}
\left\langle \left( \begin{array}{c} u^* \\ v^* \end{array} \right), \left(\begin{array}{c} \alpha \\ \beta \end{array} \right) \right\rangle_{L^2 \times L^2}
= \int_{I_s^- \cup I_{f} \cup I_s^+} u^*(\xi) \alpha(\xi) \, d \xi + \int_{I_s^- \cup I_{f} \cup I_s^+} v^*(\xi) \beta(\xi) \, d \xi
\end{align}
Using that $ f $ is odd and $ g $ is even, which makes $ f' $ even and $ g' $ odd, we get to leading order
{\footnotesize
\begin{align*}
\int_{I_s^-} u^*(\xi) \alpha(\xi) d\xi 
 & = +\varepsilon^6 \mu^2 \left( \frac{6}{5 u_0^3} \right) \int_{I_s^-} e^{+2\varepsilon^2 \mu \xi} \left( - m\tilde{ \lambda}(1-\mu u_0)- -\tilde{f}'(\varepsilon^2 \mu \xi)(1-\mu u_0)- + \tilde{g}'(\varepsilon^2 \mu \xi) [e^{-\varepsilon^2 \mu \xi}+\mu u_0 - 1] \right) \, d\xi \\
 & = +\varepsilon^4 \mu \left( \frac{6}{5 u_0^3} \right) \int_0^{+\infty} e^{-2x} \left( - m\tilde{ \lambda}(1-\mu u_0) -\tilde{f}'(x)(1-\mu u_0) - \tilde{g}'(x) [e^{x}+\mu u_0 - 1] \right) \, dx + h.o.t.\\
 & = - \varepsilon^4 \mu \left( \frac{6}{5 u_0^3} \right) \int_0^{+\infty} e^{-2x} \left( m \tilde{\lambda}(1-\mu u_0) + \tilde{f}'(x)(1-\mu u_0) + \tilde{g}'(x) [e^{x}+\mu u_0 - 1] \right) \, dx + h.o.t.\\
 & = - \varepsilon^4 \mu \left( \frac{6}{5 u_0^3} \right) \left( \frac{1}{2} m(1-\mu u_0) \tilde{ \lambda} + \int_0^{+\infty} e^{-2x} \left(\tilde{f}'(x)(1-\mu u_0) + \tilde{g}'(x) [e^{x}+\mu u_0 - 1] \right) \, dx \right) + h.o.t. \\
\int_{I_s^+} u^*(\xi) \alpha(\xi) d\xi 
 & = - \varepsilon^6 \mu^2 \left( \frac{6}{5 u_0^3} \right) \int_{I_s^+} e^{-2\varepsilon^2 \mu \xi} \left( \ \ m \tilde{\lambda}(1-\mu u_0) + \tilde{f}'(\varepsilon^2 \mu \xi)(1-\mu u_0) + \tilde{g}'(\varepsilon^2 \mu \xi) [e^{+\varepsilon^2 \mu \xi}+\mu u_0 - 1] \right) \, d\xi \\
% & = - \varepsilon^4 \mu \left( \frac{6}{5 u_0^3} \right) \int_0^{+\infty} e^{-2x} \left( m \tilde{\lambda} + \tilde{f}'(x) + \tilde{g}'(x) [e^{x}+\mu u_0 - 1] \right) \, dx \\
 & = - \varepsilon^4 \mu \left( \frac{6}{5 u_0^3} \right) \left( \frac{1}{2} m(1-\mu u_0) \tilde{ \lambda} + \int_0^{+\infty} e^{-2x} \left(\tilde{f}'(x)(1-\mu u_0) + \tilde{g}'(x) [e^{x}+\mu u_0 - 1] \right) \, dx \right) + h.o.t.\\
\int_{I_f} u^*(\xi) \alpha(\xi) \,  d\xi
	& = \int_{I_f} \varepsilon^6 \mu^2 \tilde{g}'(\varepsilon^2 \mu \xi) u_0 d\xi  = \mathcal{O}(\varepsilon^{6-1/2}\mu^2) \\
\int_{I_s^\pm} v^*(\xi) \beta(\xi) d\xi 
	& = h.o.t \\
\int_{I_f} v^*(\xi) \beta(\xi) d\xi 
	& = \int_{I_f}  \tilde{\lambda} \frac{1}{u_0^2} \omega'(\xi)^2 d\xi =  \tilde{\lambda} u_0 \left( \frac{6}{5 u_0^3} \right) + h.o.t. \, .
\end{align*}}
Putting all pieces together, the solvability condition reads
{
\begin{align*}% \label{eq:smallEigenvalueFredholmCondition_more_details}
\left\langle \left( \begin{array}{c} u^* \\ v^* \end{array} \right), \left(\begin{array}{c} \alpha \\ \beta \end{array} \right) \right\rangle_{L^2 \times L^2}
= & \left( \frac{6}{5 u_0^3} \right)\left[ \tilde{\lambda} u_0 -\varepsilon^4 \mu \left( m \tilde{ \lambda}(1-\mu u_0) \right.\right. \\ & \left.\left.+ 2\int_0^{+\infty} e^{-2x} \left(\tilde{f}'(x)(1-\mu u_0) + \tilde{g}'(x) [e^{x}+\mu u_0 - 1] \right) \, dx \right) \right] + h.o.t.= 0 \, ,
\end{align*}
}
which can be rearranged to
\begin{equation}
	\tilde{\lambda} = \frac{2 \varepsilon^4 \mu}{u_0 - \varepsilon^4 \mu m (1 - \mu u_0)} \int_0^{+\infty} e^{-2x} \left(\tilde{f}'(x)(1-\mu u_0) + \tilde{g}'(x) [e^{x}+\mu u_0 - 1] \right) \, dx + h.o.t.
\end{equation}

Since the problem is solved by a regular perturbation approach, the asymptotic analysis may be validated rigorously by classical methods (i.e. by rigorously controlling the higher order terms); alternatively a geometrical approach based on Lin's method may be employed (see e.g.~\cite{modfiedKlausmeier}).

To show Corollary~\ref{cor:small_eigenvalue_double_limit}, we observe that in the double asymptotic limit $\tau:= \varepsilon^4 \mu m \ll 1$ and $\mu \ll 1$, the leading order expression for $\tilde{\lambda}$ becomes
\begin{equation}\label{eq:smallEigenvalueLimits}
	\tilde{\lambda} = \frac{2 \varepsilon^4 \mu}{3} \int_0^\infty e^{-2x} \left( \tilde{f}'(x) + \tilde{g}'(x)[e^x-1]\right)\ dx + h.o.t.
\end{equation}
where we used that $ u_0 = u_0^-(\mu) \rightarrow 3 $ for $ \mu \rightarrow 0 $ (see Corollary~\ref{cor:fg_small} and~\eqref{eq:u0_autonomous_mu_small}).

%%%%%%%%%%%%%%%%%%%%%%%%%%%%%%%%%%%%%%%%%%%%%%%%%%%%%%%%%%%%%%%%%%%%%%%%
%% SUBSECTION (Interpretation of results for ecological applications) %%
%%%%%%%%%%%%%%%%%%%%%%%%%%%%%%%%%%%%%%%%%%%%%%%%%%%%%%%%%%%%%%%%%%%%%%%%

\subsubsection{Interpretation of results for ecological applications}\label{sec:stability_ecology}

Going back to the ecological application, we set $f(x) = h'(x)$ and $g(x) = h''(x)$. Depending on the rate of topographical variation, several different simplifications can be made to Theorem~\ref{theorem:point_spectrum_small}, that allow us to make generic statements about stability of pulse solutions on these terrains.

First, if the topographical changes are small, i.e. when $h = \mathcal{O}(\delta)$, we can write $h(x) = \delta \tilde{h}(x)$ and then~\eqref{eq:smallEigenvalue} can be simplified (via integration by parts):
\begin{corollary}[\underline{small eigenvalue for height function $h$}]
Let the conditions of Theorem~\ref{theorem:point_spectrum_small} be fulfilled. If $\tilde{f}(x) = \tilde{h}'(x)$ and $\tilde{g}(x) = \tilde{h}''(x)$, then~\eqref{eq:smallEigenvalue} becomes
\begin{equation}\label{eq:smallEigenvalue_heightFunction}
	\underline{\lambda}_0 = \frac{2 \delta \tau }{u_0 - \tau (1 - \mu u_0)}
\left[-\mu u_0 \tilde{h}''(0) + \tilde{h}(0)(1-2\mu u_0) + \int_0^\infty \tilde{h}(x) \left( e^{-x} - 4 (1-\mu u_0)e^{-2x}\right) dx \right];
\end{equation}
additionally, in the double asymptotic limit $\tau := \varepsilon^4 \mu m \ll 1$, $\mu \ll1$ this further reduces to
\begin{equation}\label{eq:smallEigenvalue_heightFunctionLimit}
\underline{\lambda}_0 = \frac{2}{3} \delta \tau
\left[\tilde{h}(0) + \int_0^\infty \tilde{h}(x) \left( e^{-x} - 4 e^{-2x}\right) dx.\right] + h.o.t.
\end{equation}
\end{corollary}
\begin{remark}
	Note that $\tilde{h}$ appears in \eqref{eq:smallEigenvalue_heightFunction}, while it does not appear in the original PDE~\eqref{eq:klausmeier_model}, where only its derivatives appear. Thus, increasing $\tilde{h}$ by an additive constant does not affect the system, and in particular should not affect \eqref{eq:smallEigenvalue_heightFunction}. Since $\int_0^\infty \left(e^{-x} - 4 (1-\mu u_0) e^{-2x}\right)\ dx = - (1-2 \mu u_0)$ the result in~\eqref{eq:smallEigenvalue_heightFunction} is indeed not changed when adding a constant to the height function $\tilde{h}$.
\end{remark}

Second, if topographical variation happens only over long spatial scales (i.e. for terrains with weak curvature), we can write $\tilde{h}(x) = \hat{h}(\sigma x)$, where $0 < \sigma \ll 1$ to indicate the large-scale spatial variability. Hence, $\tilde{f}(x) = \sigma \hat{h}'(\sigma x) = \mathcal{O}(\sigma)$ and $\tilde{g}(x) = \sigma^2 \hat{h}''(\sigma x) = \mathcal{O}(\sigma^2)$. Because of the difference in size of $\tilde{f}$ and $\tilde{g}$, the sign of $\underline{\lambda}_0$ can be related to the sign of $\hat{h}''(0)$, i.e. to the local curvature at the location of the pulse.

\begin{corollary}[\underline{small eigenvalue for terrains with weak curvature}] Let the conditions of Theorem~\ref{theorem:point_spectrum_small} be fulfilled. If $\tilde{f}(x) = \sigma \hat{h}'(\sigma x)$ and $\tilde{g}(x) = \sigma^2 \hat{h}''(\sigma x)$ with $0 < \sigma \ll 1$, the leading order expansion of~\eqref{eq:smallEigenvalue} becomes
\begin{equation}\label{eq:smallEigenvalue_highCurvature}
	\underline{\lambda}_0 = \frac{\tau \delta \sigma^2 (1- \mu u_0)}{u_0 - \tau (1- \mu u_0)} \hat{h}''(0);
\end{equation}
additionally, in the double asymptotic limit $\tau := \varepsilon^4 \mu m \ll 1$, $\mu \ll 1$, this further reduces to
\begin{equation}
	\underline{\lambda}_0 = \frac{1}{3} \tau \delta \sigma^2 \hat{h}''(0) + h.o.t.
\end{equation}
Furthermore, it follows that $\mbox{sgn}\ \underline{\lambda}_0 = \mbox{sgn}\ \hat{h}''(0)$, i.e. (vegetation) pulses on hilltops are stable and in valleys are unstable.
\end{corollary}

\begin{proof}
Since $|\tilde{f}'(x)| \gg |\tilde{g}'(x)|$ we can neglect the terms with $\tilde{g}'(x)$ in~\eqref{eq:smallEigenvalue}, thus obtaining
\begin{equation}
	\underline{\lambda}_0 = \frac{2 \tau \delta (1 - \mu u_0)}{u_0 - \tau (1-\mu u_0)} \int_0^\infty \tilde{f}'(x) e^{-2x}\ dx.
\end{equation}
Substitution of $\tilde{f}'(x) = \sigma^2 \hat{h}''(\sigma x)$ and Taylor expanding $\hat{h}''$ as $\hat{h}''(x) = \hat{h}''(0) + \mathcal{O}(\sigma^3)$ immediately yields \eqref{eq:smallEigenvalue_highCurvature}; the rest of the statement follows straightforwardly.
\end{proof}

Third, if topographical variation happens over short spatial scales (i.e. for terrains with strong curvature), we can write $\tilde{h}(x) = \breve{h}\left(x / \sigma\right)$, where $0 < \sigma \ll 1$ to indicate the short spatial scales. Hence, $\tilde{f}(x) = \breve{h}'\left(x / \sigma\right)/\sigma = \mathcal{O}(1/\sigma)$ and $\tilde{g}(x) = \breve{h}''\left(x / \sigma\right)/\sigma^2 = \mathcal{O}(1/\sigma^2)$. Again, the sign of $\underline{\lambda}_0$ can be related to the sign of $\breve{h}''(0)$, though the results are now flipped:

\begin{corollary}[\underline{small eigenvalue for terrains with strong curvature}] Let the conditions of Theorem~\ref{theorem:point_spectrum_small} be fulfilled. If $\tilde{f}(x) = \breve{h}'\left(x / \sigma\right)/\sigma$ and $\tilde{g}(x) = \breve{h}''\left(x / \sigma\right)/\sigma^2$ with $0 < \sigma \ll 1$ and $\breve{h}(y), \breve{h}'(y), \breve{h}''(y) \rightarrow 0$ exponentially fast for $|y| \rightarrow \infty$, the leading (and next-leading) order expansion of~\eqref{eq:smallEigenvalue} becomes
\begin{equation}
	\underline{\lambda}_0 = \frac{2 \tau \delta}{u_0 - \tau (1 - \mu u_0)} \left[ \frac{- \mu u_0}{\sigma^2}\breve{h}''(0) + \left(1 - 2 \mu u_0\right) \breve{h}(0) \right];
\end{equation}
additionally, in the double asymptotic limit $\tau := \varepsilon^4 \mu m \ll 1$, $\mu \ll 1$, this further reduces to
\begin{equation}
	\underline{\lambda}_0 = \frac{2}{3} \tau \delta \breve{h}(0).
\end{equation}
Furthermore, it follows that $\mbox{sgn}\ \underline{\lambda}_0 = - \mbox{sgn}\ \breve{h}''(0)$ when $\mu \neq 0$, i.e. (vegetation) pulses on hilltops are unstable and in valleys are stable; and $\mbox{sgn}\ \underline{\lambda}_0 = \mbox{sgn}\ \breve{h}(0)$ when $\mu = 0$.
\end{corollary}
\begin{proof}
Substitution of $\tilde{h}(x) = \breve{h}(x/\sigma)$ and the use of the transformation $y = x / \sigma$ in \eqref{eq:smallEigenvalue_heightFunction} yields
\begin{equation}
	\underline{\lambda}_0 = \frac{2 \delta \tau}{u_0 - \tau (1-\mu u_0)} \left[ - \frac{\mu u_0}{\sigma^2} \breve{h}''(0) + \left(1-2\mu u_0\right) \breve{h}(0) + \sigma \int_0^\infty \breve{h}(y) \left( e^{-\sigma y} - 4 (1-\mu u_0)e^{-2\sigma y}\right)\ dy. \right]
\end{equation}
Taylor expanding the exponential functions then indicates the integral contributes only at order $\mathcal{O}(\delta \tau \sigma)$. Hence the claimed results follow.
\end{proof}

Thus, the corollaries in this section indicate that -- under certain assumptions on the limiting behavior of the topography function $h$ -- vegetation patterns concentrated on hilltops are stable if the terrain has weak curvature and unstable if the terrain has strong curvature; similarly, patterns concentrated in valleys are unstable for terrains with weak curvature, but they become stable if the terrain has strong curvature. A more in-depth inspection of this phenomena can be found in section~\ref{sec:explicitExamples}, where a few explicit terrain functions $h$ are studied numerically.

%%%%%%%%%%%%%%
%% PULSE LOCATION ODE SECTION
%%%%%%%%%%%%%%
%\newpage
\section{The effect of the small eigenvalue: movement of pulses}\label{sec:pulseLocationODE}

In the previous section we found that, under certain `standard' assumptions on the system's parameters, all large eigenvalues of a homoclinic pulse solution reside to the left of the imaginary axis. Only one small eigenvalue can lead to destabilization of the pulse solution. Since this small eigenvalue is closely related to the translation invariance of the system without spatially varying coefficients, it is possible to study its effects by projecting the whole system unto the corresponding eigenspace.

This derivation enables us to reduce the full PDE dynamics of~\eqref{eq:klausmeier_model} to a simpler ODE that describes the movement of the pulse's location. Concretely, let $P$ denote the location of the center of the pulse. Then the time-evolution of $P$ is given by
\begin{equation}	\label{eq:pulseLocationOde}
	\frac{dP}{dt} = \tau \frac{1}{6} \left[ \tilde{u}_x(P^+)^2 - \tilde{u}_x(P^-)^2\right],
\end{equation}
where the superscripts $\pm$ denote taking the upper respectively lower limit, $\tau := \varepsilon^4 \mu m = \frac{D a^2}{m\sqrt{m}}$ and $\tilde{u}$ solves the differential-algebraic equation
\begin{equation}\label{eq:DAE}
	\left\{
	\begin{array}{rcl}
	\tilde{u}_{xx} + f(x) \tilde{u}_x + g(x) \tilde{u} + 1 - \tilde{u} &=&  0\\
	\tilde{u}(P) &=& \mu u_0 \\
	\tilde{u}_x(P^+) - \tilde{u_x}(P^-) &=& \frac{6}{u_0}
	\end{array}\right.
\end{equation}

We follow~\cite{BD18} and only give a short formal derivation of this PDE-to-ODE reduction, in section~\ref{sec:pdeToOdeReduction}. We refrain from going into the details of (proving) the validity of this reduction. Although the renormalization group approach of~\cite{BDKTP13, doelman2007nonlinear} for semi-strong pulse interactions has not yet been applied to systems with inhomogeneous terms, it can naturally be extended to include these effects. However, it should be noted that, so far, the results and techniques of~\cite{BDKTP13, doelman2007nonlinear} only cover strongly restricted region in parameter space: the general issue of validity of the reduction of semi-strong pulse interactions to finite dimensional settings still largely remains an open question in the field -- see also~\cite{BD18}. As a consequence, we formulate the main results of this section as Propositions and only provide their formal derivations.

Using the pulse location ODE~\eqref{eq:pulseLocationOde} we use formal analysis in section~\ref{sec:odeStability} to present a scheme by which we can determine the stability of the homoclinic pulse patterns of Theorem~\ref{sec:existenceResults} for any functions $f$ and $g$, i.e. without the restriction on their size by which we obtained Theorem~\ref{theorem:point_spectrum_small}; in section~\ref{sec:small_ev_validation} we (formally) validate this scheme by reducing it to the setting of Theorem~\ref{theorem:point_spectrum_small}, i.e. by assuming that $f,g = \mathcal{O}(\delta)$ (with $\delta \ll 1$), and showing that this indeed confirms the results of Theorem~\ref{theorem:point_spectrum_small}. Next, we study a few explicit functions in section~\ref{sec:explicitExamples} -- focusing on what happens when the pulse solution changes stability type. Finally, we briefly consider multi-pulse dynamics in section~\ref{sec:multiPulses}.

\subsection{Formal derivation of pulse location ODE}\label{sec:pdeToOdeReduction}

In this section we formally derive the pulse location ODE~\eqref{eq:pulseLocationOde}. Mathematically, this amounts to tracking perturbations along translational eigenvalues; this approach is sometimes called the `collective coordinate method'. Specifically, in this section, we show
\begin{proposition}\label{theorem:pdeToOdeReduction}
	Let $\varepsilon = \frac{a}{m} \ll 1$, $\tau = \frac{D a^2}{m\sqrt{m}} \ll 1$ and $\mu = \frac{D m \sqrt{m}}{a^2} \leq \mathcal{O}(1)$ (w.r.t. $\varepsilon$). Let $P$ denote the location of the homoclinic pulse's center. Then the evolution of $P$ is described by the pulse location ODE~\eqref{eq:pulseLocationOde}.
\end{proposition}
%\begin{proof}[Formal derivation, cf.~\cite{BD18}]
\par\textit{Formal derivation, cf.~\cite{BD18}.}
We introduce the stretched travelling-wave coordinate 
\begin{equation*}
\xi = \frac{\sqrt{m}}{D}\left(x - P(t)\right) = \frac{\sqrt{m}}{D}\left(x - P(0) - \int_0^t \frac{dP}{dt}(s) ds \right),
\end{equation*}
scale $\frac{dP}{dt} = \frac{D a^2}{m \sqrt{m}} c(t)$ and use scalings~\eqref{eq:scaling_1} to transform~\eqref{eq:klausmeier_model} to get
\begin{equation}\label{eq:innerRegion}
\left\{
	\begin{array}{rcl}
	- \frac{a^2}{m^2} \frac{D m \sqrt{m}}{a^2} \frac{D a^2}{m \sqrt{m}} c(t) u_\xi & = & u_{\xi\xi} - 	\frac{a^2}{m^2} \left[ \frac{D^2m}{a^2} u - \frac{Dm\sqrt{m}}{a^2} f\left( \frac{D}{\sqrt{m}} \xi\right) u_{\xi}- \frac{D^2m}{a^2} g\left( \frac{D}{\sqrt{m}} \xi\right) u  - \frac{D}{\sqrt{m}}  +  u v^2 \right] \\
	- \frac{a^2}{m^2} c(t) v_\xi & = & v_{\xi\xi} - v + u v^2
	\end{array}
\right.
\end{equation}
To find the solution in the fast region $I_f = \left[ - 1 / \sqrt{\varepsilon}, 1 / \sqrt{\varepsilon}\right]$, close to the pulse location, we expand $u$ and $v$ in terms of $\varepsilon$ and look for solution of the form
\begin{equation}
	\begin{cases}
		u & = u_0 + \varepsilon^2 u_1 + \ldots \\
		v & = v_0 + \varepsilon^2 v_1 + \ldots
	\end{cases}
\end{equation}
To leading order~\eqref{eq:innerRegion} is given by
\begin{equation}
\left\{
	\begin{array}{rcl}
		0 & = & u_0'', \\
		0 & = & v_0'' - v_0 + u_0 v_0^2.
	\end{array}
\right.
\end{equation}
Hence we find $u_0$ to be constant and
\begin{equation}
	v_0(\xi) = \frac{3}{2} \frac{1}{u_0} \sech(\xi/2)^2.
\end{equation}
The next order of~\eqref{eq:innerRegion} is
\begin{equation}\label{eq:innerRegionSecondOrder}
\left\{
	\begin{array}{rcl}
		u_1'' & = & u_0 v_0^2, \\
		v_1'' - v_1 +2 u_0 v_0 v_1 & = & - c(t) v_0' - v_0^2 u_1.
	\end{array}
\right.
\end{equation}
It is not a priori clear whether the $v$-equation is solvable; the self-adjoint operator $\mathcal{L} := \partial_\xi^2 - 1 + 2 u_0 v_0$ has a non-empty kernel, since $\mathcal{L} v_0' = 0$, and therefore the inhomogeneous $v$-equation is only solvable when the following Fredholm condition holds
\begin{equation}
	\int_{I_f} c(t) v_0'(\eta)^2 d\eta = - \int_{I_f} v_0(\eta)^2 u_1(\eta) v_0'(\eta) d \eta.
\end{equation}
Upon integrating by parts twice on the right-hand side we obtain
\begin{equation}
	\int_{I_f} c(t) v_0'(\eta)^2 d\eta = - \frac{1}{3} \left[u_1'(\eta)\int_0^\eta v_0(y)^3 dy\right]_{\eta = - 1 / \sqrt{\varepsilon}}^{\eta = + 1 / \sqrt{\varepsilon}} + \frac{1}{3} \int_{I_f} u_1''(\eta) \int_0^\eta v_0(y)^3 dy d\eta + h.o.t.
\end{equation}
Since $v_0$ is an even function, $u_1''$ is an even function and $\eta \mapsto \int_0^\eta v_0(y)^3 dy$ is an odd function. Therefore the last integral vanishes and we obtain
\begin{equation}
	c(t) \int_{I_f} v_0'(\eta)^2 d\eta = \frac{1}{6} \left[ u_1'\left(\frac{1}{\sqrt{\varepsilon}}\right) + u_1'\left(-\frac{1}{\sqrt{\varepsilon}}\right) \right] \int_{I_f} v_0(\eta)^3 d\eta.
\end{equation}
The integrals over the fast field $I_f$ can be approximated by integrals over $\mathbb{R}$, since $v_0$ decays exponentially within fast field. Hence we find
\begin{equation}\label{eq:speed1}
	c(t) = \frac{1}{u_0} \left[ u_1'\left(\frac{1}{\sqrt{\varepsilon}}\right) + u_1'\left(-\frac{1}{\sqrt{\varepsilon}}\right) \right].
\end{equation}	
Finally, it follows from the $u$-equation in~\eqref{eq:innerRegionSecondOrder} that
\begin{equation}
	u_1'\left(\frac{1}{\sqrt{\varepsilon}}\right) - u_1'\left(-\frac{1}{\sqrt{\varepsilon}}\right) = \int_{I_f} u_1''(\eta) d\eta = \int_{I_f} u_0 v_0(\eta)^2 d\eta = \frac{6}{u_0} + h.o.t.
\end{equation}
Combining this with~\eqref{eq:speed1} we obtain
\begin{equation}
	c(t) = \frac{1}{6} \left[ u_1'\left(\frac{1}{\sqrt{\varepsilon}}\right)^2 - u_1'\left(-\frac{1}{\sqrt{\varepsilon}}\right)^2 \right]
\end{equation}
The values of $u_1'(\pm 1/\sqrt{\varepsilon})$ can be matched to the solutions $\hat{u}$ in the slow fields. Careful inspection of the scalings involved reveals $u_1'(\pm 1/\sqrt{\varepsilon}) = \hat{u}_x(P^\pm)$, where $\hat{u}$ satisfies the differential-algebraic equation~\eqref{eq:DAE}. Since $\frac{dP}{dt} = \tau c(t)$ this concludes the proof.
%\end{proof}

\begin{remark}
	Note the link with the notation in section~\ref{sec:existence}: $u_1' = \hat{p}$. See also Remark~\ref{remark:scaling_p_phat}.
\end{remark}

%%%%
%% Stability condition for fixed points of the ODE
%%%%

\subsection{Stability of fixed points of pulse location ODE~\eqref{eq:pulseLocationOde}}\label{sec:odeStability}

The pulse location ODE~\eqref{eq:pulseLocationOde} describes the movement of a pulse over time. In general, for generic functions $f$ and $g$, it is not possible to solve~\eqref{eq:DAE} in closed form, and therefore the pulse location ODE~\eqref{eq:pulseLocationOde} cannot be expressed more explicitly for generic functions $f$ and $g$. Thus, in general,~\eqref{eq:pulseLocationOde} can only be solved numerically -- for instance using the numerical scheme developed in~\cite{BD18}. Moreover, for generic $f$ and $g$ fixed points of~\eqref{eq:pulseLocationOde} can only be obtained numerically. However, when $f$ and $g$ obey the symmetry assumptions (A2), one can readily obtain that $P_* = 0$ is a fixed point. It is possible to determine the stability of fixed points using~\eqref{eq:pulseLocationOde} via direct numerics, but this can be rather time-intensive and is prone to errors close to bifurcation points. Instead, it is better to first use asymptotic expansions to derive a stability condition that can be checked (numerically) more easily.

\begin{proposition}\label{theorem:stabilityInOde}
	Let the conditions of Proposition~\ref{theorem:pdeToOdeReduction} be satisfied, let $\mu \ll 1$ and let $P_*$ be a fixed point of~\eqref{eq:pulseLocationOde}. Then, the eigenvalue $\underline{\lambda}$ -- where $\underline{\lambda} = m \lambda$, see~\eqref{eq:scaling_eigenvalue} -- corresponding to the pulse solution with a pulse located at the fixed point $P_*$ is given by
	\begin{equation}\label{eq:eigenvalueFixedPointOde}
		\underline{\lambda} = \frac{\tau}{6} \left\{ 2 \tilde{u}'(P_*^+)\left[ \tilde{u}''(P_*^+)+\tilde{w}'(P_*^+)\right] - 2 \tilde{u}'(P_*^-) \left[ \tilde{u}''(P_*^-) + \tilde{w}'(P_*^-) \right] \right\}.
	\end{equation}
Here $\tilde{u}$ and $\tilde{w}$ solve the coupled ODE system
\begin{equation}\label{eq:stabilityODEs}
\left\{
	\begin{array}{rcl}
		0 & = & \tilde{u}'' + f \tilde{u}' + g \tilde{u} - \tilde{u} + 1, \\
		0 & = & \tilde{w}'' + f \tilde{w}' + g \tilde{w} - \tilde{w}, \\
	\tilde{u}(P_*) & = & 0, \\
	\tilde{w}(P_*^\pm) & = & - \tilde{u}'(P_*^\pm).
	\end{array}
\right.
\end{equation}

\end{proposition}

\begin{remark}
	If $f$ and $g$ satisfy the symmetry assumption (A2) and $P_*$ is located at the point of symmetry, i.e. $P_* = 0$, then symmetry forces $\tilde{u}'(P_*^+) = - \tilde{u}'(P_*^-)$, $\tilde{u}''(P_*^+) = \tilde{u}''(P_*^-)$ and $\tilde{w}'(P_*^+) = \tilde{w}(P_*^-)$. Therefore,~\eqref{eq:eigenvalueFixedPointOde} reduces to
\begin{equation}\label{eq:eigenvalueFixedPointOdeSymmetric}
	\underline{\lambda} = \frac{2\tau}{3} \tilde{u}'(P_*^+)\left[ \tilde{u}''(P_*^+)+\tilde{w}'(P_*^+)\right].
\end{equation}
\end{remark}

\begin{remark}
	The condition $\mu \ll 1$ in Theorem~\eqref{theorem:stabilityInOde} is not strictly necessary. When this condition holds, the differential-algebraic system~\eqref{eq:DAE} simplifies to a normal boundary value problem, since $\tilde{u}(P) = 0$ to leading order. However, when $\mu = \mathcal{O}(1)$ (w.r.t. $\varepsilon$) the procedure explained below is still applicable and one can derive a similar result; only this time, $u_0$ in~\eqref{eq:DAE} needs to be expanded as well and $\tilde{u}$ and $\tilde{w}$ satisfy the coupled differential-algebraic system
\begin{equation}
\left\{
	\begin{array}{rcl}
		0 & = & \tilde{u}'' + f \tilde{u}' + g \tilde{u} - \tilde{u} + 1, \\
		0 & = & \tilde{w}'' + f \tilde{w}' + g \tilde{w} - \tilde{w}, \\
	\tilde{u}(P_*) & = & \mu u_{0}, \\
	\tilde{w}(P_*^\pm) & = & - \tilde{u}'(P_*^\pm) + \mu w_{0},\\
	\tilde{u}'(P_*^+) - \tilde{u}'(P_*^-) & = & \frac{6}{u_{0}}, \\
	\tilde{w}'(P_*^+) - \tilde{w}'(P_*^-) & = & \frac{6 w_{0}}{u_{0}^2} + \tilde{u}''(P_*^-) - \tilde{u}''(P_*^+).
	\end{array}
\right.
\end{equation}

\end{remark}

%\begin{proof}[Formal derivation]
\par\textit{Formal derivation.}
	To find the eigenvalue $\underline{\lambda}$ we need to evaluate the derivative of the right-hand side of~\eqref{eq:pulseLocationOde} at the fixed point $P_*$. That is,
\begin{align}
	\underline{\lambda} & = \frac{d}{dP} \left[ \frac{\tau}{6} \left( \tilde{u}'(P^+)^2 - \tilde{u}'(P^-)^2\right)\right]_{P = P_*} \nonumber\\&= \frac{\tau}{6} \left[ 2 \tilde{u}'(P_*^+) \left(\frac{d}{dP}\tilde{u}'(P^+)\right)_{P=P_*} - 2 \tilde{u}'(P_*^-) \left( \frac{d}{dP} \tilde{u}'(P^-)\right)_{P=P_*}\right].\label{eq:ODEeigenvalue1}
\end{align}
By definition of the derivative
\begin{equation}\label{eq:derivativeOfUx}
	\frac{d}{dP} \left[ \tilde{u}'(P^\pm)\right] = \lim_{\phi \rightarrow 0} \frac{\tilde{u}_\phi'( (P+\phi)^\pm) - \tilde{u}'(P^\pm)}{\phi},
\end{equation}
where $\tilde{u}_\phi$ solves~\eqref{eq:DAE} with every $P$ replaced by $P+\phi$. For small $\phi$, $\tilde{u}_\phi$ can be related to $\tilde{u}$ via a regular expansion. Specifically, let $|\phi| \ll 1$, and expand $\tilde{u}_\phi = \tilde{u} + \phi \tilde{w}$. Substitution in~\eqref{eq:DAE} and careful bookkeeping readiliy shows that $\tilde{u}$ and $\tilde{w}$ satisfy~\eqref{eq:stabilityODEs}. Finally, upon substituting the expansion for $\tilde{u}_\phi$ into~\eqref{eq:derivativeOfUx} and the use of a Taylor expansion we obtain
\begin{align*}
	\frac{d}{dP} \left[ \tilde{u}'(P^\pm)\right] &= \lim_{\phi \rightarrow 0} \frac{\tilde{u}'((P+\phi)^\pm) + \phi \tilde{w}'((P+\phi)^\pm) - \tilde{u}(P^\pm)}{\phi} = \lim_{\phi \rightarrow 0} \frac{ \tilde{u}'(P^\pm) + \phi \tilde{u}''(P^\pm) + \phi \tilde{w}'(P^\pm) - \tilde{u}'(P^\pm)}{\phi} \\ & = \tilde{u}''(P^\pm) + \tilde{w}'(P^\pm).
\end{align*}
Finally, substitution into~\eqref{eq:ODEeigenvalue1} gives~\eqref{eq:eigenvalueFixedPointOde}.
%\end{proof}

\subsection{Small eigenvalue in case of small spatially varying coefficients}\label{sec:small_ev_validation}

As an example of the use of Proposition~\ref{theorem:stabilityInOde}, in this section we use Proposition~\ref{theorem:stabilityInOde} to give another proof for Theorem~\ref{theorem:point_spectrum_small} in the limit $\mu \ll 1$. This not only shows the applicability of Proposition~\ref{theorem:stabilityInOde} but especially the relevance of the pulse location ODE~\eqref{eq:pulseLocationOde}. Moreover, it also provides a confirmation of the validity of the formal results in this section.

%\begin{proof}[Alternative formal derivation of Theorem~\ref{theorem:point_spectrum}(ii) for $\mu \ll 1$]
\par\textit{Alternative formal derivation of Theorem~\ref{theorem:point_spectrum_small} for $\mu \ll 1$.}
	Since $f$ and $g$ satisfy the symmetry assumption (A2), the eigenvalue $\underline{\lambda}$ is given by~\eqref{eq:eigenvalueFixedPointOdeSymmetric}. Therefore, it suffices to only look at the solutions $\tilde{u}$ and $\tilde{w}$ to~\eqref{eq:stabilityODEs} for $x > 0$. Since $f,g = \mathcal{O}(\delta)$ with $\delta \ll 1$, we use regular expansions for $\tilde{u}$ and $\tilde{w}$; that is, we set
\begin{align*}
	\tilde{u} & = \tilde{u}_{0} + \delta \tilde{u}_{1} + \ldots, \\
	\tilde{w} & = \tilde{w}_{0} + \delta \tilde{w}_{1} + \ldots.
\end{align*}
Substitution in~\eqref{eq:stabilityODEs} gives at leading order
\begin{equation}
\left\{
	\begin{array}{rcl}
		0 & = & \tilde{u}_{0}'' - \tilde{u}_{0} + 1, \\
		0 & = & \tilde{w}_{0}'' - \tilde{u}_{1}, \\
	\tilde{u}_{0}(0) & = & 0, \\
	\tilde{w}_{0}(0^+) & = & - \tilde{w}_{0}'(0^+);
	\end{array}
\right.
\end{equation}
and at the next order, $\mathcal{O}(\delta)$, we find
\begin{equation}
\left\{
	\begin{array}{rcl}
		\tilde{u}_{1}'' - \tilde{u}_{1} & = & - \tilde{f} \tilde{u}_{0}' - \tilde{g} \tilde{u}_{0}, \\
		\tilde{w}_{1}'' - \tilde{w}_{1} & = & - \tilde{f} \tilde{w}_{0}' - \tilde{g} \tilde{w}_{0}, \\
		\tilde{u}_{1}(0) & = & 0, \\
		\tilde{w}_{1}(0^+) & = & - \tilde{u}_{1}'(0^+).
	\end{array}
\right.
\end{equation}
Using the usual techniques to solve these ODEs, one can verify that
\begin{align}
	\tilde{u}_{0}(x) & = 1 - e^{-x} \\
	\tilde{u}_{1}(x) & = \frac{1}{2} e^x \int_x^\infty F(z)e^{-z}dz - \frac{1}{2} \int_0^\infty F(z) e^{-z} dz + \frac{1}{2} e^{-x} \int_0^x F(z) e^{z} dz \\
	\tilde{w}_{0}(x) & = - e^{-x} \\
	\tilde{w}_{1}(x) & = \frac{1}{2}e^{x} \int_x^\infty G(z)e^{-z}dz - \frac{1}{2}e^{-x} \int_0^\infty G(z)e^{-z} dz + \frac{1}{2} e^{-x}\int_0^x G(z) e^{z} dz - e^{-x} \int_0^\infty F(z)e^{-z} dz
\end{align}
where
\begin{align}
	F(z) &:= \tilde{f}(z)e^{-z} + \tilde{g}(z)(1-e^{-z}), \\
	G(z) &:= \tilde{f}(z)e^{-z} - \tilde{g}(z)e^{-z}.
\end{align}
Substitution of these expansions in~\eqref{eq:eigenvalueFixedPointOdeSymmetric} then yields
\begin{align*}
\underline{\lambda}
	& = \frac{2}{3} \tau \left[ \tilde{u}_{0}'(0) + \delta \tilde{u}_{1}'(0)\right]\left[ \tilde{u}_{0}''(0) + \delta \tilde{u}_{1}''(0) + \tilde{w}_{0}'(0) + \delta \tilde{w}_{1}'(0) \right] + \mathcal{O}(\delta^2) \\
	& = \frac{2}{3} \tau \left[ 1 + \delta \int_0^\infty F(z)e^{-z} e^{-z} \right] \left[-1 + 1 + \delta \int_0^\infty \left( F(z) + G(z) \right) e^{-z} dz\right] + \mathcal{O}(\delta^2) \\
	& = \frac{2}{3} \delta \tau \int_0^\infty \left( F(z) + G(z) \right) e^{-z} dz + \mathcal{O}(\delta^2) \\
	& = \frac{2}{3} \delta \tau \int_0^\infty \left( 2 \tilde{f}(z)e^{-2z} + \tilde{g}(z)[1-2e^{-z}]e^{-z} \right) dz + \mathcal{O}(\delta^2)\\
	& = \frac{2}{3} \delta \tau \int_0^\infty \left( \tilde{f}'(z)e^{-2z} + \tilde{g}'(z)(1-e^{-z})e^{-z} \right) dz + \mathcal{O}(\delta^2).
\end{align*}
Finally, we note that the eigenvalue has been rescaled as $\underline{\lambda} = m \lambda$ in Theorem~\ref{theorem:point_spectrum}. Since $\tau / m = \varepsilon^4 \mu$ and $u_0 = 3$ in the limit $\mu \ll 1$, we have indeed recovered \eqref{eq:smallEigenvalueLimits}, i.e. Theorem~\ref{theorem:point_spectrum_small}, in the case $\mu \ll 1$.
%\end{proof}

\subsection{Examples of stationary single-pulse solutions}\label{sec:explicitExamples}
In this section, we study a few explicit functions $f$ and $g$; in all examples we specify a function $h$ and take $f = h'$, $g = h''$. Not all functions we consider here limit to $0$ as $|x| \rightarrow \infty$; that is, some violate assumption (A4). Therefore, these examples also form an outlook, illustrating how the results in this paper are expected to extend beyond the imposed assumptions on functions $f$ and $g$. Specifically, we consider the following four examples:
\begin{itemize}
	\item[(i)] $h(x) = A e^{-Bx^2}$, ($A \in \mathbb{R}$, $B > 0$);
	\item[(ii)] $h(x) = A \sech(Bx)$, ($A \in \mathbb{R}$, $B > 0$);
	\item[(iii)] $h(x) = A \cos(Bx)$, ($A \in \mathbb{R}$, $B > 0$);
	\item[(iv)] $h(x) = -2 \ln(\cosh(\beta x))$, ($\beta > 0$).
\end{itemize}
Note that $\lim_{|x| \rightarrow \infty} f(x),g(x) = 0$ in cases (i)--(ii), which therefore satisfy assumption (A4). In case (iii) $f$ and $g$ are periodic when $|x| \gg 1$; in case (iv) $f$ and $g$ do have well-defined (though non-zero) limits for $|x| \rightarrow \infty$.

\begin{remark}
Note that $A > 0$ in (i)--(ii) corresponds to `hill-like' topographies and $A < 0$ to `valley-like' topographies. The value of $B$ in (i)--(iii) is a measure of the curvature of the terrain; the higher the value of $B$, the stronger the curvature of the terrain modeled by the function $h$.
\end{remark}

\begin{figure}
	\centering
		\begin{subfigure}[t]{0.27\textwidth}
			\centering
			\includegraphics[width = \textwidth]{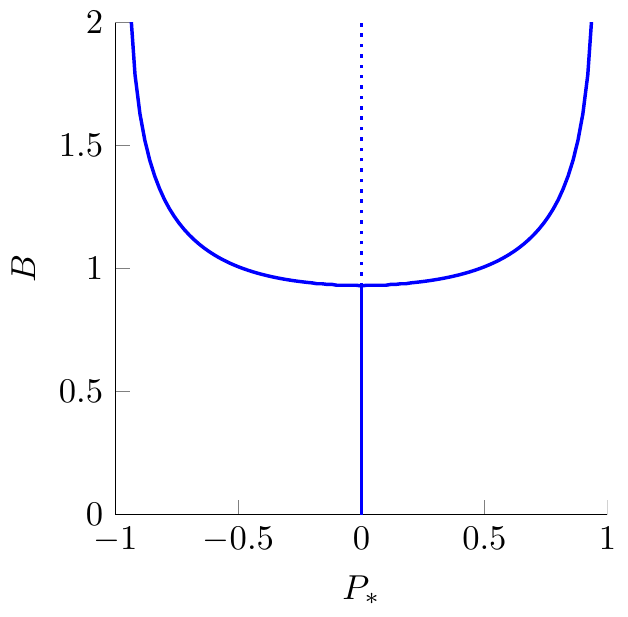}
			\caption{Bifurcation diagram for $A = 1$}
		\end{subfigure}
~
		\begin{subfigure}[t]{0.27\textwidth}
			\centering
			\includegraphics[width = \textwidth]{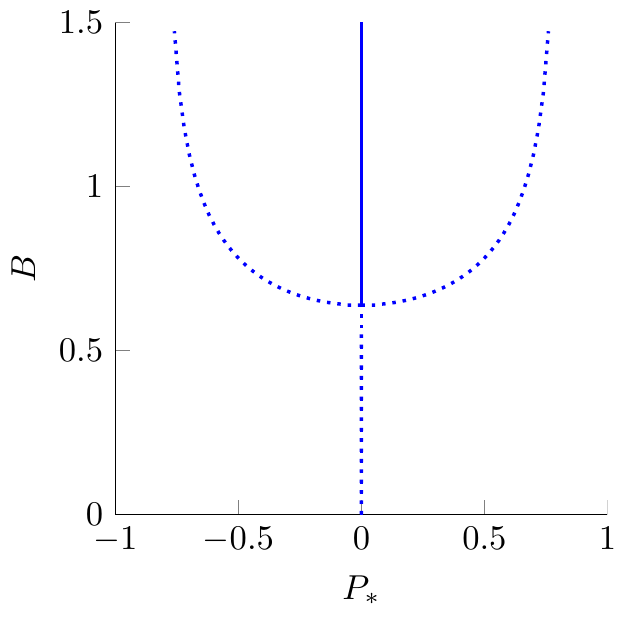}
			\caption{Bifurcation diagram for$A = -1$}
		\end{subfigure}
~
		\begin{subfigure}[t]{0.27\textwidth}
			\centering
			\includegraphics[width = \textwidth]{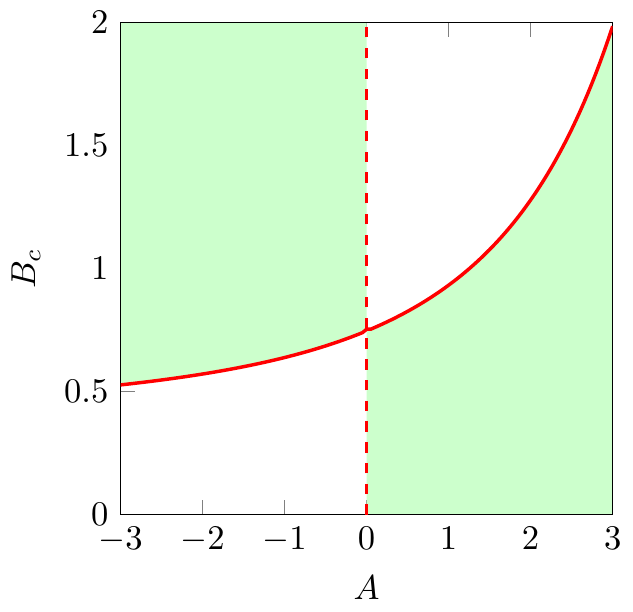}
			\caption{Bifurcation value $B_c(A)$}
		\end{subfigure}
\\
		\begin{subfigure}[t]{0.235\textwidth}
			\centering
			\includegraphics[width = \textwidth]{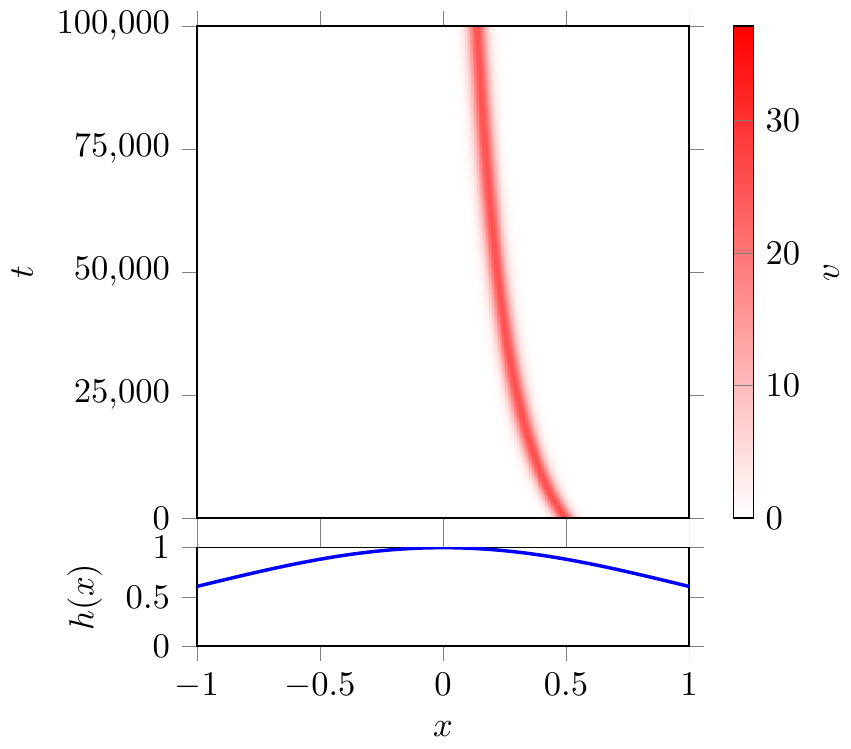}
			\caption{$A = 1$, $B = 0.5$}
		\end{subfigure}
		\begin{subfigure}[t]{0.235\textwidth}
			\centering
			\includegraphics[width = \textwidth]{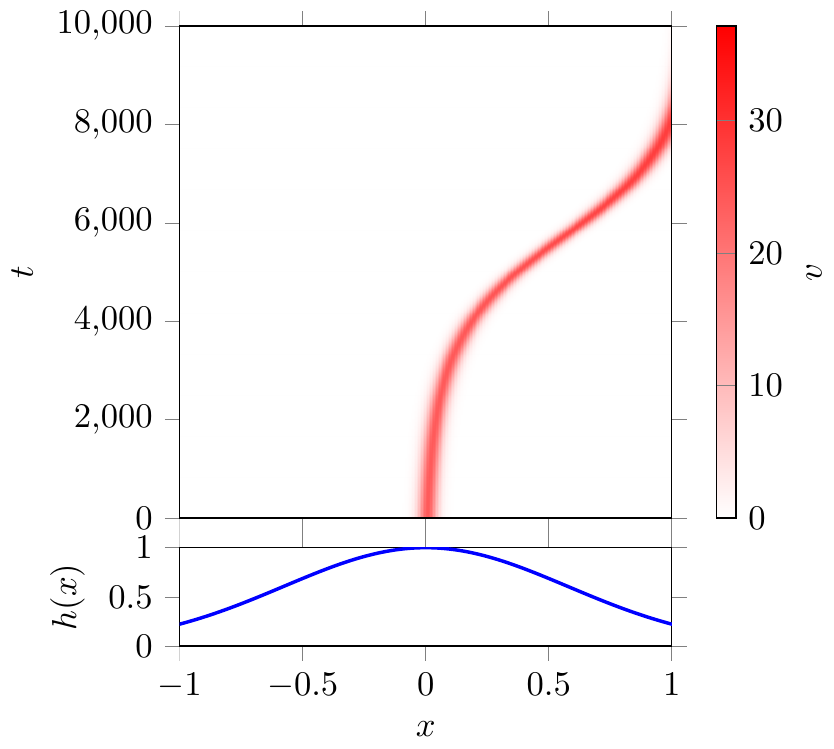}
			\caption{$A = 1$, $B = 1.5$}
		\end{subfigure}
		\begin{subfigure}[t]{0.235\textwidth}
			\centering
			\includegraphics[width = \textwidth]{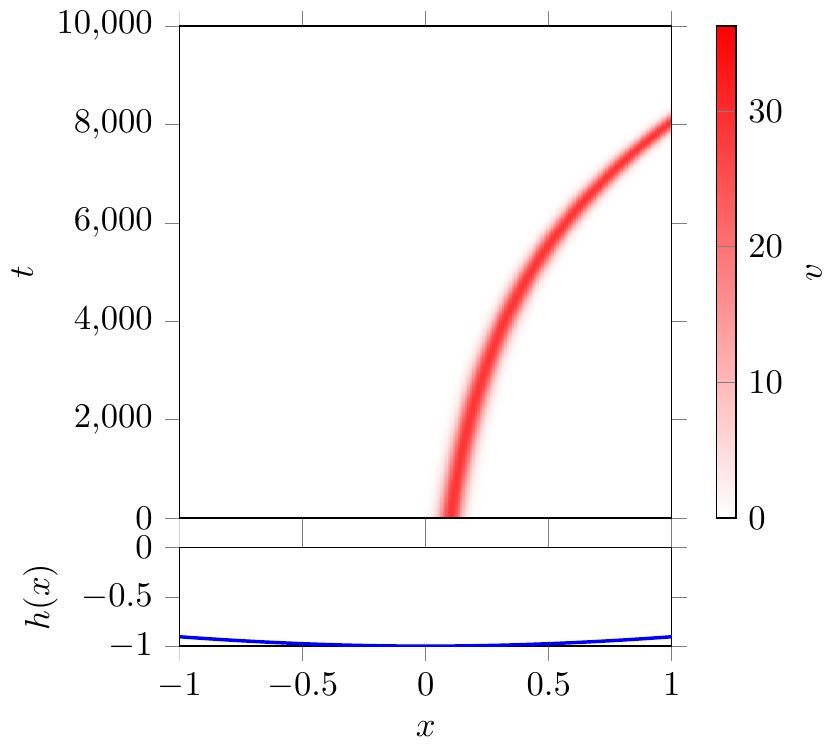}
			\caption{$A = -1$, $B = 0.1$}
		\end{subfigure}
		\begin{subfigure}[t]{0.235\textwidth}
			\centering
			\includegraphics[width = \textwidth]{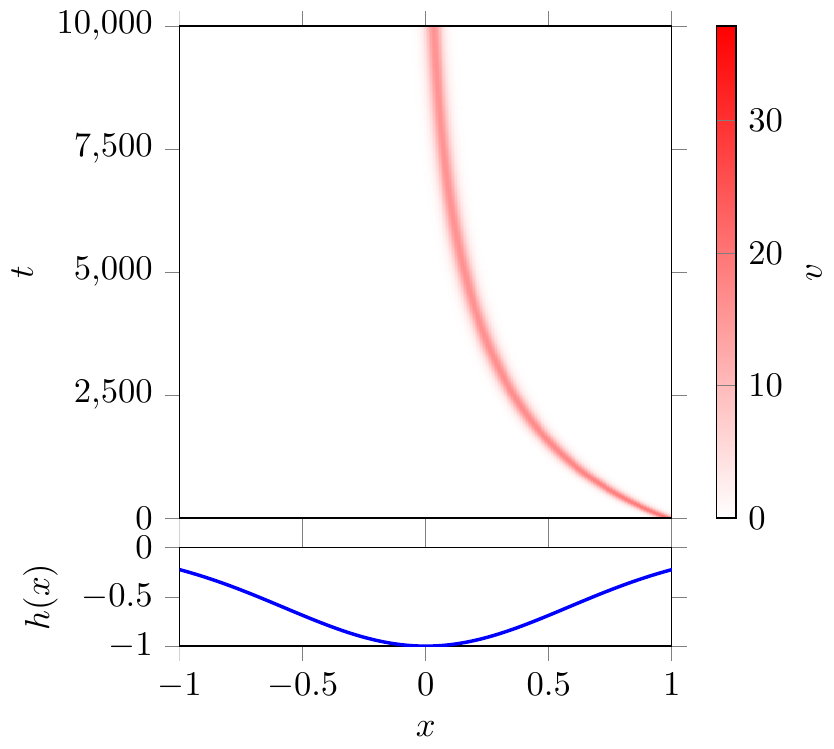}
			\caption{$A = -1$, $B = 1.5$}
		\end{subfigure}
	\caption{Numerical results for $h(x) = A e^{-Bx^2}$. Shown are bifurcation diagrams for $A = 1$ (a) and $A = -1$ (b), the bifurcation value $B_c(A)$ of the pitchfork bifurcation (c), and (parts of) various simulations of the full PDE illustrating the change of stability along with a plot of the function $h(x)$ (d-g). The green areas in (c) indicate the parameter region in which the fixed point $P_* = 0$ is stable. In the PDE simulations we have used parameters $a = 0.5$, $m = 0.45$, $D = 0.01$ and taken $x \in [-30,30]$.  }
	\label{fig:numericalExponential}
\end{figure}

Using the pulse location ODE~\eqref{eq:pulseLocationOde} and Proposition~\ref{theorem:stabilityInOde}, we have tracked the fixed points and their stability for these examples in the limit $\mu \ll 1$, using numerical continuation methods. The resulting bifurcation diagrams for (i) are shown in Figure~\ref{fig:numericalExponential}(a-b), for (ii) in Figure~\ref{fig:numericalSech}(a-b) and for (iii) in Figure~\ref{fig:numericalCos}(a). In all of these cases, we find fixed points at the point of symmetry, corroborating the results in section~\ref{sec:existence}. For small $B$ values -- i.e. for weak curvature topographies -- the stability of these fixed points is determined by the sign of $A$: $A > 0$ leads to stable and $A < 0$ to unstable fixed points -- corroborating previous intuition indicating that pulses migrate in uphill direction~\cite{siteur2014beyond, SD17, BD18}. However, for sufficiently large values of $B$ --i.e. topographies with strong curvature -- the stability of those fixed points changes through a pitchfork bifurcation and new behavior is observed. In case (iii) this even leads to the possibility that both the tops ($BP = 0$) as well as the valleys ($BP = \pm \pi$) form stable fixed points of~\eqref{eq:pulseLocationOde}. The bifurcation value of the pitchfork bifurcation, $B_c(A)$, depends on the value of $A$. Using numerical continuation methods we also tracked this value; the results are in Figures~\ref{fig:numericalExponential}(c), \ref{fig:numericalSech}(c) and \ref{fig:numericalCos}(b) (for topographies (i), (ii) and (iii)).

\begin{remark}
	Theorem~\ref{theorem:point_spectrum_small}, and in particular~\eqref{eq:smallEigenvalueLimits} and~\eqref{eq:smallEigenvalue_heightFunctionLimit}, provide a leading order analytic expression for $B_c(0)$. Evaluating these yields  $B_c(0) \approx 0.75$ (i), $B_c(0) \approx 1.23$ (ii) and $B_c(0) = \sqrt{2}$ (iii), which is confirmed by the numerical continuation that indicate $B_c(0) \approx 0.75$ (i), $B_c(0) \approx 1.24$ (ii) and $B_c(0) = 1.43$ (iii). Note that $A = 0$ is, indeed, just the flat terrain $h(x) \equiv 0$; however, these results for $A = 0$ should be interpreted to apply to `small' topographical functions only, where $A$ is asymptotically small.
\end{remark}

\begin{figure}
	\centering
		\begin{subfigure}[t]{0.27\textwidth}
			\centering
			\includegraphics[width = \textwidth]{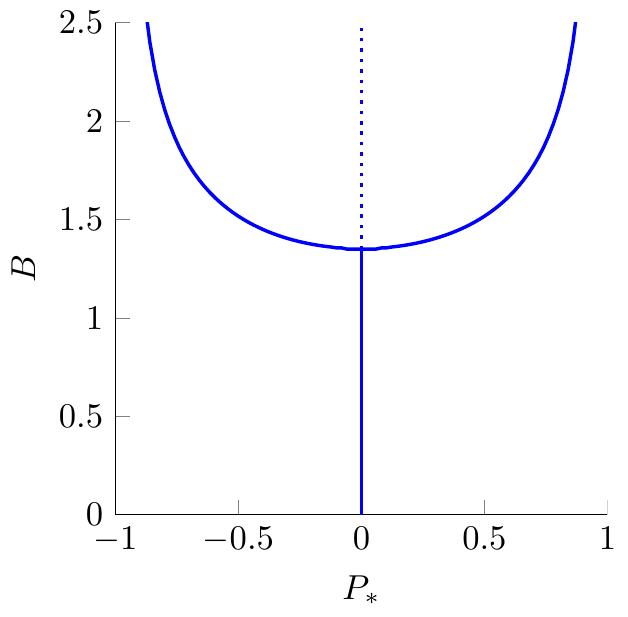}
			\caption{Bifurcation diagram for $A = 1$}
		\end{subfigure}
~
		\begin{subfigure}[t]{0.27\textwidth}
			\centering
			\includegraphics[width = \textwidth]{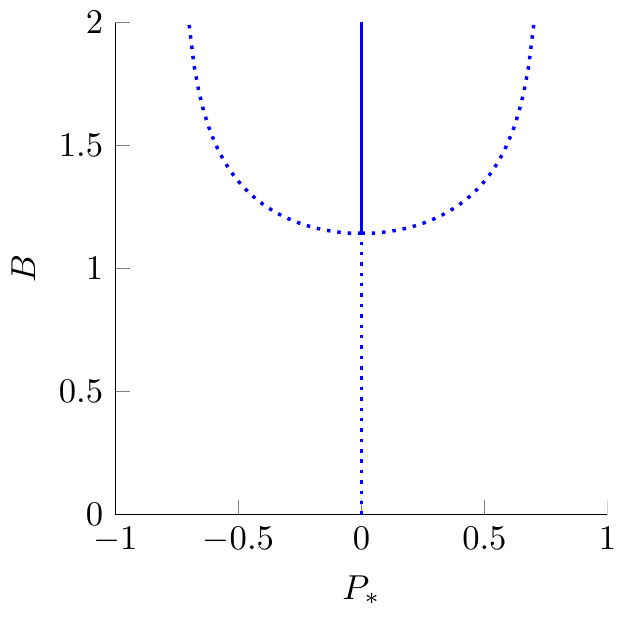}
			\caption{Bifurcation diagram for$A = -1$}
		\end{subfigure}
~
		\begin{subfigure}[t]{0.27\textwidth}
			\centering
			\includegraphics[width = \textwidth]{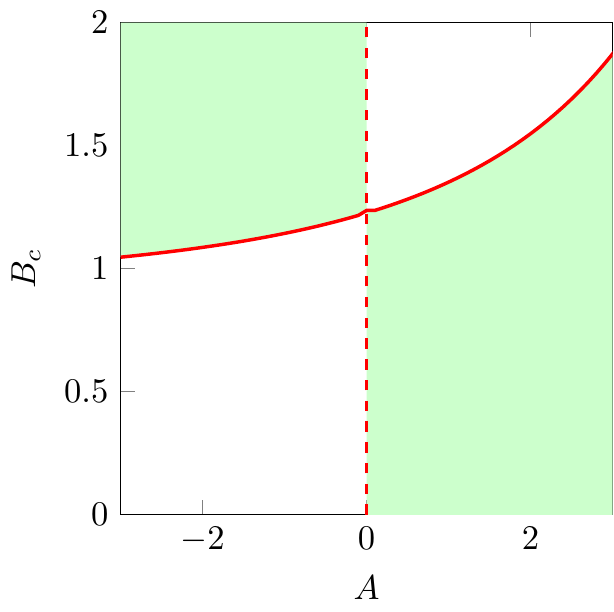}
			\caption{Bifurcation value $B_c(A)$}
		\end{subfigure}
\\
		\begin{subfigure}[t]{0.235\textwidth}
			\centering
			\includegraphics[width = \textwidth]{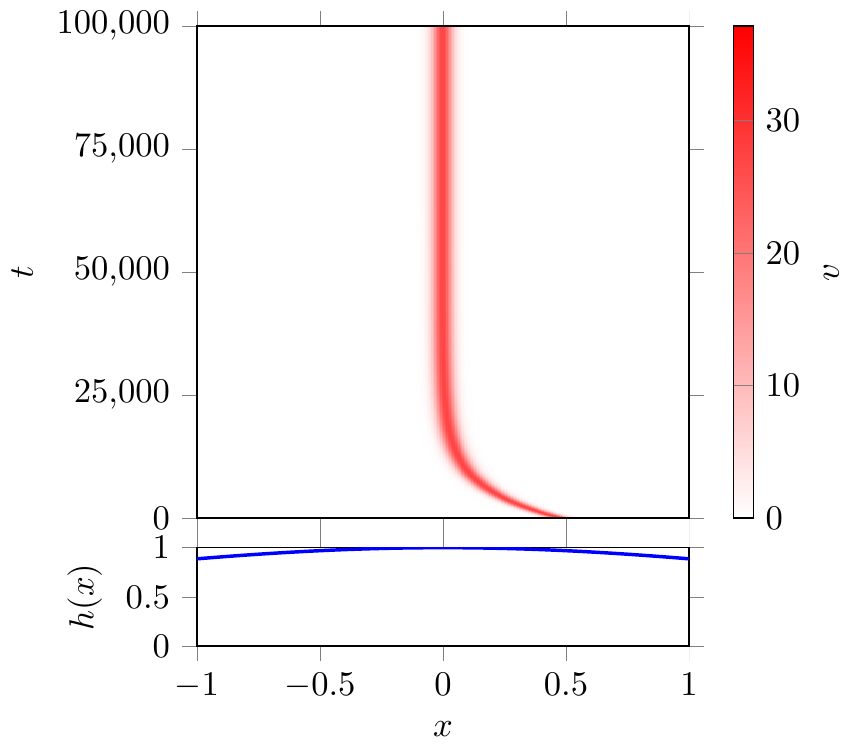}
			\caption{$A = 1$, $B = 0.5$}
		\end{subfigure}
		\begin{subfigure}[t]{0.235\textwidth}
			\centering
			\includegraphics[width = \textwidth]{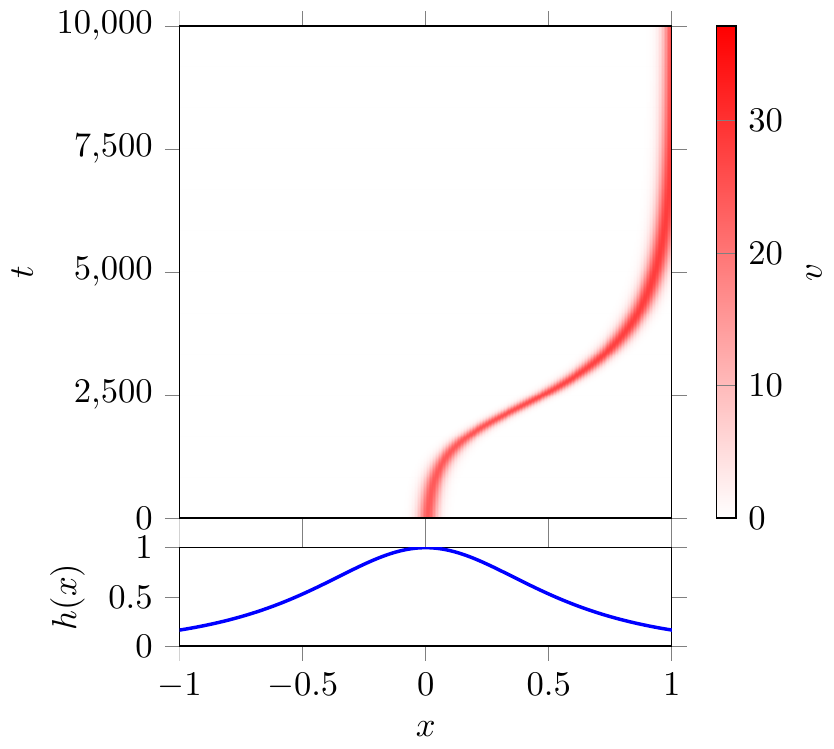}
			\caption{$A = 1$, $B = 2.5$}
		\end{subfigure}
		\begin{subfigure}[t]{0.235\textwidth}
			\centering
			\includegraphics[width = \textwidth]{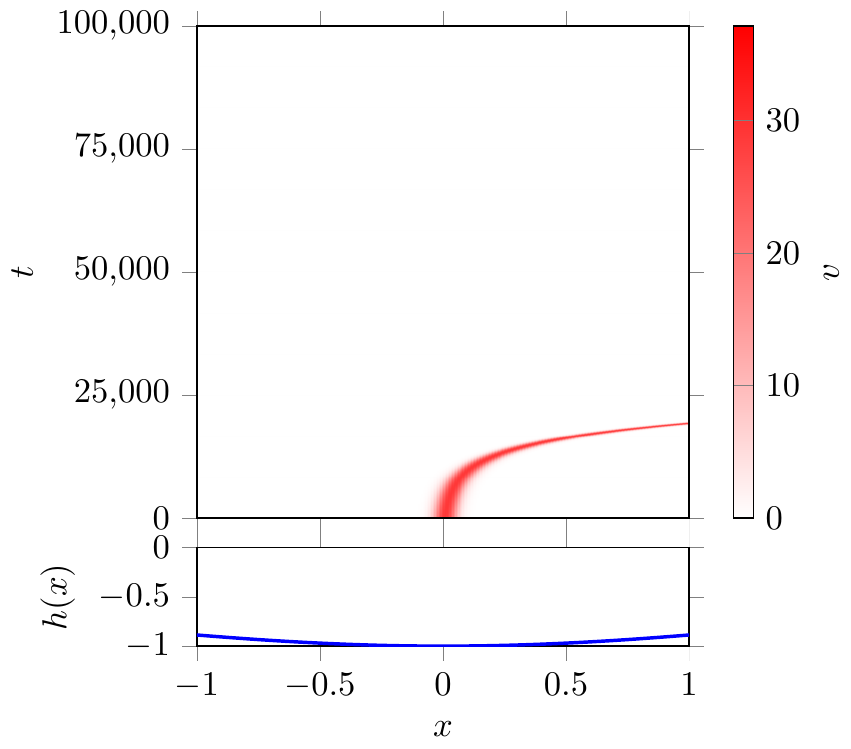}
			\caption{$A = -1$, $B = 0.5$}
		\end{subfigure}
		\begin{subfigure}[t]{0.235\textwidth}
			\centering
			\includegraphics[width = \textwidth]{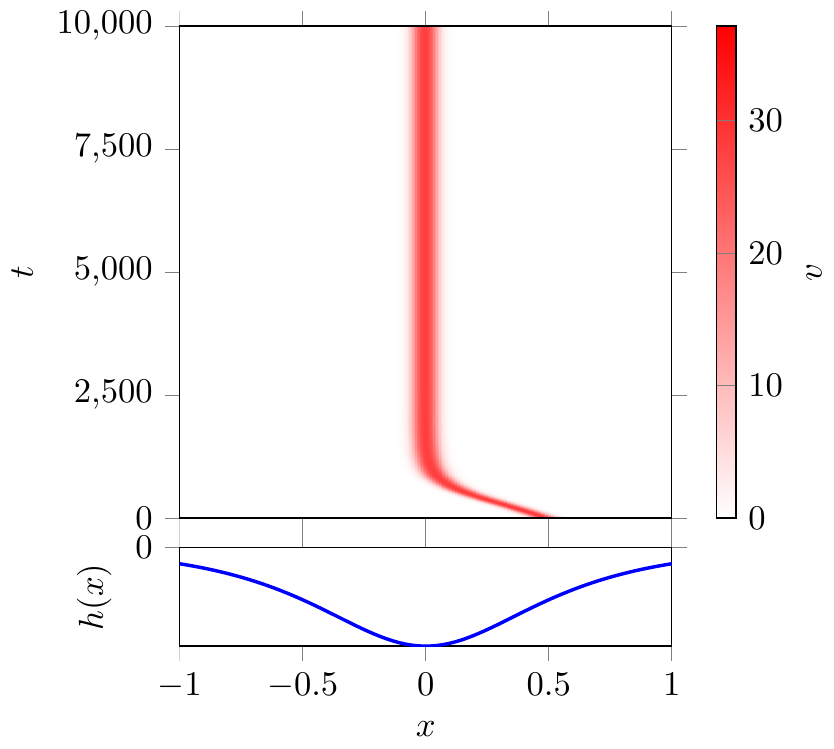}
			\caption{$A = -1$, $B = 2.5$}
		\end{subfigure}
	\caption{Numerical results for $h(x) = A \sech(Bx)$. Shown are bifurcation diagrams (solid for stable; dashed for unstable fixed points) for $A = 1$ (a) and $A = -1$ (b), the bifurcation value $B_c(A)$ of the pitchfork bifurcation (c), and (parts of) various simulations of the full PDE illustrating the change of stability along with a plot of the function $h(x)$ (d-g). The green areas in (c) indicate the parameter region in which the fixed point $P_* = 0$ is stable. In the PDE simulations we have used parameters $a = 0.5$, $m = 0.45$, $D = 0.01$ and taken $x \in [-30,30]$. }
	\label{fig:numericalSech}
\end{figure}

\begin{figure}
	\centering
		\begin{subfigure}[t]{0.54\textwidth}
			\centering
			\includegraphics[height = 0.5\textwidth]{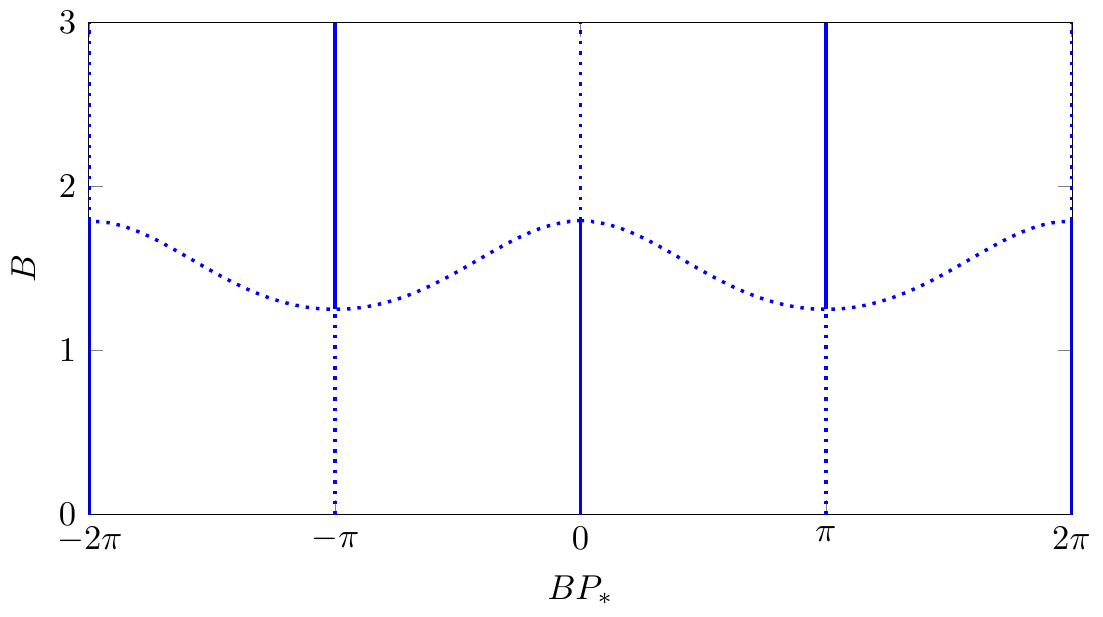}
			\caption{Bifurcation diagram for $A = 1$}
		\end{subfigure}
~
		\begin{subfigure}[t]{0.27\textwidth}
			\centering
			\includegraphics[width = \textwidth]{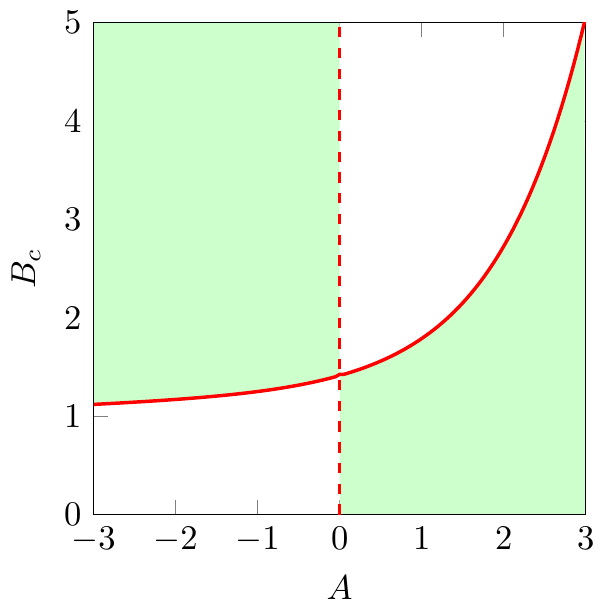}
			\caption{Bifurcation value $B_c(A)$}
		\end{subfigure}
\\
		\begin{subfigure}[t]{0.245\textwidth}
			\centering
			\includegraphics[width = \textwidth]{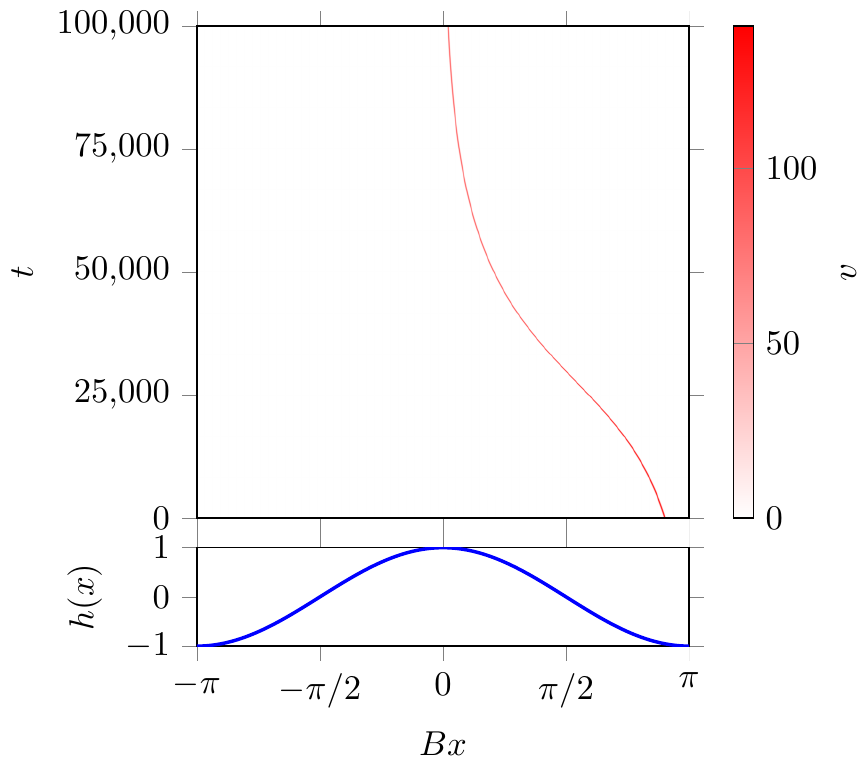}
			\caption{$A = 1$, $B = 1$}
		\end{subfigure}
		\begin{subfigure}[t]{0.235\textwidth}
			\centering
			\includegraphics[width = \textwidth]{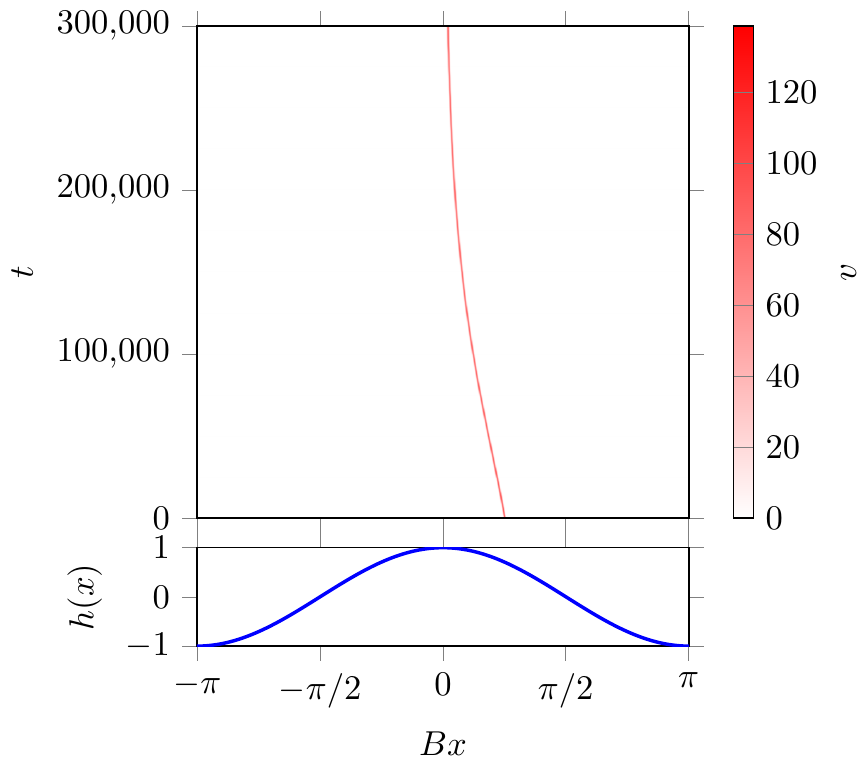}
			\caption{$A = 1$, $B = 1.5$}
		\end{subfigure}
		\begin{subfigure}[t]{0.235\textwidth}
			\centering
			\includegraphics[width = \textwidth]{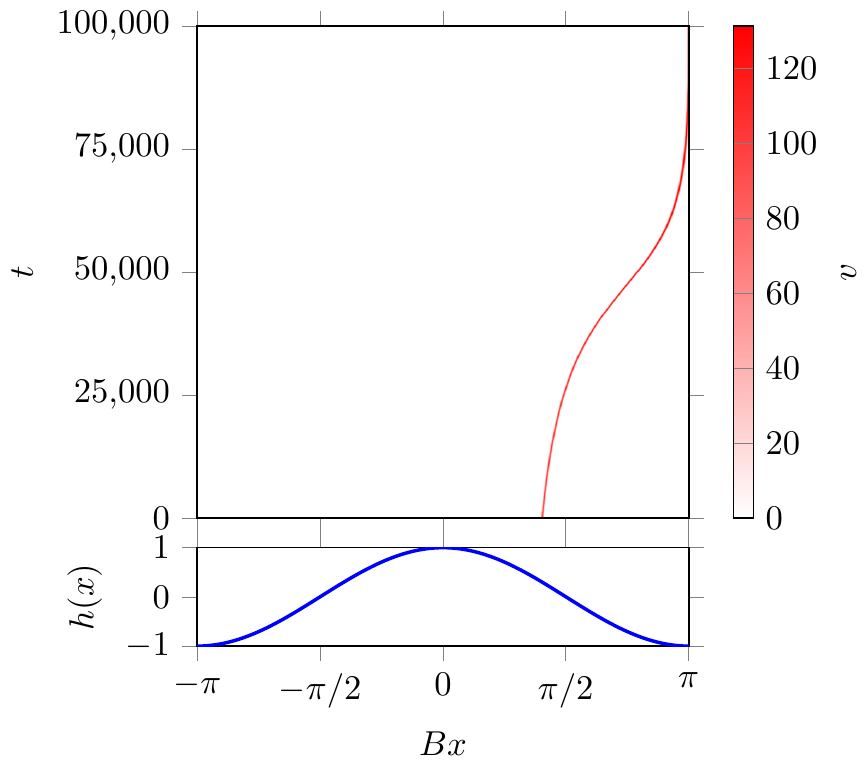}
			\caption{$A = 1$, $B = 1.5$}
		\end{subfigure}
		\begin{subfigure}[t]{0.235\textwidth}
			\centering
			\includegraphics[width = \textwidth]{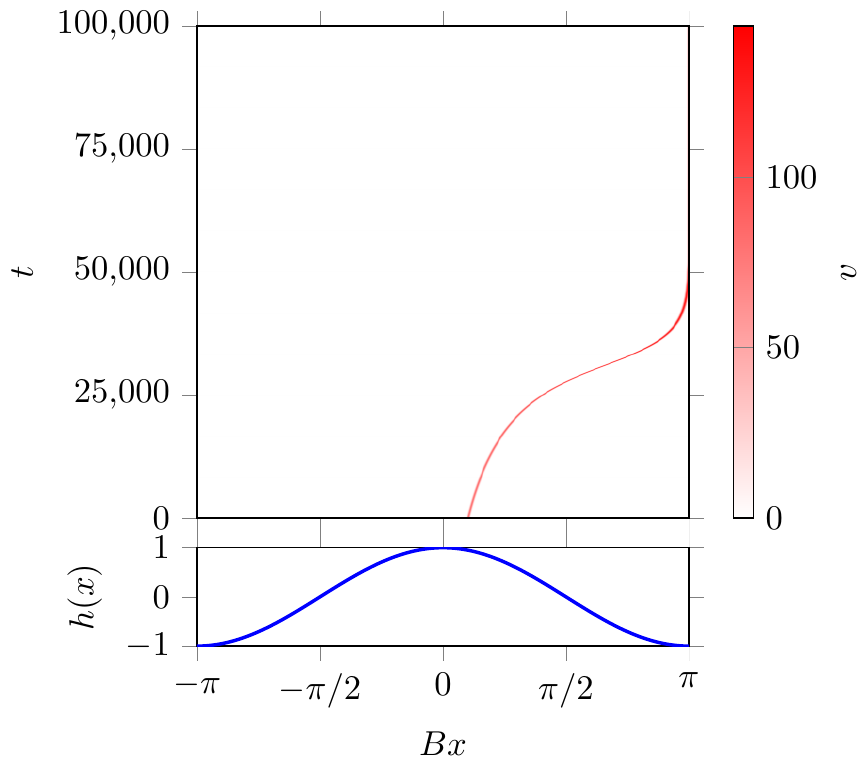}
			\caption{$A = 1$, $B = 2$}
		\end{subfigure}
	\caption{Numerical results for $h(x) = A \cos(Bx)$. Shown are the bifurcation diagram (solid for stable; dashed for unstable fixed points) for $A = 1$ (a), the bifurcation value $B_c(A)$ of the pitchfork bifurcation at $x = 0$ (b), and (parts of) various simulations of the full PDE illustrating the change of stability along with a plot of the function $h(x)$ (c-f). The green areas in (b) indicate the parameter region in which the fixed point $P_* = 0$ is stable. In the PDE simulations we have used parameters $a = 0.4$, $m = 0.45$, $D = 0.002$ and taken $x \in [-30,30]$.  }
	\label{fig:numericalCos}
\end{figure}

Moreover, these observations are validated by numerical simulation of the full PDE -- see Figure~\ref{fig:numericalExponential}(d-g) for (i), Figure~\ref{fig:numericalSech}(d-g) for (ii) and Figure~\ref{fig:numericalCos}(c-f) for (iii). Here, we observe the change in stability of the fixed points and, for well-chosen parameter values, these simulations show convergence to fixed points not located at the point of symmetry. Note also that in the case of periodic topography (i.e. case (iii)), there indeed is a region of $B$-values for which both a pulse at the top of a hill and one at the bottom of a valley can be stable (for the same $B$ value). Thus, we are led to conclude that a pitchfork bifurcation occurs at the critical values $B_c(0)$. Simulations indicate that these exist also when the asymptotic limit $\mu \ll 1$ does not hold.

For the last function, (iv), it is possible to derive the pulse location ODE~\eqref{eq:pulseLocationOde} explicitly, since~\eqref{eq:DAE} can be solved explicitly -- see Corollary~\ref{cor:h_example}. Using the expressions given in Corollary~\ref{cor:h_example}, a straightforward computation reduces~\eqref{eq:pulseLocationOde} to
\begin{equation}
	\frac{dP}{dt} = \frac{\tau}{6} \left[ \left( \cosh(\beta P) \mathcal{I}_1(P)\right)^2 - \left(\cosh(\beta P) \mathcal{I}_2(P)\right)^2\right],\label{eq:pulseLocationOdeMCB}
\end{equation}
where
\begin{equation}
	\mathcal{I}_1(P) := \int_P^\infty e^{r(P-z)} \sech(\beta z)\ dz; \qquad
	\mathcal{I}_2(P) := \int_{-\infty}^P e^{-r(P-z)} \sech(\beta z)\ dz.
\end{equation}
Thus, a point $P_*$ is a fixed point if and only if $\mathcal{I}_1(P_*) = \mathcal{I}_2(P_*)$. Straightforward inspection reveals that $P_* = 0$ therefore is the unique fixed point in case (iv) for all values of $\beta > 0$. By Proposition~\ref{theorem:stabilityInOde} and equation~\eqref{eq:eigenvalueFixedPointOdeSymmetric} the corresponding (small) eigenvalue $\underline{\lambda}$ can be approximated by
\begin{equation}
	\underline{\lambda} = \frac{2\tau}{3} \mathcal{I}_1(0) \left( r \mathcal{I}_1(0) - 1\right).
\end{equation}
Upon noting that
\begin{equation}
	r \mathcal{I}_1(0) - 1 = -\beta \int_0^\infty \sech(\beta z) \tanh(\beta z) e^{-rz}\ dz < 0,
\end{equation}
it is clear that $\underline{\lambda} < 0$. Hence, $P_* = 0$ is the only fixed point of~\eqref{eq:pulseLocationOdeMCB} in case (iv), which is (globally) stable -- for all $\beta > 0$. Direct PDE simulations verify this -- even when the asymptotic limit $\mu \ll 1$ does not hold -- see Figure~\ref{fig:numericalMCB}.

\begin{figure}[t]
	\centering
		\begin{subfigure}[t]{0.235\textwidth}
			\centering
			\includegraphics[width=\textwidth]{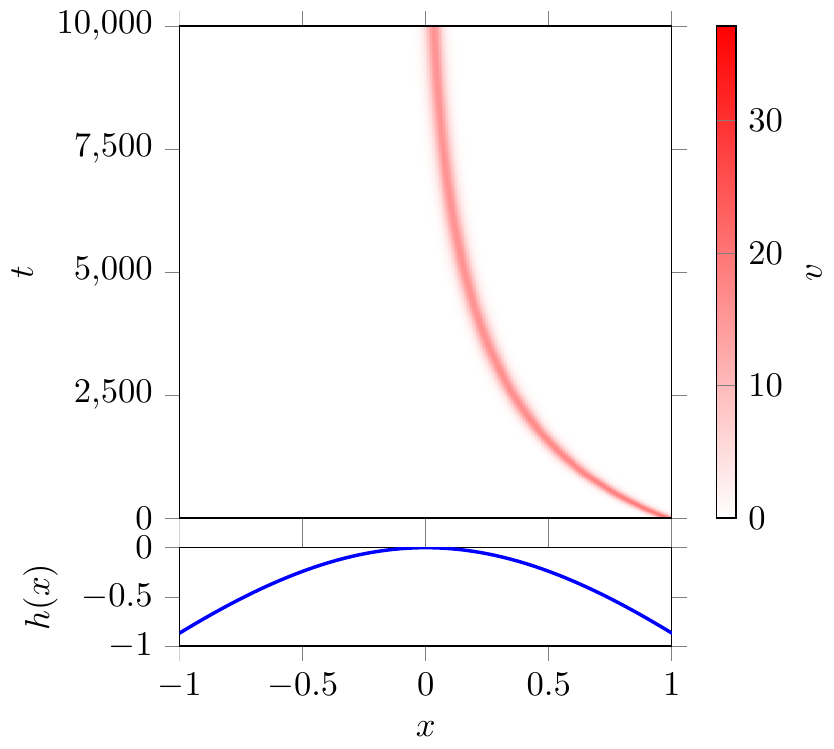}
		\end{subfigure}
	\caption{Direct numerical PDE simulation for $h(x) =  -2 \ln(\cosh(\beta x)$ for $\beta = 1$ along with a plot of the function $h(x)$. In the PDE simulation we have used the parameters $a = 0.5, m = 0.45, D = 0.01$ and taken $x \in [-30,30]$.}
	\label{fig:numericalMCB}
\end{figure}

\subsection{Stationary multi-pulse solutions}\label{sec:multiPulses}

The focus in this article has been on single pulse solutions to~\eqref{eq:klausmeier_model}. As a short encore we briefly discuss the possibility of stationary multi-pulse solutions -- i.e. solutions with multiple fast excursions. The movement of these solutions can be captured in an ODE much akin to~\ref{eq:pulseLocationOde}. Specifically, let $P_1,\ldots,P_N$ denote the location of $N$ pulses. Then their movement is described by the ODE
\begin{equation}\label{eq:NPulseLocationODE}
	\frac{dP_j}{dt} = \frac{\tau}{6} \left[ \tilde{u}_x(P_j^+)^2 - \tilde{u}_x(P_j^-)^2 \right], \qquad (j = 1,\ldots,N)
\end{equation}
where $\tilde{u}$ satisfies the differential-algebraic system
\begin{equation}\label{eq:DAENPulse}
	\left\{
	\begin{array}{rcll}
	\tilde{u}_{xx} + f(x) \tilde{u}_x + g(x) \tilde{u} + 1 - \tilde{u} &=&  0\\
	\tilde{u}(P_j) &=& \mu u_{0j} & (j=1,\ldots,N)\\
	\tilde{u}_x(P_j^+) - \tilde{u_x}(P_j^-) &=& \frac{6}{u_{0j}}&(j=1,\ldots,N)
	\end{array}\right.
\end{equation}
The derivation is similar to that of Proposition~\ref{theorem:stabilityInOde}; we omit the details here and refer the interested reader to~\cite{BD18} for a full coverage.

In case of constant coefficients $f,g\equiv 0$, it is well-known that stationary multi-pulse solutions do not exist~\cite{dek1siam,BD18}. In fact, from~\eqref{eq:NPulseLocationODE} one can verify that in $2$-pulse solutions the pulses typically move away from each other with a speed proportional to  $e^{-\Delta P}$, where $\Delta P := P_2 - P_1$ is the distance between the pulses -- see~\cite{dek1siam,BD18}.

However, the non-autonomous terms $f$ and $g$ affect the movement speed and can cancel this repulsive movement. Therefore stationary pulse solutions do exist in~\eqref{eq:klausmeier_model} for well-chosen $f$ and $g$. In Figure~\ref{fig:numericalMultiPulses} we show several numerical examples of (stable) stationary multi-pulse solutions for various choices of $f$ and $g$.

\begin{remark}\label{remark:findingFixedPointsNPulses}
	The spatially varying $f$ and $g$ have a order $\mathcal{O}(f,g)$ effect on the movement speed of the pulses. Finding fixed points of~\eqref{eq:NPulseLocationODE} -- i.e. finding stationary multi-pulse solutions to~\eqref{eq:klausmeier_model} -- thus boils down to balancing two effects of different size. In particular, if $f,g = \mathcal{O}(\delta)$, only multi-pulse solutions exist with $\Delta P = \mathcal{O}\left( - \ln(\delta)\right) \gg 1$. In this case, existence of stationary multi-pulse solutions can be established rigorously by asymptotic analysis and the methods of geometric singular perturbation theory.
\end{remark}

\begin{remark}
We do not present a full analysis of the spectrum of (evolving) multi-pulse solutions here; they can be stable and unstable depending on the parameter values -- similar to the one-pulse variants. A description of how to find the spectrum of multi-pulse solutions can be found in~\cite{BD18}.
\end{remark}

\begin{figure}
	\centering
		\begin{subfigure}[t]{0.235\textwidth}
			\centering
				\includegraphics[width = \textwidth]{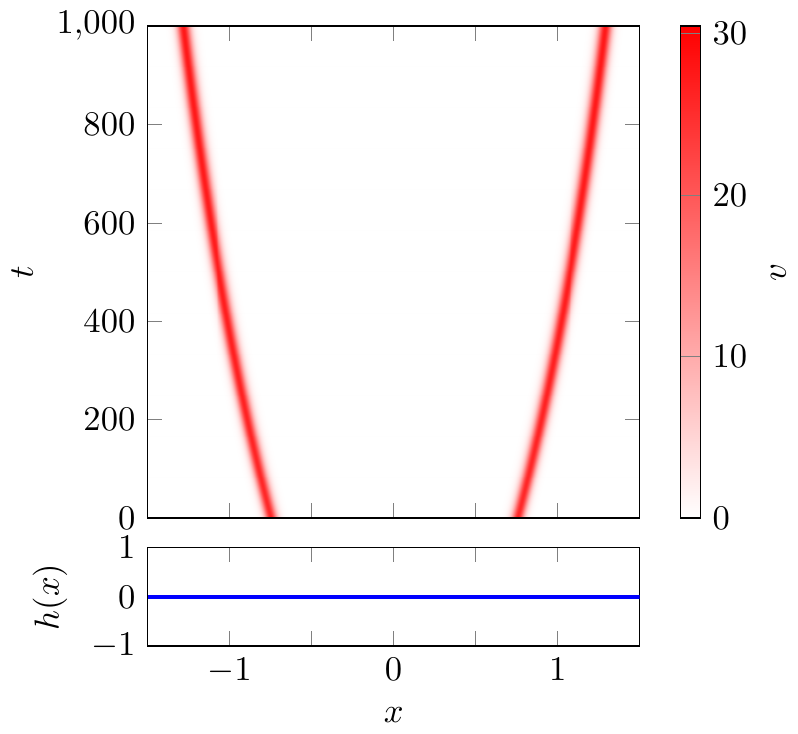}
			\caption{$h(x) = 0$}
		\end{subfigure}
		\begin{subfigure}[t]{0.235\textwidth}
			\centering
				\includegraphics[width = \textwidth]{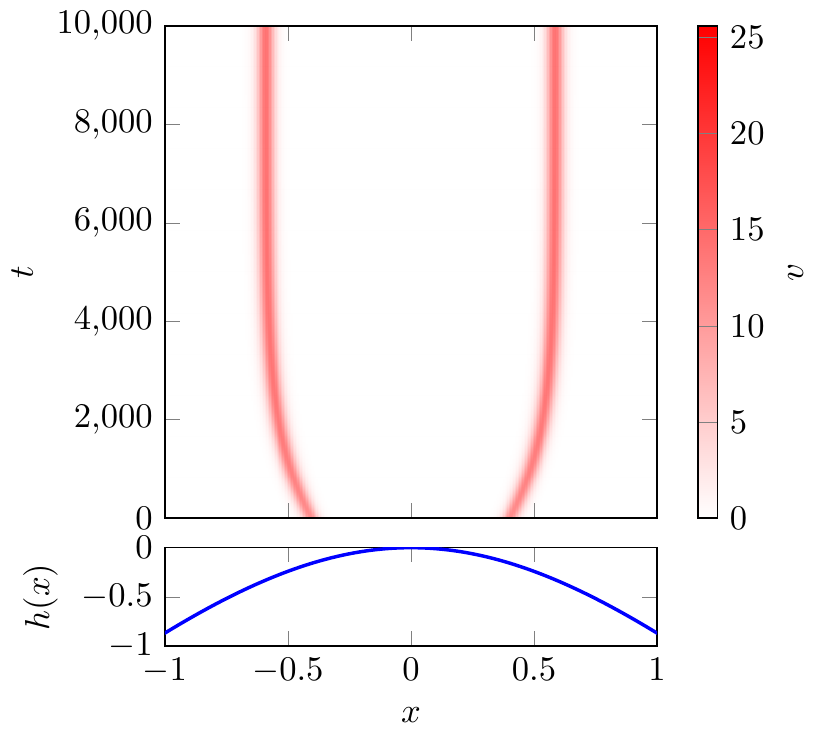}
			\caption{$h(x) = -2 \ln \cosh(x)$}
		\end{subfigure}
		\begin{subfigure}[t]{0.235\textwidth}
			\centering
				\includegraphics[width = \textwidth]{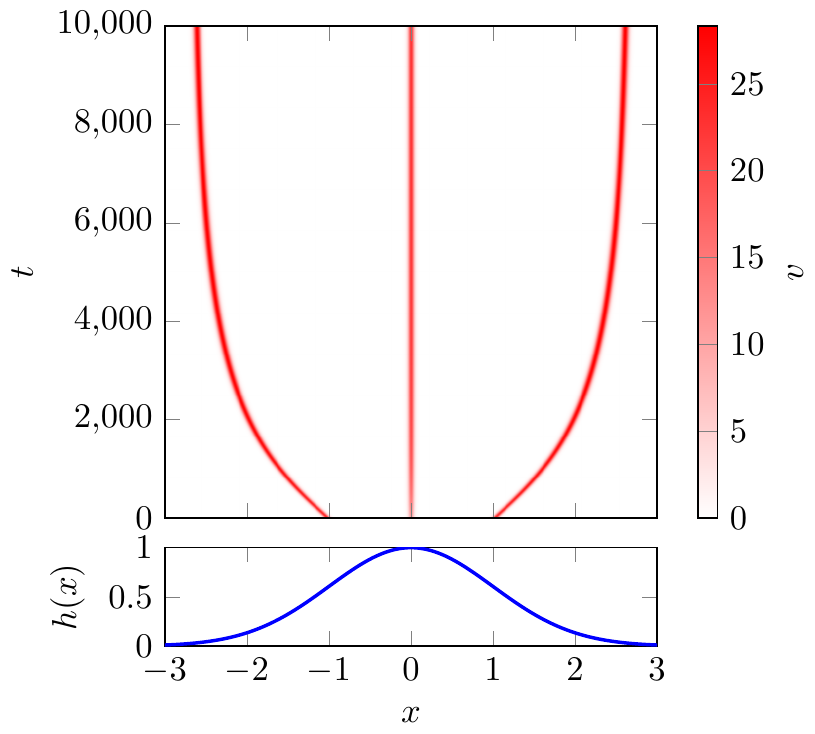}
			\caption{$h(x) = e^{-x^2/2}$}
		\end{subfigure}
		\begin{subfigure}[t]{0.235\textwidth}
			\centering
				\includegraphics[width = \textwidth]{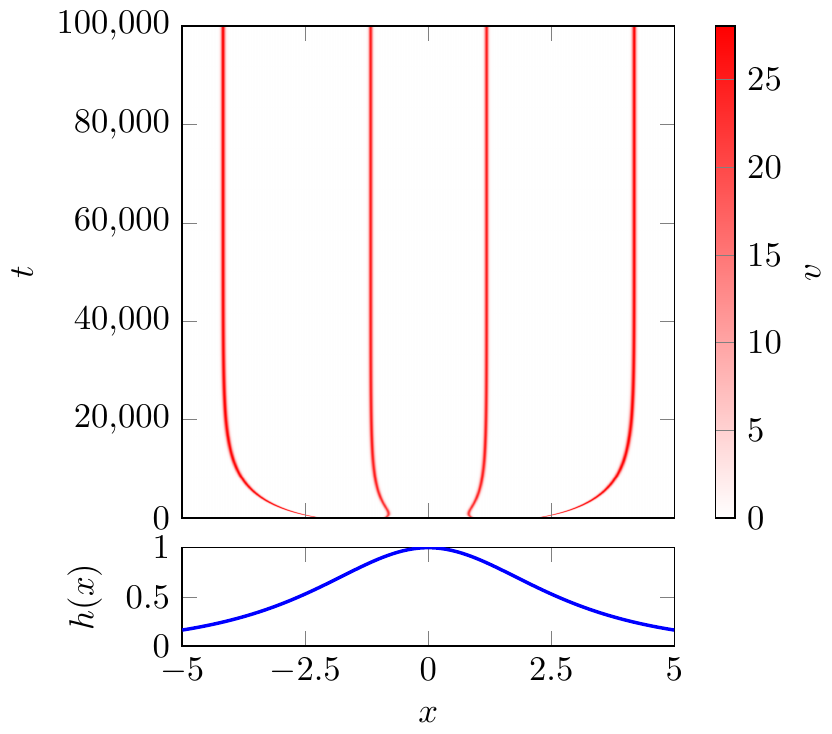}
			\caption{$h(x) = \sech(x/2)$}
		\end{subfigure}
	\caption{Numerical simulation of several multi-pulse solutions to~\eqref{eq:klausmeier_model} for various $h$, with $f = h'$ and $g = h''$. (a) $h(x) = 0$: no stable stationary multi-pulse solution is found; (b) $h(x) = -2 \ln \cosh(x)$: the existence of a stable two-pulse solution; (c) $h(x) = e^{-x^2/2}$: a stable three-pulse solution; (d) $h(x)=\sech(x/2)$: a stable four-pulse solution. In blue the form of the terrain is plotted. Note that only part of $x$-domain is shown for clarity. Also note that, using~\eqref{eq:twoPulsePosition}, it is found that $P_* \approx 0.51$ in (b).}
	\label{fig:numericalMultiPulses}
\end{figure}

For generic functions $f$ and $g$ it is, at the moment, not possible to prove existence of stationary multi-pulse solutions (however, see Remark~\ref{remark:findingFixedPointsNPulses} for the case of small $f$, $g$). We do remark however that stationary multi-pulse solutions can be constructed for $f$ and $g$ such that~\eqref{eq:DAENPulse} can be solved explicitly, as illustrated by the following proposition.

\begin{proposition}\label{theorem:twoPulseSolution}
	Let $h(x) = -2 \ln \cosh(\beta x)$, $\beta > 0$, $f = h'$, $g = h''$ and let $\mu \ll 1$. Then there exists a $P_* > 0$ such that~\eqref{eq:klausmeier_model} admits a stationary symmetric two-pulse solutions with pulses at $P_1 = -P_*$ and $P_2 = P_*$.
\end{proposition}
%\begin{proof}[Formal derivation]
\par\textit{Formal derivation.}
	By symmetry of the desired two-pulse solution, we may set $P_2 = P$, $P_1 = - P$. Moreover, necessarily $\tilde{u}'(0) = 0$. Since $\mu \ll 1$, to leading order we have $\tilde{u}(P) = \tilde{u}(-P) = 0$. Therefore $\tilde{u}$ is given to leading order by
\begin{equation}
	\tilde{u}(x) =
		\begin{cases}
			\hat{u}_b(x) - \frac{\hat{u}_b(-P)}{\hat{u}_-(-P)} \hat{u}_-(x), & x <- P,\\
			\hat{u}_b(x) - \frac{\hat{u}_b(P)}{\hat{u}_+(P)+\hat{u}_-(P)}\left(\hat{u}_+(x)+\hat{u}_-(x)\right), & -P < x < P;\\
			\hat{u}_b(x) - \frac{\hat{u}_b(P)}{\hat{u}_+(P)} \hat{u}_+(x), & x > P;
		\end{cases}
\end{equation}
where $\tilde{u}_\pm$ and $\tilde{u}_b$ are as in Corollary~\ref{cor:h_example}. To have stationary pulse solutions, by~\eqref{eq:NPulseLocationODE} we need to have
\begin{equation}\label{eq:twoPulsePosition}
	\mathcal{T}(P) := \hat{u}_b'(P) - \hat{u}_b(P) \left[ \frac{\sqrt{1+\beta^2}}{2} \left( \tanh(\sqrt{1+\beta^2}P)-1\right) + \beta \tanh(\beta P)\right] = 0,
\end{equation}
Upon noting that
\begin{equation}
	\mathcal{T}(0) = \frac{1}{2} \int_0^\infty e^{-\sqrt{1+\beta^2}z}\sech(\beta z)\ dz > 0,
\end{equation}
and, since $\lim_{P \rightarrow \infty} \hat{u}_b(P) = 1$ and $\lim_{P \rightarrow \infty} \hat{u}_b'(P) = 0$,
\begin{equation}
	\lim_{P \rightarrow \infty} \mathcal{T}(P) = - \beta < 0,
\end{equation}
continuity of $\mathcal{T}$ guarantees the existence of $P_* > 0$ as claimed.
%\end{proof}

\begin{remark}
	This result can be established rigorously by geometric singular perturbation theory, using the methods detailed in section~\ref{sec:existence}. We refrain from giving the details of this procedure.
\end{remark}

\section{Discussion}\label{sec:discussion}

In this paper, we studied pulse solutions in a reaction-advection-diffusion system with spatially varying coefficients. The existence of stationary pulse solutions at a point of symmetry was established by combining the usual techniques from geometric singular perturbation theory with the tools from the theory of exponential dichotomies. The latter has been used to generate a saddle-like structure in the slow subsystem, and to obtain bounds on the stable/unstable manifolds of this subsystem. These techniques have also been used to determine the spectral stability of these pulse solutions. None of these concepts or ideas are model-dependent and therefore could be used in a wider variety of models, including Gierer-Meinhardt type models.

Analysis of the spectrum associated to these pulse solutions showed that `large' eigenvalues can be bounded to the stable half-plane, under conditions similar to the usual, constant coefficient case. Although we did not focus on the dynamics of solutions when a large eigenvalue crosses the imaginary axis, simulations show the usual pulse annihilation and pulse splitting phenomena. However, the introduction of spatially varying coefficients does have a significant effect on the so-called `small' eigenvalues (close to $\lambda = 0$) because of the break-down of the translation invariance in the system. Therefore, well-chosen $f$ and $g$ can either stabilize or destabilize solutions. When the small eigenvalue is in the unstable half-plane, the pulse solution is unstable and as an effect its \emph{position} changes. In some cases, this in turn can subsequently lead to a pulse annihilation or a pulse splitting~\cite{BD18}. We expect that a careful tuning of $f$ and $g$ can either prevent or force these subsequent bifurcations, which may have a relevance in the maintenance of vegetation patterns in semi-arid climates.

%Analysis of the spectrum associated to these pulse solutions showed that `large' eigenvalues can be bounded to the stable half-plane $\{\lambda \in \mathbb{C}: \mbox{Re} \lambda < 0 \}$, under conditions similar to the usual, constant coefficient case. Although we did not focus on the dynamics of solutions when a large eigenvalue crosses the imaginary axis, simulations show the usual pulse annihilation and pulse splitting phenomena. However, the introduction of spatially varying coefficients does have a significant effect on the so-called `small' eigenvalues (close to $\lambda = 0$) because of the break-down of the translation invariance in the system. Therefore, well-chosen $f$ and $g$ can either stabilize or destabilize solutions. When the small eigenvalue is in the unstable half-plane $\{\lambda \in \mathbb{C}: \mbox{Re} \lambda > 0 \}$, the pulse solution is unstable and as an effect its \emph{position} changes. In some cases, this in turn can subsequently lead to a pulse annihilation or a pulse splitting~\cite{BD18}. We expect that a careful tuning of $f$ and $g$ can either prevent or force these subsequent bifurcations, which may have a relevance in the maintenance of vegetation patterns in semi-arid climates.

The small eigenvalues were studied more in-depth in the case of $f = h'$, $g = h''$ (where $h$ is used to model the topography of a dryland ecosystem). Here, we were able to link the stability of (stationary) pulse solution to the curvature of $h$. If the curvature is weak, the pulse is stable if $h''(0) < 0$ and unstable if $h''(0) > 0$; for strong curvature the opposite is true: the pulse is stable if $h''(0) > 0$ and unstable if $h''(0) < 0$. We found that this change in stability typically happens via a pitchfork bifurcation, and showed that the associated parameter combinations can be obtained numerically. However, we did not consider a fully general class of functions $f$ and $g$, and we do not know in which way these results generalize to other functions $f$ and $g$ -- although for choices $f$ and $g$ for which~\eqref{eq:klausmeier_model} does not posses the symmetry $(x,u) \rightarrow (-x,u)$ (i.e. when assumption (A2) does not hold), the pitchfork bifurcation will break down. A precise treatment of such generic functions could be the topic of subsequent work.

Moreover, in case of spatially varying coefficients, the system~\eqref{eq:klausmeier_model} can also posses stationary multi-pulse solutions -- i.e. solutions that have multiple fast excursions. When $f, g \equiv 0$, these solutions do not exist. Because the spatially varying coefficients break the translation invariance of the system, these multi-pulse solutions can exist -- for well-chosen functions $f$ and $g$. In this article we gave numerical evidence for this and showed their existence for a specific choice of functions. We do not think their existence can be proven in as much generality as the existence of stationary one pulse solutions -- certainly, the bounds used in this paper, provided by the theory of exponential dichotomies, are not sufficient in the regions between pulses. For sufficiently small $f$ and $g$, an asymptotic analysis can be developed to overcome this issue, although the distance between subsequent pulses then becomes asymptotically large and asymptotic analysis needs to be done with great care to keep track of the right scalings; this is topic of ongoing research.

Finally, the extended Klausmeier model studied in this paper has its application in ecology, where it is used to model dryland ecosystems. The studied pulse solutions in this model correspond to vegetation `patches' that are typically found in those ecosystems. Naturally, the results in this paper can therefore be used for this application. Specifically, the treatment of a spatially varying height function $h$ is new and is inherently more realistic than taking a constant topography (or a constantly sloped topography) as has been done in the past (see e.g.~\cite{siteur2014beyond, Bastiaansens2018, klausmeier1999, modfiedKlausmeier, modfiedKlausmeier}). Typically, the constant coefficient models exhibit pulses that only move uphill. However, as illustrated with numerics, we have shown that a varying topography can lead to both uphill \emph{and} downhill movement of pulses. This aligns better with measurements, where also both uphill and downhill movement can be observed -- even within the same general region~\cite{dunkerley2014, Bastiaansens2018}. In this regard, the study in this paper can be seen as a first step to better understand the role of topographic variability in pattern formation.

\section*{Acknowledgements}
We like to thank Marco Wolters for his exploratory (bachelor) research on the migration of vegetation pulses on periodic topographies. This work was funded by NWO's Mathematics of Planet Earth program.

\bibliographystyle{plain}
\bibliography{klausmeierModelVaryingTerrain}

\end{document}